\documentclass[final]{siamart1116}

\usepackage{setspace,url}
\usepackage{graphics,graphicx,epstopdf}
\usepackage{tabularx}
\usepackage{longtable}
\usepackage{booktabs,multirow,array,multicol}
\usepackage{dsfont}
\usepackage{enumitem}
\usepackage{subfig}
\usepackage[leqno]{amsmath}
\usepackage{amsxtra,amsfonts,amscd,amssymb,bm,epsf}
\usepackage{mathtools}
\usepackage{comment}
\usepackage[algo2e,linesnumbered,vlined,ruled]{algorithm2e}
\usepackage{amsopn}
\usepackage{xcolor}
\usepackage[normalem]{ulem}
\usepackage{times}

\ifpdf
  \DeclareGraphicsExtensions{.eps,.pdf,.png,.jpg}
\else
  \DeclareGraphicsExtensions{.eps}
\fi

\numberwithin{theorem}{section}
\numberwithin{equation}{section}

\newtheorem{remark}[theorem]{Remark}
\newtheorem{example}[theorem]{Example}
\newtheorem{assumption}[theorem]{Assumption}
\newtheorem{fact}[theorem]{Fact}

\newcommand{\TheTitle}{A stochastic semismooth Newton method for nonsmooth nonconvex optimization} 
\newcommand{\TheAuthors}{A. Milzarek, X. Xiao, S. Cen, Z. Wen, and M. Ulbrich}

\headers{A stochastic semismooth Newton method for nonsmooth optimization}{\TheAuthors}

\title{{\TheTitle}%\thanks{Submitted to the editors DATE.}
}

\author{
Andre Milzarek\thanks{Beijing International Center for Mathematical Research, BICMR, Peking University, Beijing, China
(\email{andremilzarek@bicmr.pku.edu.cn}). Research supported by the Boya Postdoctoral Fellowship.} \and
Xiantao Xiao\thanks{School of Mathematical Sciences, Dalian University of Technology, Dalian, China (\email{xtxiao@dlut.edu.cn}).} \and
Shicong Cen\thanks{School of Mathematical Sciences, Peking University, Beijing, China (\email{tsen9731@pku.edu.cn}).} \and
Zaiwen Wen\thanks{Beijing International Center for Mathematical Research, BICMR, Peking University, Beijing, China
(\email{wenzw@pku.edu.cn}). Research supported in part by the NSFC grant 11421101 and by the National Basic Research Project under the grant 2015CB856002.} \and
Michael Ulbrich\thanks{Chair of Mathematical Optimization, Department of Mathematics, Technical University Munich, Garching b. M\"unchen, Germany (\email{mulbrich@ma.tum.de}).}
}

\newcommand{\rmn}[1]{\textup{\textrm{#1}}}

\definecolor{tumb}{RGB}{0,101,189}

\newcommand{\R}{\mathbb{R}}
\newcommand{\Rn}{\mathbb{R}^n}

\newcommand{\Rex}{(-\infty,+\infty]}

\newcommand{\N}{\mathbb{N}}
\newcommand{\Sn}{\mathbb{S}^n}
\newcommand{\Spp}{\mathbb{S}^n_{++}}
\newcommand{\dom}{\mathrm{dom}}
\newcommand{\epi}{\mathrm{epi}}
\newcommand{\lev}{\mathrm{lev}}
\newcommand{\cl}{\mathrm{cl}}

\newcommand{\inter}{\mathrm{int}}

\newcommand{\gra}{\mathrm{gra}}

\newcommand{\vp}{\varphi}
\newcommand{\veps}{\varepsilon}
\newcommand{\ewmin}{\lambda_{\min}}
\newcommand{\ewmax}{\lambda_{\max}}
\newcommand{\lamm}{\lambda_{m}}
\newcommand{\lamM}{\lambda_{M}}
\newcommand{\unl}{\underline{\lambda}}
\newcommand{\ovl}{\overline{\lambda}}

\newcommand{\cF}{\mathcal{F}}  
  
\newcommand{\cG}{\mathcal{G}}  
  
\newcommand{\cH}{\mathcal{H}} 
\newcommand{\cM}{\mathcal{M}}
 
\newcommand{\cS}{{\mathcal{S}}}
\newcommand{\cO}{{\mathcal{O}}}

\newcommand{\proxt}[2]{{\mathrm{prox}}^{#1}_{#2}}

\newcommand{\FopL}{F^\Lambda}

\newcommand{\Exp}{\mathbb{E}} 
\newcommand{\Prob}{\mathbb{P}} 

\newcommand{\gfsub}{G_{s}} 
\newcommand{\gfsubk}{G_{{s}^k}}
\newcommand{\gfsubkk}{G_{{s}^{k+1}}}

\newcommand{\errgk}{\mathcal E^{\sf g}_{k}}

\newcommand{\psub}{p^\Lambda_{s}}
\newcommand{\psubk}{p^{\Lambda_k}_{{s}^k}}
\newcommand{\usub}{u^\Lambda_{s}}
\newcommand{\usubk}{u^{\Lambda_k}_{{s}^k}}

\newcommand{\hfsub}{H_{t}}
\newcommand{\hfsubk}{H_{t^k}}

\newcommand{\errhk}{\mathcal E^{\sf h}_{k}}

\newcommand{\Fsub}{{F}^{\Lambda}_{s}}
\newcommand{\Fsubk}{F^{\Lambda_k}_{{s}^k}}
\newcommand{\Fsubkk}{F^{\Lambda_{k+1}}_{{s}^{k+1}}}
\newcommand{\dsetM}{\cM^{\Lambda}_{{s},t}}
\newcommand{\dsetMk}{\cM^{\Lambda_k}_{{s}^k,t^k}}

\newcommand{\ngk}{{\bf n}^{\sf g}_k}
\newcommand{\nhk}{{\bf n}^{\sf h}_k}

\newcommand{\nstep}{z^k_{\sf n}}
\newcommand{\pstep}{z^k_{\sf p}}

\newcommand{\minimize}{\mathop{\textrm{min}}}

\newcommand{\half}{\frac{1}{2}}
\newcommand{\iprod}[2]{\langle #1, #2 \rangle}

\newcommand{\sfo}{\cS\cF\cO}
\newcommand{\sso}{\cS\cS\cO}

\DeclareMathOperator*{\argmin}{arg\,min}

\DeclareMathOperator*{\diag}{diag}

\newcommand{\be}{\begin{equation}}
\newcommand{\ee}{\end{equation}}
\newcommand{\bee}{\begin{equation*}}
\newcommand{\eee}{\end{equation*}}

\newcommand{\bea}{\begin{eqnarray}}
\newcommand{\eea}{\end{eqnarray}}
\newcommand{\beaa}{\begin{eqnarray*}}
\newcommand{\eeaa}{\end{eqnarray*}}

\newcommand{\tr}{\mathrm{tr}}

\ifpdf
\hypersetup{
  pdftitle={A stochastic semismooth Newton method for nonsmooth nonconvex optimization},
  pdfauthor={\TheAuthors}
}
\fi

\begin{document}

\maketitle

%----------------------------------------------------------------------------------------------------------------%
% ABSTRACT
%----------------------------------------------------------------------------------------------------------------%

\begin{abstract}
In this work, we present a globalized stochastic semismooth Newton method for solving stochastic optimization problems involving smooth nonconvex and nonsmooth convex terms in the objective function. We assume that only noisy gradient and Hessian information of the smooth part of the objective function is available via calling stochastic first and second order oracles. The proposed method can be seen as a hybrid approach combining stochastic semismooth Newton steps and stochastic proximal gradient steps. Two inexact growth conditions are incorporated to monitor the convergence and the acceptance of the semismooth Newton steps and it is shown that the algorithm converges globally to stationary points in expectation. Moreover, under standard assumptions and utilizing random matrix concentration inequalities, we prove that the proposed approach locally turns into a pure stochastic semismooth Newton method and converges r-superlinearly with high probability. We present numerical results and comparisons on $\ell_1$-regularized logistic regression and nonconvex binary classification that demonstrate the efficiency of our  algorithm.  \end{abstract} 

%----------------------------------------------------------------------------------------------------------------%
% KEYWORDS
%----------------------------------------------------------------------------------------------------------------%

\begin{keywords}
nonsmooth stochastic optimization, stochastic approximation, semismooth Newton method, sto\-chas\-tic second order information, global and local convergence.
\end{keywords}

%----------------------------------------------------------------------------------------------------------------%
% AMS
%----------------------------------------------------------------------------------------------------------------%

\begin{AMS}
49M15, 65C60, 65K05,  90C06
\end{AMS}

%----------------------------------------------------------------------------------------------------------------%
% SECTIONS
%----------------------------------------------------------------------------------------------------------------%

%----------------------------------------------------------------------------------------------------------------%
% INTRODUCTION
%----------------------------------------------------------------------------------------------------------------%

\section{Introduction}\label{sec:intro}
In this paper, we propose and analyze a stochastic semismooth Newton framework for solving general nonsmooth, nonconvex optimization problems of the form %the following  nonsmooth stochastic optimizaton problem 
%\be\label{eq:prob}
%\minimize\limits_{x\in\R^n}\quad \psi(x):= f(x)+r(x),\quad
%f(x)=\frac{1}{N} \sum_{i=1}^N f_i(x),
%\ee
%
\be\label{eq:prob}
\minimize\limits_{x\in\R^n}~\psi(x):= f(x)+r(x),%\quad
%f(x)=\Exp[F(x,\xi)],
\ee
where 
%where $\xi:\Omega\rightarrow\Xi$ is a random variable defined on a probability space $(\Omega,\mathcal{F},\Prob)$ which induces an expectation $\Exp$ of $\xi$,  $F : \R^n\times\Xi \to \R$ is a random operator and 
$f: \Rn \to \R$ is a (twice) continuously differentiable but possibly nonconvex function
%w.r.t. $x$, 
and $r : \R^n \to \Rex$ is a convex, lower semicontinuous, and proper mapping. Although the function $f$ is smooth, we assume that a full evaluation of $f$ and an exact computation of the gradient and Hessian values $\nabla f(x)$ and $\nabla^2 f(x)$ is either not completely possible or too expensive in practice. Instead, we suppose that only noisy gradient and Hessian information is available which can be accessed via calls to \textit{stochastic first} ($\cS\cF\cO$) and \textit{second order oracles} ($\cS\cS\cO$). 
%\revise{Problems of this form are also known as \textit{convex composite problems} and have been studied intensively throughout the last decade. In this paper, we assume that the evaluation of $f$ and the computation of $f(x)$ is either not possible or too expensive and gradient and Hessian information can only be accessed via \textit{stochastic first} ($\cS\cF\cO$) and \textit{second order oracles} ($\cS\cS\cO$). 
%Motivated by deterministic Newton-type approaches \cite{MilUlb14,Mil16,XLWZ2016}, our proposed method combines stochastic semismooth Newton steps, stochastic proximal gradient steps, and a globalization strategy that is based on controlling the acceptance of the Newton steps via growth conditions. In this way, the resulting stochastic algorithm can be guaranteed to converge globally in expectation and transition to fast local r-linear or r-superlinear convergence can be established with high probability. 
Composite problems of the type \cref{eq:prob} arise frequently in statistics and in large-scale statistical learning, see, e.g., \cite{HTF2009,MaiBacPonSap09,BacJenMaiObo11,ShaBen14,BCN2016}, and in many other applications. In these examples and problems, the smooth mapping $f$ is typically of the form
\be \label{eq:prob-ex} f(x) := \Exp[F(x,\xi)] = \int_{\Omega} F(x,\xi(\omega))\,{\rm d}\Prob(\omega), \quad \text{or} \quad f(x) := \frac{1}{N} \sum_{i=1}^N f_i(x), \ee
where $\xi : \Omega \to W$ is a random variable defined on a given probability space $(\Omega,\mathcal F,\Prob)$, $W$ is a measurable space, and $F : \Rn \times W \to \R$ and the component functions $f_i : \Rn \to \R$, $i = 1,...,N$, correspond to certain loss models. More specifically, in the latter case, when the nonsmooth term $r \equiv 0$ vanishes, the problem \cref{eq:prob} reduces to the so-called and well-studied \textit{empirical risk minimization problem}
\be \label{eq:finite-sum} \min_{x\in\R^n}~ f(x),\quad f(x) := \frac{1}{N} \sum_{i=1}^N f_i(x). \ee
Since the distribution $\Prob$ in \cref{eq:prob-ex} might not be fully known and the number of components $N$ in \cref{eq:finite-sum} %(i.e., the size of the underlying dataset) 
can be extremely large, stochastic approximation techniques, such as the mentioned stochastic oracles, have become an increasingly important tool in the design of efficient and computationally tractable numerical algorithms for the problems \cref{eq:prob} and \cref{eq:finite-sum}, \cite{NemJudLanSha08,GL2013,ShaBen14,GL2016,GhaLanZha16,WMGL2017,WZ2017}. Moreover, in various interesting problems such as deep learning, dictionary learning, training of neural networks, and classification tasks with nonconvex activation functions, \cite{MasBaxBarFre99,MaiBacPonSap09,BacJenMaiObo11,DenYu14,LeCBenHin15,Sch15,GooBenCou16}, the loss function $f$ is nonconvex, which represents another major challenge for stochastic optimization approaches. For further applications and additional connections to simulation-based optimization, we refer to \cite{GL2013,GhaLanZha16}. %Recently, there has been extensive interest in stochastic algorithms for \cref{eq:prob} and \cref{eq:finite-sum} in the nonconvex setting.}
%
%A special case of problem \cref{eq:prob} that often arises in large-scale machine learning and statistic problems (see, e.g., \cite{HTF2009,BCN2016}) is in the form of
%. 
%Due to modern massive datasets, the number of components $N$ is very large.
% which makes the classical deterministic optimization algorithms impractical. In contrast, stochastic methods are observed to be very efficient in this situation because of the low per iteration cost. Recent years lots of fast randomized algorithms are developed for solving this problem. The readers are referred to \cite{BCN2016} for a detailed survey.
%When  the mapping $r$ is smooth or vanished, the problem \cref{eq:prob} is reduced to a smooth stochastic optimization problem,
%\be\label{eq:finite-sum}
%\minimize\limits_{x\in\R^n}\quad f(x)=\Exp[F(x,\xi)],
%\ee
%which is relatively convenient to deal with. 
%
%\revise{  %By using a proximal approach, a few stochastic first-order algorithms can be easily revised to handle nonsmooth problem. 
%However, it is unknown that how to extend stochastic second-order methods to solve nonsmooth problem \cref{eq:prob} or extend the deterministic nonsmooth Newton-type methods \cite{MilUlb14,XLWZ2016} to the stochastic setting. To the best of our knowledge,  nonsmooth stochastic second-order methods have not been studied in the literature. }

%----------------------------------------------------------------------------------------------------------------%
% (1.1) RELATED WORK
%----------------------------------------------------------------------------------------------------------------%
 
\subsection{Related Work} 

The pioneering idea of utilizing stochastic approximations and the development of the associated, classical stochastic gradient descent method (SGD) for problem \cref{eq:finite-sum} and other stochastic programs can be traced back to the seminal work of Robbins and Monro \cite{RM1951}. Since then, a plethora of stochastic optimization methods, strategies, and extensions has been studied and proposed for different problem formulations and under different basic assumptions. In the following, we give a brief overview of related research directions and related work. %In particular, focus on the finite-sum problem \cref{eq:finite-sum} with different assumptions on the (strong) convexity of $f$ or $f_i$'s.  or to more general functions $f$ as in \cref{eq:prob-ex} or using more general stochastic oracles. , %stochastic Newton and quasi-Newton approaches for \cref{eq:finite-sum}, and stochastic first order algorithms for the general, nonsmooth and nonconvex problem \cref{eq:prob}.}

\textit{First order methods}. Advances in the research of stochastic first order methods for the smooth empirical risk problem \cref{eq:finite-sum} are numerous and we will only name a few recent directions here. Lately, based on the popularity and flexible applicability of the basic SGD method, a strong focus has been on the development and analysis of more sophisticated stochastic first order oracles to reduce the variance induced by gradient sampling and to improve the overall performance of the underlying SGD method. Examples of algorithms that utilize such variance reduction techniques include SVRG \cite{JZ2013}, SDCA \cite{SSZ2013}, SAG \cite{SLB2017}, and SAGA \cite{DBLJ2014}. Moreover, Friedlander and Schmidt \cite{FS2012} analyze the convergence of a mini-batch stochastic gradient method for strongly convex $f$, in which the sampling rates are increased progressively. Incorporating different acceleration strategies, the first order algorithms Catalyst \cite{LMH2015} and Katyusha \cite{AZ2017} further improve the iteration complexity of the (proximal) SGD method.   
%In \cite{WZ2017}, Wang and Zhang  establish the convergence results of a sub-sampled SVRG algorithm.

Several of the mentioned algorithms can also be extended to the nonsmooth setting $r \not\equiv 0$ by using the proximity operator of $r$ and associated stochastic proximal gradient steps, see, e.g., the perturbed proximal gradient method \cite{AtcForMou17} studied by Atchad\'{e} et al$\text{.}$, prox-SVRG \cite{XZ2014}, prox-SAGA \cite{DBLJ2014}, and prox-SDCA \cite{SSZ2016}.
%Stochastic first-order methods are convenient to be developed to solve nonsmooth stochastic optimization problem \cref{eq:prob} when the proximal mapping of $r$ is easy to compute. These methods include stochastic proximal gradient method \cite{AtcForMou17}, prox-SVRG \cite{XZ2014}, prox-SAGA \cite{DBLJ2014} and prox-SDCA \cite{SSZ2016}. % \cite{SST2011} ?
AdaGrad \cite{DHS2011} is another extension of the classical SGD method that utilizes special adaptive step size strategies. %and is able to solve the finite-sum problem \cref{eq:finite-sum} with an additional nonsmooth regularization. 
Under the assumption that $r$ is block separable, Richt\'arik and Tak\'a\v c \cite{RT2014} develop a randomized block-coordinate descent method for \cref{eq:prob}. An accelerated variant of this approach is investigated by Lin et al$\text{.}$ \cite{LLX2015}. %In \cite{ZX2015}, Zhang and Xiao investigate a stochastic primal-dual coordinated method for strongly convex problems, which alternates between (randomized) maximization and minimization. %$$ maximizing over the randomly chosen dual variable and minimizing over the primal variable.

The methods discussed so far either require convexity of $f$ or of each of the component functions $f_i$ or even stronger assumptions. Ghadimi and Lan \cite{GL2013,GL2016} generalize the basic and accelerated SGD method to solve nonconvex and smooth minimization problems. Allen-Zhu and Hazan \cite{AZH2016} and Reddi et al$\text{.}$ \cite{RHSPS2016} analyze stochastic variance reduction techniques for the nonconvex version of problem \cref{eq:finite-sum}. Moreover, Reddi et al$\text{.}$ \cite{RSPS2016} and Allen-Zhu \cite{Allen2017} further extend existing stochastic first order methods to find approximate stationary points of the general nonconvex, nonsmooth model \cref{eq:prob}. In \cite{GhaLanZha16}, Ghadimi et al$\text{.}$  discuss complexity and convergence results for a mini-batch stochastic projected gradient algorithm for problem \cref{eq:prob}. Xu and Yin \cite{XuYin15} present and analyze a block stochastic gradient method for convex, nonconvex, and nonsmooth variants of the problem \cref{eq:prob}.

\textit{Quasi-Newton and second order methods}. %Emerging studies of stochastic second-order methods for \cref{eq:finite-sum} are recently witnessed. 
Recently, in order to accelerate and robustify the convergence of first order algorithms, stochastic second order methods have gained much attention. So far, the majority of stochastic second order methods is designed for the smooth problem \cref{eq:finite-sum} and is based on variants of the sub-sampled Newton method in which approximations of the gradient and Hessian of $f$ are generated by selecting only a sub-sample or mini-batch of the components $\nabla f_i$ and $\nabla^2 f_i$, $i = 1,..., N$. In \cite{BCNN2011,BCNW2012}, assuming positive definiteness of the sub-sampled Hessians, the authors analyze the convergence of a sub-sampled Newton-CG method and discuss strategies for selecting the sample sets. Erdogdu and Montanari \cite{EM2015} derive convergence rates of a projected sub-sampled Newton method with rank thresholding. In \cite{RKM2016I,RKM2016II}, Roosta-Khorasani and Mahoney establish non-asymptotic, probabilistic global and local convergence rates for sub-sampled Newton methods by applying matrix concentration inequalities. Xu et al$\text{.}$ \cite{XYRRM2016} present convergence and complexity results for a sub-sampled Newton-type approach with non-uniform sampling. Bollapragada et al$\text{.}$ \cite{BBN2016} consider a sub-sampled Newton method for problems with the more general loss function given in \cref{eq:prob-ex} and derive r-superlinear convergence rates in expectation using a ``bounded moment'' condition to overcome the nonequivalence of norms in infinite dimensions. In \cite{WZ2017}, Wang and Zhang propose an algorithm that combines the advantages of variance reduction techniques and sub-sampled Newton methods. Convergence properties are studied under the assumption that $f$ is strongly convex and the Hessians $\nabla^2 f_i$ are Lipschitz continuous (with a uniform constant). Based on the existence of a suitable square-root decomposition of the Hessian, Pilanci and Wainwright \cite{PW2017} propose a Newton sketch method for general convex,  smooth programs. In \cite{BBN2017}, the numerical performance of the Newton sketch method and different sub-sampled Newton approaches is compared. Furthermore, based on unbiased estimators of the inverse Hessian, a stochastic method called LiSSA is studied in \cite{ABH2016}. A recent discussion of different stochastic second order algorithms can also be found in \cite{YeLuoZha17}. 

Stochastic quasi-Newton methods represent another large and important class of stochastic numerical algorithms for problem \cref{eq:finite-sum}. Typically, these methods combine specific sub-sampling schemes for $\nabla f$ with randomized BFGS or BFGS-type updates to approximate the Hessian $\nabla^2 f$. %and to accelerate the SGD method. 
In \cite{SchYuGue07}, the authors propose a basic sto\-chas\-tic quasi-Newton algorithm for quadratic loss functions. Bordes et al$\text{.}$ \cite{BBG2009} present a quasi-Newton approach that is based on diagonal curvature estimation. Mokhtari and Ribeiro \cite{MR2014} investigate a regularized stochastic BFGS method for solving strongly convex problems. In \cite{BHNS2016}, Byrd et al$\text{.}$ consider a stochastic limited-memory BFGS (L-BFGS) algorithm that in\-cor\-po\-rates exact Hessian information of the functions $f_i$ to build the BFGS-type updates. The stochastic L-BFGS method discussed in \cite{MNJ2016} uses variance reduction techniques to improve its convergence and performance. Moreover, Gower et al$\text{.}$ \cite{GGR2016} establish linear convergence of a stochastic block L-BFGS method if the functions $f_i$ are strongly convex. 

 %Lei et al. \cite{LJCJ2017} develop a class of stochastically controlled stochastic gradient methods for nonconvex finite-sum optimization.
In contrast, the number of stochastic second order algorithms for smooth but nonconvex problems seems to be still quite limited. Based on a damping strategy for BFGS-type updates introduced by Powell \cite{Pow78} and using general stochastic first order oracles, Wang et al$\text{.}$ \cite{WMGL2017} propose a stochastic L-BFGS method for smooth, nonconvex problems. Under the assumption that the full gradient of the objective function is available, Xu et al$\text{.}$ \cite{XRKM2017} derive worst-case optimal iteration complexity results for an adaptive cubic regularization method with inexact or sub-sampled Hessian information. Generalizations and further aspects of this approach have been considered very recently in \cite{XuRooMah17,YaoXuRooMah18}. %We note that although the algorithms mentioned in the last paragraphs and their associated convergence results rely on different (global) convexity assumptions, it might be possible to weaken and restrict these conditions to local regions. %of interest. 

Finally, in \cite{WZ2017}, Wang and Zhang mention an extension of their hybrid method to the nonsmooth setting. A similar and related idea has also been presented in \cite{ShiLiu15}. In particular, these approaches can be interpreted as stochastic variants of the proximal Newton method \cite{LeeSunSau14} for the general problem \cref{eq:prob}. Nevertheless, strong and uniform convexity assumptions are still required to guarantee convergence and well-definedness of the inner steps and subproblems.
%Agarwal et al. \cite{AZBHM2017} design a second-order algorithm for smooth nonconvex problem which is guaranteed to return an approximate local minimum in linear time.   

Let us note that there is also a vast literature on incremental methods for \cref{eq:prob} and \cref{eq:finite-sum}, see, e.g., \cite{Bert2015,GOP2015}, which is beyond the scope of this paper.

%----------------------------------------------------------------------------------------------------------------%
% (1.2) CONTENTS AND CONTRIBUTIONS
%----------------------------------------------------------------------------------------------------------------%

\subsection{Contents and Contributions}
In this paper, we develop a stochastic second order framework for the general optimization problem \cref{eq:prob}. Our basic idea is to apply a semismooth Newton method to approximately solve the % is based on a reformulation of the associated first order optimality conditions of problem \cref{eq:prob} as a 
nonsmooth fixed point-type equation  
\be \label{eq:opt-intro} F^\Lambda(x) := x - \proxt{\Lambda}{r}(x - \Lambda^{-1} \nabla f(x)) = 0, \quad \Lambda \in \Spp, \ee
which represents a reformulation of the associated first order optimality conditions of problem \cref{eq:prob}. Specifically, we will consider stochastic variants of the nonsmooth residual \cref{eq:opt-intro} and of the semismooth Newton method, in which the gradient and Hessian of $f$ are substituted by stochastic oracles. 
%obtain the semismooth Newton steps by carefully choosing the sub-sampled gradients and Hessians of $f$ and the associated generalized derivatives.
Motivated by deterministic Newton-type approaches \cite{MilUlb14,Mil16,XLWZ2016}, our proposed method combines stochastic semismooth Newton steps, stochastic proximal gradient steps, and a globalization strategy that is based on controlling the acceptance of the Newton steps via growth conditions. In this way, the resulting stochastic algorithm can be guaranteed to converge globally in expectation and almost surely, i.e., for a generated sequence of iterates $(x^k)_k$, we have 
\[ \Exp[\|F^\Lambda(x^k)\|^2] \to 0 \quad \text{and} \quad F^\Lambda(x^k) \to 0 \;\; \text{almost surely}, \quad k \to \infty. \] %thus, every accumulation point $x^*$ of $(x^k)_k$ is a stationary point of \cref{eq:prob} satisfying $F^\Lambda(x^*) = 0$ with probability 1. 
Furthermore, inspired by \cite{RKM2016I,RKM2016II} and using matrix concentration inequalities \cite{Tro12}, we prove that transition to fast local r-linear or r-superlinear convergence can be established with high probability if the sample sizes and sampling rates are chosen appropriately and increase sufficiently fast. To the best of our knowledge, rigorous extensions of existing stochastic sec\-ond order methods to the nonsmooth,  nonconvex setting considered in this work do not seem to be available.  %Our algorithm combines the fast local convergence of the sub-sampled semismooth Newton method with the globalization property of standard proximal gradient method.   %The semismooth Newton step is accepted if the sub-sampled residual and the forward-backward envelope are sufficiently decreased. Otherwise, a standard proximal gradient step is performed to ensure  global convergence. Broadly speaking our method can be viewed as a stochastic extension of the ideas in \cite{MilUlb14,XLWZ2016}. 
%However, the theoretical analysis of our method is more complicated due to the combination of nonsmoothness, nonconvexity and randomness.
%Under standard assumptions in the existing literature, our method is shown to converge globally in expectation to a stationary point of problem \cref{eq:prob}. %Furthermore, we show that  the sub-sampled semismooth Newton steps achieve fast local convergence with high probability under suitable local assumptions. %In addition to an expected semismooth assumption, the method is proved to be local superlinearly convergent in expectation. 
We now briefly summarize some of the main challenges and contributions. \vspace{.5ex}
\begin{itemize}
\item We provide a unified convergence theory for the proposed stochastic Newton method covering global and local aspects and transition to fast local convergence. In contrast to many other works, convexity of the smooth function $f$ or of the objective function $\psi$ is not required in our analysis. %Hence, our method and theory is applicable for both convex and nonconvex problems.
%\item A sub-sampled semismooth Newton method is proposed and analyzed for solving  large-scale nonsmooth composite minimization problem \cref{eq:prob}. There are few existing studies on stochastic second-order methods for nonsmooth problem, even in the convex setting.
\item In order to ensure global convergence and based on an acceptance test, the algorithm is allowed to switch between Newton and proximal gradient steps. Hence, a priori, it is not clear whether the generated iterates correspond to measurable random variables or to a stochastic process. This structural mechanism is significantly different from other existing stochastic approaches and will be discussed in detail in \cref{sec:global}. %Moreover, due to violated independency, standard martingale theory  %s,  which might also be helpful for the analysis of other stochastic approaches. %These results guarantee that the proposed method converges globally with a fast local rate by a suitable sampling strategy.
%\item %Our method is proved to globally converge to a stationary point. 
%When the iterate is sufficiently close to a stationary point, the sub-sampled semismooth Newton step is shown to converge superlinearly if the sample sizes of sub-sampled Hessian and gradient increase fast enough.	
%\item Implementation?
\item Our algorithmic approach and theoretical results are applicable for general stochastic oracles. Consequently, a large variety of approximation schemes, such as basic sub-sampling strategies or more elaborate variance reduction techniques, \cite{JZ2013,XZ2014,RSPS2016,WMGL2017}, can be used within our framework. In particular, in our numerical experiments, we investigate a variance reduced version of our method. Similar to \cite{WMGL2017}, the numerical results indicate that the combination of second order information and variance reduction techniques is also very effective in the nonsmooth setting. We note that the proposed method (using different stochastic oracles) performs quite well in comparison with other state-of-the-art algorithms in general. 
\end{itemize}
 
%----------------------------------------------------------------------------------------------------------------%
% (1.3) ORGANIZATION
%----------------------------------------------------------------------------------------------------------------%
 
\subsection{Organization}
This  paper is organized as follows. Our specific stochastic setup, a derivation of the equation \cref{eq:opt-intro}, and the main algorithm are stated in \cref{sec:algo}. The global and local convergence results are presented in \cref{sec:global} and \cref{sec:local}, respectively. Finally, in \cref{sec:numerical}, we report and discuss our numerical comparisons and experiments in detail. 

\subsection{Notation} For any $n \in \N$, we set $[n] := \{1,...,n\}$. By $\iprod{\cdot}{\cdot}$ and $\|\cdot\| := \|\cdot\|_2$ we denote the standard Euclidean inner product and norm. The set of symmetric and positive definite $n \times n$ matrices is denoted by $\Spp$. For a given matrix $\Lambda \in \Spp$, we define the inner product $\iprod{x}{y}_\Lambda := \iprod{x}{\Lambda y} = \iprod{\Lambda x}{y}$ and $\|x\|_\Lambda := \sqrt{\iprod{x}{x}_\Lambda}$. The set $\lev_\alpha\,f := \{x : f(x) \leq \alpha\}$ denotes the lower level set of a function $f$ at level $\alpha \in \R$. For a given set $S \subset \Rn$, the set $\cl~S$ denotes the closure of $S$ and $\mathds 1_S : \Rn \to \{0,1\}$ is the associated characteristic function of $S$. For $p \in (0,\infty)$ the space $\ell^p_+$ consists of all sequences $(x_n)_{n\geq 0}$ satisfying $x_n \geq 0$, $n \geq 0$, and $\sum x_n^p < \infty$. Let $(\Omega,\mathcal F,\Prob)$ be a given probability space. The space $L^p(\Omega) := L^p(\Omega,\Prob)$, $p \in [1,\infty]$, denotes the standard $L^p$ space on $\Omega$. We write $\xi \in \mathcal F$ for ``$\xi$ is $\mathcal F$-measurable''. Moreover, we use $\sigma(\xi^1,...,\xi^k)$ to denote the $\sigma$-algebra generated by the family of random variables $\xi^1,...,\xi^k$. For a random variable $\xi \in L^1(\Omega)$ and a sub-$\sigma$-algebra $\cH \subseteq \mathcal F$, the conditional expectation of $\xi$ given $\cH$ is denoted by $\Exp[\xi \,|\, \cH]$. The conditional probability of $S \in \mathcal F$ given $\cH$ is defined as $\Prob(S \mid \cH) := \Exp[{\mathds 1}_S \mid \cH]$. We use the abbreviations ``$\text{a.e.}$'' and ``$\text{a.s.}$'' for ``almost everywhere'' and ``almost surely'', respectively. 

%----------------------------------------------------------------------------------------------------------------%
% A STOCHASTIC SEMISMOOTH NEWTON METHOD
%----------------------------------------------------------------------------------------------------------------%

\section{A Stochastic Semismooth Newton Method} \label{sec:algo}

%----------------------------------------------------------------------------------------------------------------%
% (2.1) PRELIMINARIES
%----------------------------------------------------------------------------------------------------------------%

\subsection{Probabilistic Setting and Preliminaries} \label{sec:prelim}

In this section, we introduce several basic definitions and preparatory results. %In particular, we verify that the first order optimality conditions for problem \cref{eq:prob} can be represented as a nonsmooth equation. 
We start with an overview of the stochastic setting and the sampling strategy.

%----------------------------------------------------------------------------------------------------------------%
% (2.1.1) STOCHASTIC SETUP
%----------------------------------------------------------------------------------------------------------------%

\subsubsection{Stochastic Setup} \label{sec:setup} 
Although the function $f$ is smooth, we assume that an exact or full evaluation of the gradient $\nabla f$ and Hessian $\nabla^2 f$ is not possible or is simply too expensive. Hence, we will work with \textit{stochastic first} $(\cS\cF\cO)$ and \textit{second order oracles} $(\cS\cS\cO)$,
\[ \cG : \Rn \times \Xi \to \Rn, \quad \cH : \Rn \times \Xi \to \Sn \]
to approximate gradient and Hessian information. Specifically, given an underlying probability space $(\Omega,\cF,\Prob)$ and a measurable space $(\Xi,\mathcal X)$, we generate two mini-batches of random samples
\[ {s}^k:=\{{s}^k_1,\ldots,{s}^k_{\ngk}\} \quad \mbox{and}\quad {t}^k := \{{t}^k_1,\ldots,{t}^k_{\nhk}\} \]
and calculate the stochastic approximations $\cG(x,s^k_i) \approx \nabla f(x)$ and $\cH(x,t^k_j) \approx \nabla^2 f(x)$ in each iteration. Here, we assume that the space $\Omega$ is sufficiently rich allowing us to model and describe the (possibly independent) sample batches $s^k$, $t^k$ and other associated stochastic processes in a unified way. Moreover, each of the samples $s^k_i, t^k_j : \Omega \to \Xi$, $i \in [\ngk]$, $j \in [\nhk]$, corresponds to an $(\cF,\mathcal X)$-measurable, random mapping and $\ngk$ and $\nhk$ denote the chosen sample rates or sample sizes of $s^k$ and $t^k$, respectively. %For each $k$ and all $i = 1,...,\ngk$, $j =1,..., \nhk$, we assume that $s_i^k$ and $t_j^k$ are independent and identically distributed (iid) samples of the random variable $\xi : \Omega \to \Xi$. Furthermore, we assume that the sample collections $s^k$ and $t^k$ are chosen independently of each other and of the other sample batches $s^\ell$ and $t^\ell$, $\ell \in \N_0 \backslash \{k\}$. 
Similar to \cite{DanLan15,XuYin15,GL2016,GhaLanZha16,WMGL2017}, we then construct a mini-batch-type, stochastic gradient $\gfsubk(x)$ and Hessian $\hfsubk(x)$ as follows
\be \label{eq:def-subgh} \gfsubk(x):=\frac{1}{\ngk}\sum_{i = 1}^{\ngk}\cG(x,s^k_i), \quad \quad \hfsubk(x):=\frac{1}{\nhk}\sum_{i = 1}^{\nhk}\cH(x,t^k_i). \ee
%
%Under some standard integrability assumptions and if $F$ is a Carath\'{e}odory function\footnote{A mapping $F: \Rn \times \Xi \to \R$ is called \textit{Carath\'{e}odory function} if $F(\cdot,z) : \Rn \to \R$ is continuous for all $z \in \Xi$ and if $F(x,\cdot) : \Xi \to \R$ is measurable for all $x \in \Rn$.} and $F(\cdot,z) : \Rn \to \R$ is twice continuously differentiable for almost every $z \in \Xi$, then the sub-sampled gradient $\gfsubk$ and the sub-sampled Hessian $\hfsubk$ can be shown to be \textit{unbiased estimators}, i.e., it holds
%
%$ \Exp{[\gfsubk(x)]} = \frac{1}{\ngk}\sum_{i=1}^{\ngk} \Exp{[\nabla_x F(x,s^k_i)]} = \Exp{[\nabla_x F(x,\xi)]} = \nabla f(x) $
% 
%and $\Exp{[\hfsubk(x)]} =  \nabla^2 f(x)$ for any $x \in \Rn$, 
Throughout this work, we assume that the stochastic oracles $\cG$ and $\cH$ are \textit{Carath\'{e}odory functions}\,\footnote{A mapping $F: \Rn \times \Xi \to \R$ is called \textit{Carath\'{e}odory function} if $F(\cdot,z) : \Rn \to \R$ is continuous for all $z \in \Xi$ and if $F(x,\cdot) : \Xi \to \R$ is measurable for all $x \in \Rn$.}. Further assumptions on the stochastic setting will be introduced later in \cref{sec:global-assumption} and \cref{sec:loc-condconc}. 
%We note that the different independency conditions are added for ease of exposition. In particular, our global convergence analysis, presented in \cref{sec:global}, does also hold under more general and relaxed assumptions.

We will also sometimes drop the index $k$ from the mini-batches $s^k$, $t^k$ and sample sizes $\ngk$, $\nhk$ when we consider a general pair of batches $s$ and $t$. 

%----------------------------------------------------------------------------------------------------------------%
% (2.1.2) DEFINITIONS AND FIRST ORDER OPTIMALITY
%----------------------------------------------------------------------------------------------------------------%

\subsubsection{Definitions and First Order Optimality}
In the following, we derive first order optimality conditions for the composite problem \cref{eq:prob}. Suppose that $x^* \in \dom~r$ is a local solution of problem \cref{eq:prob}. Then, $x^*$ satisfies the mixed-type variational inequality 
\be \label{eq:fst-opt-vip} \iprod{\nabla f(x^*)}{x-x^*} + r(x) - r(x^*) \geq 0, \quad \forall~x \in \Rn. \ee
By definition, the latter condition is equivalent to $- \nabla f(x^*) \in \partial r(x^*)$, where $\partial r$ denotes the convex subdifferential of $r$.  We now introduce the well-known \textit{proximal mapping} $\proxt{\Lambda}{r} : \Rn \to \Rn$ of $r$. For an arbitrary parameter matrix $\Lambda \in \Spp$, the proximity operator $\proxt{\Lambda}{r}(x)$ of $r$ at $x$ is defined as
\be \label{eq:def-prox}  \proxt{\Lambda}{r}(x) := \argmin_{y \in \Rn}~r(y) + \half \|x-y\|_\Lambda^2. \ee
The proximity operator is a $\Lambda$-firmly nonexpansive mapping, i.e., it satisfies 
\[  \|\proxt{\Lambda}{r}(x) - \proxt{\Lambda}{r}(y)\|^2_\Lambda \leq \iprod{\proxt{\Lambda}{r}(x) - \proxt{\Lambda}{r}(y)}{x - y}_\Lambda, \quad \forall~x,y \in \Rn. \]
Consequently, $\proxt{\Lambda}{r}$ is Lipschitz continuous with modulus 1 with respect to the norm $\|\cdot\|_\Lambda$. We refer to \cite{Mor65,ComWaj05,BacJenMaiObo11,BauCom11} for more details and (computational) properties. Let us further note that the proximity operator can also be uniquely characterized by the optimality conditions of the underlying optimization problem \cref{eq:def-prox}, i.e., 
\be \label{eq:prox:sub-prox} \proxt{\Lambda}{r}(x) \in x - \Lambda^{-1} \cdot \partial r(\proxt{\Lambda}{r}(x)). \ee
Using this characterization, condition \cref{eq:fst-opt-vip} can be equivalently rewritten as follows: 
\be \label{eq:non-eq} F^\Lambda(x^*) = 0, \quad \text{where} \quad F^\Lambda(x) := x - \proxt{\Lambda}{r}(x- \Lambda^{-1} \nabla f(x)). \ee
We call $x \in \Rn$ a \textit{stationary point}  of problem \cref{eq:prob}, if it is a solution of  the nonsmooth equation \cref{eq:non-eq}. If the problem is convex, e.g., if $f$ is a convex function, then every stationary point is automatically a local and global solution of \cref{eq:prob}. The fixed point-type equation \cref{eq:non-eq} forms the basis of the proximal gradient method, \cite{FukMin81-1,FukMin81,ComWaj05,ParBoy14}, which has been studied intensively during the last decades. %In the subsequent sections, we will also use the so-called \textit{Moreau envelope} $\envt{\Lambda}{r}$, 
%
%\[ \envt{\Lambda}{r} : \Rn \to \R, \quad \envt{\Lambda}{r}(x) := \min_{y \in \Rn}~r(y) + \half \|x-y\|_\Lambda^2. \]
%
%It is well-known that the Moreau envelope $\envt{\Lambda}{r}$ is a real-valued, continuously differentiable, and convex function. Its gradient  can be computed via $\nabla \envt{\Lambda}{r}(x) = \Lambda(x - \proxt{\Lambda}{r}(x))$, see \cite{Mor65,ComWaj05,Mil16}. 

For an arbitrary sample $s$, the corresponding stochastic residual is given by
\[ \Fsub(x) := x - \proxt{\Lambda}{r}(x- \Lambda^{-1} \gfsub(x)). \]
We will also use $\usub(x) := x - \Lambda^{-1} \gfsub(x)$ and $\psub(x) := \proxt{\Lambda}{r}(\usub(x))$ to denote the stochastic (proximal) gradient steps. 
 
%----------------------------------------------------------------------------------------------------------------%
% (2.2) ALGORITHMIC FRAMEWORK
%----------------------------------------------------------------------------------------------------------------%

\subsection{Algorithmic Framework} \label{subsec:algo} 
In this section, we describe our algorithmic approach in detail. The overall idea is to use a stochastic semismooth Newton method to calculate an approximate solution of the optimality system 
\[ F^\Lambda(x) = 0. \] 
The associated Newton step $d^k$ at iteration $k$ is then given by the linear system of equations
 \be \label{eq:newton-step} M_k d^k = - \Fsubk(x^k), \quad M_k \in \dsetMk(x^k). \ee
Here, we consider the following set of generalized derivatives
 \be \label{eq:gen-deriv} \dsetM(x) := \{M \in \R^{n \times n}: M = (I - D) + D\Lambda^{-1} \hfsub(x), \, \, D \in \partial \proxt{\Lambda}{r}(\usub(x))\}, \ee
where $\partial \proxt{\Lambda}{r}(\usub(x))$ denotes the Clarke subdifferential of $ \proxt{\Lambda}{r}$ at the point $\usub(x)$. The set $\dsetM(x)$ depends on the stochastic gradient and on the stochastic Hessian defined in \cref{eq:def-subgh}. Moreover, the samples $s^k$, $t^k$ and the matrix $\Lambda_k$ used in \cref{eq:newton-step} may change in each iteration, see also \cref{remark:ada-lambda}. We further note that, in practice, the system \cref{eq:newton-step} can be solved inexactly via iterative approaches such as the conjugate gradient or other Krylov subspace methods. 

In the deterministic setting, the set $\dsetM(x)$ reduces to $\mathcal M^\Lambda(x) := \{M = (I-D) + D\Lambda^{-1}\nabla^2 f(x), D \in \partial\proxt{\Lambda}{r}(u^\Lambda(x))\}$ with $u^\Lambda(x) = x - \Lambda^{-1}\nabla f(x)$. In general, $\mathcal M^\Lambda(x)$ does not coincide with Clarke's subdifferential $\partial F^\Lambda(x)$. As shown in \cite{Clarke1990}, we can only guarantee $\partial F^\Lambda(x)h \subseteq \text{co}(\mathcal M^\Lambda(x)h)$ for $h \in \Rn$. However, the set-valued mapping $\mathcal M^\Lambda : \Rn \rightrightarrows \R^{n \times n}$ defines a so-called \textit{(strong) linear Newton approximation} at $x$ if the proximity operator $\proxt{\Lambda}{r}$ is (strongly) semismooth at $u^\Lambda(x)$. In particular, $\mathcal M^\Lambda$ is upper semicontinuous and compact-valued. We refer to \cite[Chapter 7]{FP2003II} and \cite{PatSteBem14} for more details. We also note that the chain rule for semismooth functions implies that $F^\Lambda(x)$ is semismooth at $x$ with respect to $\mathcal M^\Lambda(x)$ if $\proxt{\Lambda}{r}$ is semismooth at $u^\Lambda(x)$. Furthermore, in various important examples including, e.g., $\ell_1$- or nuclear norm-regularized optimization, group sparse problems or semidefinite programming, the associated proximal mapping $\proxt{\Lambda}{r}$ can be shown to be (strongly) semismooth and there exist explicit and computationally tractable representations of the generalized derivatives $D \in \partial \proxt{\Lambda}{r}(\cdot)$, see \cite{PatSteBem14,Mil16,XLWZ2016} for a detailed discussion. 

%We note that, in the case $s \equiv t$, the set $\mathcal M^\Lambda_{s,s}(x)$ does not coincide with the subdifferential $\partial \Fsub(x)$ in general.
 %As shown in \cite{Clarke1990}, we can only guarantee $\partial \Fsub(x)h \subseteq \text{co}(\mathcal M^\Lambda_{s,s}(x)h)$ for $h \in \Rn$. However, for any fixed batch of samples $s$, the set-valued mapping $\mathcal M^\Lambda_{s,s} : \Rn \rightrightarrows \R^{n \times n}$ defines a \textit{(strong) linear Newton approximation} at $x$ if the proximity operator $\proxt{\Lambda}{r}$ is (strongly) semismooth at $\usub(x)$. Let us refer to \cite[Chapter 7]{FP2003II} for more details. We also note that the chain rule for semismooth functions implies that $\Fsub(x)$ is semismooth at $x$ with respect to $\mathcal M^\Lambda_{s,s}(x)$ if $\proxt{\Lambda}{r}$ is semismooth at $\usub(x)$. Furthermore, in various important examples and applications, including, e.g., $\ell_1$- or nuclear norm-regularized optimization, group sparse problems or semidefinite programming, the associated proximal mapping $\proxt{\Lambda}{r}$ can be shown to be (strongly) semismooth and there exist explicit and computationally tractable representations of the generalized derivatives $D \in \partial \proxt{\Lambda}{r}(\usub(x))$, see \cite{PatSteBem14,Mil16,XLWZ2016} and \cref{sec:numerical} for a detailed discussion. 
 
\begin{algorithm2e}[t]
\caption{A Stochastic Semismooth Newton Method}    
\label{alg:ssn}
\lnlset{alg:SSN1}{1}{Initialization: ~~Choose an initial point $x^0 \in \dom~r$, $\theta_0 \in \R_+$, and mini-batches $s^0, t^0$. Select sample sizes $(\ngk)_{k}$, $(\nhk)_{k}$, parameter matrices $(\Lambda_k)_k \subset \Spp$, and step sizes $(\alpha_k)_{k}$. Choose $\eta, p \in (0,1)$, $\beta > 0$, and $(\nu_k)_k$, $(\veps_k^1)_k$, $(\veps^2_k)_k$. Set iteration $k:=0$.} 
\\ \vspace{.5ex}
\While{did not converge}{
\lnlset{alg:SSN2}{3}{Compute $\Fsubk(x^k)$ and choose $M_k \in \dsetMk(x^k)$. For all $i = 1,...,{\bf n}^{\sf g}_{k+1}$ and $j =1,...,{\bf n}^{\sf h}_{k+1}$ select new samples $s^{k+1}_i, t^{k+1}_j$.} \\ \vspace{.5ex}
\lnlset{alg:SSN3}{4}{Compute the Newton step $d^k$ by solving
\begin{center} $M_kd^k = -\Fsubk(x^k)$. \end{center}} %\\
\lnlset{alg:SSN4}{5}{Set $\nstep = x^k + d^k$. If the conditions $\nstep \in \dom~r$, \cref{eq:growth-1}, and \cref{eq:growth-2} are satisfied, skip step \ref{alg:SSN5} and set $x^{k+1} = \nstep$, $\theta_{k+1} = \|\Fsubkk(\nstep)\|$. Otherwise go to step \ref{alg:SSN5}.} \\ \vspace{.5ex}
\lnlset{alg:SSN5}{6}{Set $v^k = - \Fsubk(x^k)$, $x^{k+1} = x^k + \alpha_k v^k $, and $\theta_{k+1} = \theta_k$.} \\ \vspace{.5ex}
\lnlset{alg:SSN6}{7}{Set $k\gets k+1$.} \\ }
\end{algorithm2e}
 
In order to control the acceptance of the Newton steps and to achieve global convergence of our algorithm, we introduce the following growth conditions for the trial step $\nstep = x^k + d^k$:
\begin{align} \label{eq:growth-1} \|\Fsubkk(\nstep)\| &\leq (\eta + \nu_k) \cdot \theta_k + \veps^1_k, \\ \label{eq:growth-2} \psi(\nstep) &\leq \psi(x^k) + \beta \cdot \theta_k^{1-p} \|\Fsubkk(\nstep)\|^p   + \veps_k^2. \end{align}
If the trial point $\nstep$ satisfies both conditions and is feasible, i.e., if $\nstep \in \dom~r$, we accept it and compute the new iterate $x^{k+1}$ via $x^{k+1} = \nstep$. The parameter sequences $(\nu_k)_k$, $(\veps^1_k)_k$, and $(\veps_k^2)_k$ are supposed to be nonnegative and summable and can be chosen during the initialization or during the iteration process. Furthermore, the parameter $\theta_k$ keeps track of the norm of the residual $F^{\Lambda_i}_{s^i}(x^i)$ of the last \textit{accepted} Newton iterate $x^i$, $i < k$, and is updated after a successful Newton step. The parameters $\beta > 0$, $\eta, p \in (0,1)$ are given constants. If the trial point $\nstep$ does not satisfy the conditions \cref{eq:growth-1} and \cref{eq:growth-2}, we reject it and perform an alternative proximal gradient step using the stochastic residual $\Fsubk(x^k)$ as an approximate descent direction. We also introduce a step size $\alpha_k$ to damp the proximal gradient step and to guarantee sufficient decrease in the objective function $\psi$. A precise bound for the step sizes $\alpha_k$ is derived in \cref{prop:prox-descent}. The details of the method are summarized in Algorithm \ref{alg:ssn}. 

Our method can be seen as a hybrid of the semismooth Newton method and the standard proximal gradient method generalizing the deterministic Newton approaches presented in \cite{MilUlb14,Mil16} to the stochastic setting. Our globalization technique is inspired by \cite{MilUlb14}, where a filter globalization strategy was proposed to control the acceptance of the Newton steps. Similar to \cite{MilUlb14,Mil16}, we add condition \cref{eq:growth-1} to monitor the behavior and convergence of the Newton steps. The second condition \cref{eq:growth-2} (together with the feasibility condition $\nstep \in \dom~r$) is required to bound the possible $\psi$-ascent of intermediate Newton steps. In contrast to smooth optimization problems, descent-based damping techniques or step size selections, as used in, e.g., \cite{BCNN2011,BBN2016,BHNS2016,RKM2016I,RKM2016II,WMGL2017}, can not always guarantee sufficient $\psi$-descent of the semismooth Newton steps due to the nonsmooth nature of problem \cref{eq:prob}. This complicates the analysis and globalization of semismooth Newton methods in general. In practice, the second growth condition \cref{eq:growth-2} can be restrictive since an evaluation of the full objective function is required. However, similar descent conditions also appeared in other globalization strategies for smooth problems, \cite{RKM2016I,XRKM2017,XuRooMah17,YaoXuRooMah18}. In the next section, we verify that Algorithm \ref{alg:ssn} using the proposed growth conditions \cref{eq:growth-1}--\cref{eq:growth-2} converges globally in expectation. Moreover, in \cref{theorem:conv-prox-strong}, we establish global convergence of Algorithm \ref{alg:ssn} without condition \cref{eq:growth-2} in a strongly convex setting. Under standard assumptions and if the sample sizes ${\bf n}^{\sf g}$ and ${\bf n}^{\sf h}$ are chosen sufficiently large, we can further show that the conditions \cref{eq:growth-1} and \cref{eq:growth-2} are satisfied locally in a neighborhood of a stationary point with high probability. This enables us to derive fast local convergence results in probability. Let us note that the growth conditions \cref{eq:growth-1}--\cref{eq:growth-2} are checked using a new sample mini-batch $s^{k+1}$. Thus, only one gradient evaluation is required per iteration if the Newton step is accepted. We note that the feasibility condition $\nstep \in \dom~r$ can be circumvented by setting $x^{k+1} = \mathcal P_{\dom~r}(\nstep)$, where $\mathcal P_{\dom~r}$ denotes the projection onto the set $\dom~r$. We refer to \cref{remark:proj-loc} for a related discussion. 

Let us mention that an alternative globalization is analyzed in \cite{PatSteBem14,SteThePat17} where the authors propose the so-called forward-backward envelope (FBE) as a smooth merit function for problem \cref{eq:prob}. Since this framework requires an additional proximal gradient step (and thus, an additional gradient evaluation) after each iteration, we do not consider this approach here.

%----------------------------------------------------------------------------------------------------------------%
% GLOBAL CONVERGENCE
%----------------------------------------------------------------------------------------------------------------%

\section{Global Convergence} \label{sec:global} 
In this section, we analyze the global convergence behavior of Algorithm \ref{alg:ssn}. We first present and summarize our main assumptions.

%----------------------------------------------------------------------------------------------------------------%
% (3.1) ASSUMPTIONS
%----------------------------------------------------------------------------------------------------------------%

\subsection{Assumptions} \label{sec:global-assumption}
Throughout this paper, we assume that $f : \Rn \to \R$ is continuously differentiable on $\Rn$ and $r : \Rn \to \Rex$ is convex, lower semicontinuous, and proper. As already mentioned, we also assume that the oracles $\cG, \cH : \Rn \times \Xi \to \R$ are Carath\'{e}odory functions.
%Moreover, we assume that $F : \Rn \times \Xi \to \R$ is a Carath\'{e}odory function and $F(\cdot,z) : \Rn \to \R$ is twice continuously differentiable for almost every $z \in \Xi$. 
In the following, we further specify the assumptions on the functions $f$ and $r$. 

\begin{assumption} \label{ass:lip} Let $f : \Rn \to \R$ be given as in \cref{eq:prob}. We assume: \vspace{.5ex}
\begin{itemize}
\item[{\rm(A.1)}] The gradient mapping $\nabla f$ is Lipschitz continuous on $\Rn$ with modulus $L > 0$. \vspace{.5ex} 
\item[{\rm(A.2)}] The objective function $\psi$ is bounded from below on $\dom~r$. \vspace{.5ex}
\item[{\rm(A.3)}] There exist parameters $\mu_f \in \R$ and $\mu_r \geq 0$ with $\bar\mu := \mu_f + \mu_r > 0$ such that the shifted functions $f - \frac{\mu_f}{2}\|\cdot\|^2$ and $r - \frac{\mu_r}{2}\|\cdot\|^2$ are convex. \vspace{.5ex}
\item[{\rm(A.4)}] There exists a constant $\bar g_r > 0$ such that for all $x \in \dom~r$ there exists $\lambda \in \partial r(x)$ with $\|\lambda\| \leq \bar g_r$.
\end{itemize}
\end{assumption}

Assumption (A.3) implies that the function $\psi$ is strongly convex with convexity parameter $\bar\mu$. Furthermore, if both assumption (A.1) and (A.3) are satisfied, then the parameter $\mu_f$ is bounded by the Lipschitz constant $L$, i.e., we have $|\mu_f| \leq L$. The assumptions (A.3)--(A.4) are only required for a variant of Algorithm \ref{alg:ssn} that uses a modified globalization strategy, see \cref{theorem:conv-prox-strong}. A concrete example for $r$ that satisfies (A.4) is given in \cref{remark:example-c3}. %Let us also note that certain integrability conditions on $F$ can be added to ensure differentiability of $f$ and assumption (A.1). Specifically, if $\Exp[F(x,\xi)] < \infty$, $\Exp[\|\nabla_x F(x,\xi)\|] < \infty$, and if there exists a  function $\theta \in L^1(\Omega)$ such that
%
%\[ \label{eq:exp-int} \|\nabla_x F(x,\xi(\omega))\| \leq \theta(\omega) \quad \forall~x \in \Rn, \]
%
%then $f$ is continuously differentiable and we have $\nabla f(x) = \Exp[\nabla_x F(x,\xi)]$. 
We continue with the assumptions on the parameters used within our algorithmic framework.

\begin{assumption}
Let $(\Lambda_k)_k \subset \Spp$ be a family of symmetric, positive definite parameter matrices and let $(\nu_k)_k$, $(\veps_k^1)_k$, and $(\veps_k^2)_k$ be given sequences. Then, for some given parameter $p \in (0,1)$ we assume: \vspace{.5ex}
\begin{itemize}
\item[{\rm(B.1)}] There exist $0 < \lambda_m \leq \lambda_M < \infty$ such that $\lambda_M I \succeq \Lambda_k \succeq \lambda_m I$ for all $k \in \N$. \vspace{0.5ex}
\item[{\rm(B.2)}] It holds $(\nu_k)_k$, $(\veps_k^2)_k \in \ell_+^1$, and $(\veps_k^1)_k \in \ell_+^p$. 
\end{itemize}
\end{assumption} 

In the following sections, we study the convergence properties of the stochastic process $(x^k)_k$ generated by Algorithm \ref{alg:ssn} with respect to the filtrations
\[ \mathcal F_k := \sigma(s^0, \ldots, s^k, t^0, \ldots, t^k), \quad \text{and} \quad \hat {\mathcal F}_k := \sigma(s^0, \ldots, s^k,s^{k+1}, t^0, \ldots, t^k). \]
The filtration $\mathcal F_k$ represents the information that is collected up to iteration $k$ and that is used to compute the trial point $\nstep$ or a proximal gradient step
\[ \pstep := x^k + \alpha_k v^k = x^k - \alpha_k \Fsubk(x^k). \] 
The filtration $\hat {\mathcal F}_k$ has a similar interpretation, but it also contains the information produced by deciding whether the Newton step $\nstep$ should be accepted or rejected, i.e., it holds $\hat {\mathcal F}_k = \sigma({\mathcal F}_k \cup \sigma(s^{k+1}))$. The filtrations $\{\mathcal F_k, \hat{\mathcal F}_k\}$ naturally describe the aggregation of information generated by Algorithm \ref{alg:ssn}. We will work with the following stochastic conditions. 

\begin{assumption}
We assume:
\begin{itemize} 
\item[{\rm(C.1)}] For all $k \in \N_0$, the generalized derivative $M_k$, chosen in step \ref{alg:SSN2} of Algorithm \ref{alg:ssn} , is an $\mathcal F_k$-measurable mapping, i.e., the function $M_k : \Omega \to \R^{n \times n}$ is an $\mathcal F_k$-measurable selection of the multifunction ${\mathcal M}_k : \Omega \rightrightarrows \R^{n \times n}$, ${\mathcal M}_k(\omega)  := \mathcal M^{\Lambda_k}_{s^k(\omega),t^k(\omega)}(x^k(\omega))$.
\item[{\rm(C.2)}] The variance of the individual stochastic gradients is bounded, i.e., for all $k \in \N$ there exists $\sigma_k \geq 0$ such that
\end{itemize}
\[ {\mathds E}[ \|\nabla f(x^k) - \gfsubk(x^k)\|^2] \leq \sigma_k^2.  \]
%
%\begin{align*} \mathcal M_k^\circ(\omega) & := \mathcal M^{\Lambda_k}_{s^k(\omega),t^k(\omega)}(x^k(\omega)) \\ &= \{M : M = I - D_k(\omega)(I - \Lambda_k^{-1} \nabla^2 f_{t^k(\omega)}(x^k(\omega))), D_k(\omega) \in \partial \proxt{\Lambda_k}{r}(u^{\Lambda_k}_{s^k(\omega)}(x^k(\omega))\}. \end{align*}
\end{assumption}

The second condition is common in stochastic programming, see, e.g., \cite{GL2013,XuYin15,BBN2016,BHNS2016,GhaLanZha16,WMGL2017}. Since the generalized derivative $M_k$ is generated iteratively and depends on the random process $(x^k)_k$ and on the mini-batches $(s^k)_k$, $(t^k)_k$, condition (C.1) is required to guarantee that the selected matrices $M_k$ actually define $\mathcal F_k$-measurable random operators. A similar assumption was also used in \cite{WMGL2017}. Furthermore, applying the techniques and theoretical results presented in \cite{Ulb02,Ulb11} for infinite-dimensional nonsmooth operator equations, we can ensure that the multifunction ${\mathcal M}_k$ admits at least one measurable selection $M_k : \Omega \to \R^{n \times n}$. We discuss this important observation together with a proof of \cref{fact:one} in \cref{sec:app-fact}. Let us note that it is also possible to generalize the assumptions and allow $\mathcal F_k$-measurable parameter matrices $\Lambda_k$. However, in order to simplify our analysis, we focus on a deterministic choice of $(\Lambda_k)_k$ and do not consider this extension here.  

As a consequence of condition (C.1) and of the assumptions on $f$, we can infer that the random processes $(\nstep)_k$, $(\pstep)_k$, and $(x^k)_k$ are adapted to the filtrations $\mathcal F_k$ and $\hat {\mathcal F}_k$. 
\begin{fact} \label{fact:one} Under assumption {\rm(C.1)}, it holds $\nstep, \pstep \in \mathcal F_k$ and $x^{k+1} \in \hat {\mathcal F}_k$ for all $k \in \N_0$. 
\end{fact} 
Since the choice of the iterate $x^{k+1}$ depends on the various criteria, the properties stated in \cref{fact:one} are not immediately obvious. In particular, we need to verify that the decision of accepting or rejecting the stochastic semismooth Newton step $\nstep$ is an $\hat{\mathcal F}_k$-measurable action. A proof of \cref{fact:one} is presented in \cref{sec:app-fact}.

%----------------------------------------------------------------------------------------------------------------%
% (3.2) PROPERTIES OF F-LAMBDA
%----------------------------------------------------------------------------------------------------------------%

\subsection{Properties of $F^\Lambda$} 
In this subsection, we discuss several useful properties of the nonsmooth function $\FopL$ and of its stochastic version $\Fsub$. The next statement shows that $\|\Fsub(x) \|$ does not grow too much when the parameter matrix $\Lambda$ changes. This result was first established by Tseng and Yun in \cite{TseYun09}. 

\begin{lemma} \label{lemma:F-bound} Let $\Lambda_1, \Lambda_2 \in \Spp$ be two arbitrary matrices. Then, for all $x \in \Rn$, for all samples $s$, and for $W := \Lambda_2^{-\half}\Lambda_1\Lambda_2^{-\half}$, it follows
\[ \|F^{\Lambda_1}_s(x)\| \leq \frac{1 + \ewmax(W) + \sqrt{1-2\ewmin(W)+\ewmax(W)^2}}{2} \frac{\ewmax(\Lambda_2)}{\ewmin(\Lambda_1)} \|F^{\Lambda_2}_s(x)\|. \] 
\end{lemma}
\begin{proof} The proof is identical to the proof of \cite[Lemma 3]{TseYun09} and will be omitted here. \end{proof}  %We refer to \cite[Lemma 3]{TseYun09} for a detailed proof. Let us briefly remark that in \cite{TseYun09} the additional restriction $x \in \dom~r$ is made. However, this assumption is not needed since, in the proof, $r$ is only evaluated at appropriate proximity operators that are always contained in $\dom~r$. \end{proof} 
\begin{remark} \label{remark:ada-lambda} Let $\Lambda \in \Spp$ be given and let $(\Lambda_k)_k \subset \Spp$ be a family of symmetric, positive definite matrices satisfying assumption {\rm(B.1)}. %there exist constants $\lambda_M \geq \lambda_m > 0$ such that $\lambda_M I \succeq \Lambda_k \succeq \lambda_m I$ for all $k \in \N$. 
Then, it easily follows
\[ \scalebox{0.9}{$\displaystyle\frac{\ewmax(\Lambda)}{\lambda_m}$} I \succeq \Lambda_k^{-\half}\Lambda\Lambda_k^{-\half}  \succeq \scalebox{0.9}{$\displaystyle\frac{\ewmin(\Lambda)}{\lambda_M}$} I \quad \text{and} \quad \scalebox{0.9}{$\displaystyle\frac{\lambda_M}{\ewmin(\Lambda)}$} I \succeq \Lambda^{-\half}\Lambda_k\Lambda^{-\half} \succeq \scalebox{0.9}{$\displaystyle\frac{\lambda_m}{\ewmax(\Lambda)}$} I, \]
for all $k \in \N$, and, due to \cref{lemma:F-bound}, we obtain the following bounds
\be \label{eq:new:misc3} \unl \cdot \| \Fsub(x)\| \leq \|F^{\Lambda_k}_s(x)\| \leq \ovl \cdot \|\Fsub(x)\|, \quad \forall~k \in \N, \ee
and for all mini-batches $s$, $x \in \Rn$. The constants $\unl$, $\ovl > 0$ do not depend on $k$, $\Lambda_k$, or $s$. Thus, %if the matrices $\Lambda_k$ remain in a bounded set, 
the latter inequalities imply:
\[ F^\Lambda(x^k) \to 0 \quad \iff \quad F^{\Lambda_k}(x^k) \to 0, \quad k \to \infty. \] 
As indicated in the last section, this can be used in the design of our algorithm. In particular, adaptive schemes %such as the Barzilai-Borwein step size rule, {\rm\cite{BarBor88}}, 
or other techniques can be applied to update $\Lambda$.
\end{remark}

The following result is a simple extension of \cite[Theorem 4]{TseYun09}; see also \cite[Lemma 3.7]{XZ2014} and \cite{XuYin15} for comparison.

\begin{lemma} \label{lemma:str-conv} Suppose that the assumptions {\rm(A.1)}, {\rm(A.3)} are satisfied and let $\Lambda \in \Spp$ be given with $\lamM I \succeq \Lambda \succeq \lamm I $. Furthermore, let $x^*$ denote the unique solution of the problem $\min_x \psi(x)$ and for any $\tau > 0$ let us set
\[ b_1 := L - 2\lamm - \mu_r, \;\, b_2 := \frac{(\lamM + \mu_r)^2}{\bar\mu}, \;\, B_1(\tau) := \frac{1+\tau}{\bar\mu}(\sqrt{b_1 + b_2 + \tau} + \sqrt{b_2})^2 \] 
Then, there exists some positive constant $B_2(\tau)$ that only depends on $\tau$ such that
\be \label{eq:str-conv-prop2} \|x - x^*\|^2 \leq  B_1(\tau) \cdot  \|\Fsub(x)\|^2 + B_2(\tau) \cdot \|\nabla f(x) - \gfsub(x) \|^2, \ee
for all $x \in \Rn$ and for every sample $s$. If the full gradient is used, the term $\|\nabla f(x) - \gfsub(x) \|$ vanishes for all $x$ and \cref{eq:str-conv-prop2} holds with $B_1(\tau) \equiv B_1(0)$.
\end{lemma}
\begin{proof} The proof of \cref{lemma:str-conv} and an explicit derivation of the constant $B_2(\tau)$ are presented in \cref{sec:app-1}. \end{proof}

%----------------------------------------------------------------------------------------------------------------%
% (3.3) CONVERGENCE ANALYSIS
%----------------------------------------------------------------------------------------------------------------%

\subsection{Convergence Analysis} 
In the following, we first verify that a stochastic proximal gradient step yields approximate $\psi$-descent whenever the step size $\alpha_k$ in step \ref{alg:SSN5}  of Algorithm \ref{alg:ssn} is chosen sufficiently small. We also give a bound for the step sizes $\alpha_k$. Let us note that similar results were shown in \cite{XuYin15,GhaLanZha16,GL2016} and that the proof of \cref{prop:prox-descent} mainly relies on the well-known descent lemma
\be \label{eq:lip-ineq} f(y) \leq f(x) + \iprod{\nabla f(x)}{y-x} + \frac{L}{2} \|y-x\|^2, \quad \forall~x,y \in \Rn, \ee  
which is a direct consequence of assumption {\rm(A.1)}.

\begin{lemma} \label{prop:prox-descent} Let $x \in \dom~r$ and $\Lambda \in \Spp$ be arbitrary and suppose that the conditions {\rm(A.1)} and {\rm(B.1)} (for $\Lambda$) are satisfied. Moreover, let $\gamma \in (0,1)$, $\rho \in (1,\gamma^{-1})$, and the mini-batch $s$ be given and set $\overline\alpha := 2(1-\gamma\rho)\lambda_m L^{-1}$. Then, for all $\alpha \in [0,\min\{1,\overline\alpha\}]$ it holds
\be \label{eq:est-prox-step} \psi(x + \alpha v) - \psi(x) \leq -\alpha \gamma \|v\|_{\Lambda}^2 + \frac{\alpha}{4\gamma(\rho-1)\lamm} \|\nabla f(x) - \gfsub(x)\|^2, \ee
where $v := - \Fsub(x)$.
\end{lemma}
\begin{proof} We first define $\Delta := \iprod{\gfsub(x)}{v} + r(x + v) - r(x)$. Then, applying the optimality condition of the proximity operator \cref{eq:prox:sub-prox}, it follows 
\[ \Delta \leq \iprod{\gfsub(x)}{v} + \iprod{-\Lambda v - \gfsub(x)}{v} = - \|v\|_{\Lambda}^2. \]
Using the descent lemma \cref{eq:lip-ineq}, the convexity of $r$, and Young's inequality, we now obtain
\begin{align*} \psi(x + \alpha v) - \psi(x) + \alpha \gamma \|v\|_{\Lambda}^2 & \\ & \hspace{-25ex}\leq  \alpha ( \iprod{\nabla f(x)}{v} + r(x + v) - r(x)) + \frac{L \alpha^2}{2} \|v\|^2 + \alpha \gamma \|v\|_{\Lambda}^2 \\ & \hspace{-25ex} \leq \frac{L \alpha^2}{2} \|v\|^2 - \alpha (1-\gamma) \|v\|^2_{\Lambda} + \alpha \iprod{\nabla f(x) - \gfsub(x)}{v} \\ &\hspace{-25ex} \leq \alpha \left ( \frac{1}{2}L\alpha - (1-\gamma\rho) \lambda_m \right) \|v\|^2 + \frac{\alpha}{4\gamma(\rho-1)\lamm} \|\nabla f(x) - \gfsub(x)\|^2.
\end{align*}
Since the first term is nonpositive for all $\alpha \leq \overline\alpha$, this establishes \cref{eq:est-prox-step}.
\end{proof}

In the special case $\lamm \geq L$ and $\rho \leq (2\gamma)^{-1}$, \cref{prop:prox-descent} implies that the approximate descent condition \cref{eq:est-prox-step} holds for all $\alpha \in [0,1]$. The next lemma is one of our key tools to analyze the stochastic behavior of the Newton iterates and to bound the associated residual terms $\|\Fsubkk(\nstep)\|$.
\begin{lemma} \label{lemma:gen-conv} Let $(Y_k)_k$ be an arbitrary binary sequence in $\{0,1\}$ and let $a_0 \geq 0$, $\eta \in (0,1),$ $p \in (0,1]$, and $(\nu_k)_k \in \ell^1_+$, $(\veps_k)_k \in \ell^p_+$ be given. Let the sequence $(a_k)_k$ be defined by
\[ a_{k+1} := (\eta + \nu_k)^{Y_k} a_k + Y_k \veps_k, \quad \forall~k \in \N_0. \]
Then, for all $R \geq 1$, $k \geq 0$, and all $q \in [p,1]$, it holds 
\[ a_{k+1} \leq C_\nu \left [ a_0 + \sum_{k=0}^\infty \veps_k \right ] \quad \text{and} \quad \sum_{k=0}^{R-1} Y_k a_{k+1}^q \leq  \frac{C_\nu^q}{1-\eta^q} \left [ (\eta a_0)^q + \sum_{k = 0}^{\infty} \veps_k^q \right ], \]
where $C_\nu := \exp\left(\eta^{-1}\sum_{i=0}^\infty \nu_i\right)$.
\end{lemma}
\begin{proof} Using an induction, we can derive an explicit representation for $a_{k+1}$
\be \label{eq:a-rec} a_{k+1} = \left \{ \prod_{i=0}^k (\eta + \nu_i)^{Y_i} \right \} a_0 + \sum_{j = 0}^{k-1} \left \{ \prod_{i=j+1}^k (\eta + \nu_i)^{Y_i} \right \} Y_j \veps_j + Y_k\veps_k, \quad \forall~k \geq 0. \ee
Next, using $Y_i \in \{0,1\}$, $i \in \N$, and $\log(1+\nu_i\eta^{-1}) \leq \nu_i\eta^{-1}$, we obtain the estimate
\be \label{eq:a-fac} \prod_{i=\ell}^k (\eta + \nu_i)^{Y_i} \leq \left \{ \prod_{i=\ell}^k \eta^{Y_i} \right \} \cdot \exp\left ( \eta^{-1} { \sum_{i=\ell}^k Y_i \nu_i } \right) \leq C_\nu \cdot \eta^{\sum_{i=\ell}^k Y_i}, \quad \ell \geq 0. \ee
The bound on $a_{k+1}$ now follows from \cref{eq:a-rec}, \cref{eq:a-fac}, and $\eta \leq 1$. Let us now define the set $K_\ell := \{i \in \{\ell,...,R-1\}: Y_i = 1\}$. Then, it holds 
\[ \sum_{k=\ell}^{R-1} Y_k \eta^{\sum_{i=\ell}^k qY_i} =  \sum_{k \in K_\ell} \eta^{\sum_{i \in K_\ell, i \leq k}q} =  \sum_{j=1}^{|K_\ell|} \eta^{jq} \leq  \sum_{k=\ell}^{R-1} \eta^{(k-\ell+1)q}, \quad \ell \in \{0,...,R-1\}. \]
Combining the last results and using the subadditivity of $x \mapsto x^q$, $q \in [p,1]$, we have
\begin{align*} \sum_{k=0}^{R-1} Y_k a_{k+1}^q & \leq C_\nu^q \sum_{k=0}^{R-1} Y_k \eta^{\sum_{i=0}^k qY_i} a_0^q + C_\nu^q \sum_{k=0}^{R-1} \sum_{j=0}^{k-1} \eta^{\sum_{i=j+1}^k qY_i} Y_k Y_j \veps_j^q+\sum_{k=0}^{R-1} Y_k\veps_k^q \\ & \leq  C_\nu^q \sum_{k=0}^{R-1} \eta^{q(k+1)} a_0^q + C_\nu^q \sum_{j=0}^{R-1} \left\{ \sum_{k=j}^{R-1} \eta^{\sum_{i=j}^k qY_i} Y_k \right \} \eta^{-qY_j} Y_j  \veps_j^q 
%\sum_{k=0}^{R-1} \sum_{j=0}^{k} \eta^{q(k-j)} \veps_j^q 
\\ & \leq \frac{C_\nu^q}{1-\eta^q} \cdot (\eta a_0)^q + C_\nu^q \sum_{j=0}^{R-1} \left\{ \sum_{k=j}^{R-1} \eta^{q(k-j)} \right \} \veps_j^q  \leq \frac{C_\nu^q}{1-\eta^q} \left [(\eta a_0)^q + \sum_{j=0}^{\infty} \veps_j^q \right ] \end{align*}
as desired. Let us also note that the inclusion $\ell^q_+ \subset \ell^p_+$ is used in the last step. 
\end{proof}

We are now in the position to establish global convergence of Algorithm \ref{alg:ssn} in the sense that the expectation ${\mathds E}[\|\FopL(x^k)\|^2]$ converges to zero as $k \to \infty$. We first show convergence of Algorithm \ref{alg:ssn} under the conditions {\rm(C.1)}--{\rm(C.2)} and under the additional assumptions that the step sizes are diminishing and that the scaled stochastic error terms $\alpha_k \sigma_k^2$, $k \in \N$, are summable which is a common requirement in the analysis of stochastic methods for nonsmooth, nonconvex optimization, see, e.g., \cite{XuYin15,GhaLanZha16,GL2016}. 

Our basic idea is to show that both the proximal gradient and the semismooth Newton step yield approximate $\psi$-descent and that the error induced by gradient and Hessian sampling can be controlled in expectation. For a proximal gradient step this basically follows from \cref{prop:prox-descent}. For a Newton step, we combine the growth conditions \cref{eq:growth-1}--\cref{eq:growth-2} and \cref{lemma:gen-conv} to establish an estimate similar to \cref{eq:est-prox-step}. An analogous strategy was also used in \cite{Mil16,MilUlb14}. In our situation, however, a more careful discussion of the possible effects of the semismooth Newton steps is needed to cope with the stochastic situation. More specifically, since our convergence result is stated in expectation, all possible realizations of the random mini-batches $s^k$ and $t^k$, $k \in \N_0$, and their influence on the conditions \cref{eq:growth-1}--\cref{eq:growth-2} have to be considered. In order to apply \cref{lemma:gen-conv}, we now set up some preparatory definitions.  

Let $k \in \N_0$ be given and let us define ${\sf Q}_k : \Rn \to \R$,  ${\sf Q}_k(a) := \|a\| - (\eta+\nu_k) \theta_k$ and ${\sf P}_k : \Rn \times \Rn \to \R$, ${\sf P}_k(a,b) := \psi(a) - \beta \theta_k^{1-p} \|b\|^p - \psi(x^k)$, and
\[ {\sf S}_k := [\dom~r \times \Rn] \; \cap \; [\Rn \times \lev_{\veps_k^1}\,{\sf Q}_k] \; \cap \; {\mathrm{lev}}_{\veps_k^2}\,{\sf P}_k. \]
Then, setting ${\sf Y}_{k+1} := {\mathds 1}_{{\sf S}_k}(\nstep,\Fsubkk(\nstep))$, it holds
\[  {\sf Y}_{k+1} = \begin{cases} 1 & \text{if $\nstep = x^k +d^k$ is feasible and satisfies the conditions \cref{eq:growth-1} and \cref{eq:growth-2}}, \\ 0 & \text{otherwise} \end{cases} \]
and consequently, each iterate $x^{k+1}$ can be calculated as follows
\be \label{eq:def-xkk} \begin{aligned} x^{k+1} &= (1-{\sf Y}_{k+1}) \pstep + {\sf Y}_{k+1} \nstep \\ &= (1-{\sf Y}_{k+1})[ x^k - \alpha_k \Fsubk(x^k) ] + {\sf Y}_{k+1} [ x^k - M_k^+ \Fsubk(x^k) ].  \end{aligned} \ee
Here, the matrix $M_k^+$ denotes the \textit{Moore-Penrose inverse} of the generalized derivative $M_k$. Let us note that this compact representation of our iterative scheme turns out to be particularly useful in the proof of \cref{fact:one}, see \cref{sec:app-fact}. We also introduce the parameters ${\sf Z}_{k}$, $k \in \N_0$, which are defined recursively via
\[ {\sf Z}_0 := \theta_0 \in \R_+, \quad {\sf Z}_{k+1} := (\eta + \nu_k)^{{\sf Y}_{k+1}} {\sf Z}_k + {\sf Y}_{k+1} \veps_k^1. \]
By construction of Algorithm \ref{alg:ssn} and by induction, we have $\theta_k \leq {\sf Z}_k$ and thus,
\be \label{eq:new-fund-2} {\sf Y}_{k+1} \| \Fsubkk(\nstep)\| \leq {\sf Y}_{k+1}{\sf Z}_{k+1}, \quad \forall~k \in \N_0. \ee
Moreover, by identifying $Y_k \equiv {\sf Y}_{k+1}$ and $a_k \equiv {\sf Z}_{k}$, \cref{lemma:gen-conv} yields the following sample-independent and uniform bounds
\be \label{eq:new-fund-3}  \sum_{k=0}^{R-1} {\sf Y}_{k+1} {\sf Z}_{k+1}^q \leq \frac{C_\nu^q}{1-\eta^q} \left [ (\eta \theta_0)^q + \sum_{k=0}^\infty (\veps_k^1)^q \right ] =: C_{z}(q) < \infty \ee
and ${\sf Z}_{k+1} \leq C_\nu \left[ \theta_0 + \sum_{k=0}^\infty \veps_k^1 \right] =: C_z $ for all $k \geq 0$, $R \in \N$, and $q \in [p,1]$. We now state our main result of this section.

\begin{theorem} \label{theorem:conv-prox-v2} Let the sequence $(x^k)_k$ be generated by Algorithm \ref{alg:ssn}. Suppose that the assumptions {\rm(A.1)}--{\rm(A.2)}, {\rm(B.1)}--{\rm(B.2)}, and {\rm(C.1)}--{\rm(C.2)} are satisfied. Furthermore, suppose that the step sizes $\alpha_k \in [0,1]$, $k \in \N$, are chosen such that the approximate descent condition \cref{eq:est-prox-step} holds for some given $\gamma$ and $\rho$. %with $\rho \equiv \rho_k$ and that there exists ${\underline\rho} > 0$ with $\rho_k \in [{\underline\rho},\infty)$ for all $k \in \N$. 
Then, under the additional assumptions
\[ (\alpha_k)_k \, \text{ is monotonically decreasing}, \quad \sum~\alpha_k = \infty, \quad \sum~\alpha_k \sigma_k^2 < \infty, \] 
it holds $\liminf_{k \to \infty} {\mathds E}[\|F^\Lambda({{x}}^k)\|^2] = 0$ and $\liminf_{k \to \infty} F^\Lambda({x}^k) = 0$ a.s. for any $\Lambda \in \Spp$.
\end{theorem}
\begin{proof}  Assumption (A.1) implies that the gradient mapping $\nabla f(x)$ is Lipschitz continuous on $\Rn$ with Lipschitz constant $L$. Thus, for any matrix $\Gamma \in \Spp$ with $\lamM I \succeq \Gamma \succeq \lamm I$, we obtain the Lipschitz constant $1+ L \lamm^{-1}$ for $u^\Gamma(x) $. Since the proximity operator $\proxt{\Gamma}{r}$ is $\Gamma$-nonexpansive, we now have 
\begin{align*} \|F^\Gamma(x) - F^\Gamma(y)\| & \leq \|x-y\| + \lamm^{-\half} \|\proxt{\Gamma}{r}(u^\Gamma(x)) - \proxt{\Gamma}{r}(u^\Gamma(y))\|_\Gamma \\ & \leq \|x-y\| + (\lamm^{-1}\lamM)^{\half} \|u^\Gamma(x) - u^\Gamma(y)\| \\ & \leq (1+({\lamm^{-1}\lamM})^\half + L ({\lamm^{-3}\lamM})^\half) \|x-y\| \end{align*}
for all $x,y \in \Rn$. Hence, by assumption (B.1), the functions $x \mapsto F^{\Lambda_k}(x)$, $k \in \N$, are all Lipschitz continuous on $\Rn$ with modulus $L_{F} :=  1+({\lamm^{-1}\lamM})^\half + L ({\lamm^{-3}\lamM})^\half$.

We first consider the case where $x^{k+1} = \pstep = x^k + \alpha_k v^k$ is generated by the proximal gradient method in step \ref{alg:SSN5}. Then, due to (B.1) and \cref{remark:ada-lambda}, there exists a constant $\underline{\lambda} = \underline{\lambda}(\lamm,\lamM)$ such that
\begin{align*} 
\|F^{\Lambda}(x^{k+1})\| & \leq {\unl}^{-1} ( L_{F} \alpha_k \|v^k\| + \lamm^{-\half} \|F^{\Lambda_k}(x^k) - \Fsubk(x^k)\|_{\Lambda_k} + \lamm^{-\half} \|\Fsubk(x^k)\|_{\Lambda_k} ) \\ & \leq \unl^{-1}\lamm^{-\half} (L_F + 1)  \|\Fsubk(x^k)\|_{\Lambda_k} + (\unl \lamm)^{-1} \|\nabla f(x^k) - \gfsubk(x^k)\|. 
\end{align*}
Here, we again used the $\Lambda_k$-nonexpansiveness of the proximity operator $\proxt{\Lambda_k}{r}$. Thus, applying \cref{prop:prox-descent}, using the estimate $\|a+b\|^2 \leq 2 \|a\|^2 + 2 \|b\|^2$, for $a, b \in \Rn$, and setting $\mathcal E_k := \|\nabla f(x_k) - \gfsubk(x^k)\|$, we obtain
\begin{align*} \psi(x^k) - \psi(x^{k+1}) & \\ &\hspace{-12ex} \geq \underbracket{\begin{minipage}[t][6ex][t]{12ex}\centering$\displaystyle\frac{\gamma\unl^2\lamm}{2(L_F+1)^2}$\end{minipage}}_{=: \, c_1} \cdot \, \alpha_k \|F^\Lambda(x^{k+1})\|^2 - \underbracket{\begin{minipage}[t][6ex][t]{30ex}\centering$\displaystyle\frac{1}{\lamm} \left( \frac{\gamma}{(L_F+1)^2} + \frac{1}{4\gamma(\rho-1)} \right)$\end{minipage}}_{=: \, c_2}  \cdot \, \alpha_k\mathcal E_k^2. 
\end{align*}

Next, we derive a similar estimate for a Newton step $x^{k+1} = \nstep = x^k + d^k$. As before and due to assumption (B.1) and \cref{remark:ada-lambda}, we have
\be \label{eq:flam-err} \|F^\Lambda(x^{k+1})\|^2 \leq 2 \unl^{-2} \|\Fsubkk(\nstep)\|^2 + 2(\unl\lamm)^{-2} \|\nabla f(x^{k+1}) - \gfsubkk(x^{k+1}) \|^2. \ee
Combining the growth condition \cref{eq:growth-2}, \cref{eq:flam-err}, and the bound $\alpha_{k+1} \leq 1$, it holds 
\begin{align*} \psi(x^k) - \psi(x^{k+1}) & \geq - \beta \theta_k^{1-p} \|\Fsubkk(\nstep)\|^p - \veps_k^2 \\
& \geq  c_1 \cdot \alpha_{k+1} \|F^\Lambda(x^{k+1})\|^2 - \veps_k^2 - 2 c_1(\unl\lamm)^{-2} \cdot  \alpha_{k+1}\mathcal E_{k+1}^2 \\ & \hspace{2ex} -  \underbracket{\begin{minipage}[t][4ex][t]{34.5ex}\centering$\displaystyle\left( 2 c_1\unl^{-2} \|\Fsubkk(\nstep)\|^{2-p} + \beta \theta_k^{1-p}\right)$\end{minipage}}_{=: \,{\sf T}_k}  \|\Fsubkk(\nstep)\|^{p}.  \end{align*}
Furthermore, using $\|\Fsubkk(\nstep)\| = \theta_{k+1} \leq C_z$ and $\theta_k \leq C_z$, it can be easily shown that the term ${\sf T}_k$ is bounded by a constant $\bar {\sf T}$ that does not depend on any of the random mini-batches $s^j$, $t^j$, $j \in \N_0$. 

Now, let $R \in \N$ be arbitrary. Then, the monotonicity of $(\alpha_k)_k$, \cref{eq:new-fund-2}--\cref{eq:new-fund-3}, and our last results imply
\begin{align*} \psi(x^0) - \psi(x^{R+1}) & \\ & \hspace{-14ex} \geq \sum_{k=0}^{R}   c_1 \min\left\{{\alpha_k},{\alpha_{k+1}}\right\} \|F^{\Lambda}(x^{k+1})\|^2 - \sum_{k=0}^R (1- {\sf Y}_{k+1}) c_2 \cdot \alpha_k  \mathcal E_k^2 \\ & \hspace{-10ex} - \sum_{k=0}^R  {\sf Y}_{k+1} \left [ 2 c_1(\unl\lamm)^{-2} \cdot \alpha_{k+1}  \mathcal E_{k+1}^2 + \veps_k^2 + \bar {\sf T} \cdot  \|\Fsubkk(\nstep)\|^p \right ] \\
&  \hspace{-14ex} \geq  \sum_{k=0}^{R} \left[   c_1 \alpha_{k+1} \|F^{\Lambda}(x^{k+1})\|^2 -  c_2 \alpha_k \mathcal E_k^2 - 2 c_1(\unl\lamm)^{-2} \alpha_{k+1} \mathcal E_{k+1}^2 \right] - \bar {\sf T}C_z(p) - \sum_{k=0}^R \veps^2_k. \end{align*}
Thus, taking expectation and setting $ c_3 := c_2 +  2 c_1(\unl\lamm)^{-2}$, we obtain
\[  \sum_{k=0}^R c_1{\alpha_{k+1}}{\mathds E}[\|F^{\Lambda}(x^{k+1})\|^2] \leq \psi(x^{0}) - {\mathds E}[\psi(x^{R+1})] + \bar {\sf T} C_z(p) + \sum_{k=0}^R \veps_k^2+ {c_3} \sum_{k=0}^{R+1} \alpha_k{\sigma_k^2}. \]
Since the objective function $\psi$ is bounded from below and the sequences $(\alpha_k\sigma_k^2)_k$ and $(\veps^2_k)_k$ are summable, this obviously yields $\sum \alpha_k {\mathds E}[\|F^{\Lambda}(x^{k})\|^2] < \infty$. Consequently, our first claim follows from the assumption $\sum \alpha_k = \infty$. On the other hand, Fatou's lemma implies
\[ \mathds E\left[ \sum_{k=0}^\infty \alpha_{k} \|F^\Lambda(x^{k})\|^2 \right] \leq \liminf_{R \to \infty} \mathds E\left[ \sum_{k=0}^R {\alpha_{k}}\|F^{\Lambda}(x^{k})\|^2 \right] < \infty \]
and hence, we have $\sum \alpha_k \|F^\Lambda(x^k)\|^2 < \infty$ with probability 1. As before we can now infer $\liminf_{k \to \infty} F^\Lambda(x^k) = 0$ with probability 1 which completes our proof.  
\end{proof}

\begin{remark} \label{remark:sample-size}
In the case $\sum \sigma_k^2 < \infty$ and if the step sizes $\alpha_k$ are fixed or bounded, our results in \cref{theorem:conv-prox-v2} can be strengthened to $\lim_{k \to \infty} \Exp [\|F^\Lambda(x^k)\|^2] = 0$ and we have $\lim_{k \to \infty} F^\Lambda(x^k) = 0$ almost surely. Let us now assume that the samples $s^k_i$, $i \in [\ngk]$, $k \in \N$, are chosen independently of each other and that the conditions
\be \label{eq:remark-uni-exp} \Exp[\cG(x,s^k_i)] = \nabla f(x), \quad\quad \Exp[\| \nabla f(x)- \cG(x,s^k_{i})\|^2] \leq \bar\sigma^2, \ee
hold uniformly for all $i \in [\ngk]$, $k \in \N_0$, and $x \in \Rn$ and for some $\bar\sigma > 0$. Then, as shown in \cite{GhaLanZha16,IusJofOliTho17}, it follows $\Exp[\|\nabla f(x) - \gfsubk(x)\|^2] \leq \bar\sigma^2 [\ngk]^{-1}$ for all $x \in \Rn$ and consequently, due to ${\sf Y}_k \in \{0,1\}$ and $z^{k-1}_{\sf n}, z^{k-1}_{\sf p} \in \mathcal F_{k-1}$, we have  
\begin{align*} \Exp[\|\nabla f(x^k) - \gfsubk(x^k)\|^2] & \leq 2 \Exp[(1-{\sf Y}_k) \|\nabla f(z^{k-1}_{\sf p})- \gfsubk(z^{k-1}_{\sf p})\|^2] \\ & \hspace{4ex}+ 2 \Exp[{\sf Y}_k \|\nabla f(z^{k-1}_{\sf n}) - \gfsubk(z^{k-1}_{\sf n}) \|^2] \leq 4\bar\sigma^2 [\ngk]^{-1}. \end{align*}
Hence, one way to guarantee summability of the error terms $\sigma_k^2$ is to asymptotically increase the sample size $\ngk$ and set $\ngk = \mathcal O(k^{1+\varpi})$ for some $\varpi > 0$. This observation is similar to the results in \cite{XuYin15,GhaLanZha16}. We will discuss the conditions \cref{eq:remark-uni-exp} in more detail in the next section.
\end{remark}

In the following, we present a situation where the approximate $\psi$-descent condition \cref{eq:growth-2} is not needed in order to guarantee global convergence of the method. In applications, this can be quite important, since calculating the full objective function $\psi$ may be similarly expensive as evaluating the full gradient $\nabla f$. The following variant of \cref{theorem:conv-prox-v2} is mainly based on the strong convexity assumption (A.3) and on the boundedness assumption (A.4).

\begin{theorem} \label{theorem:conv-prox-strong} Let the sequence $(x^k)_k$ be generated by Algorithm \ref{alg:ssn} without checking the growth condition \cref{eq:growth-2}. Suppose that the assumptions {\rm(A.1)}, {\rm(A.3)}--{\rm(A.4)}, {\rm(B.1)}--{\rm(B.2)}, and {\rm(C.1)}--{\rm(C.2)} are satisfied. Furthermore, suppose that the step sizes $(\alpha_k)_k$ are chosen via $\alpha_k \in [\underline{\alpha}, \min\{1,\overline{\alpha}\}]$ for some $\underline{\alpha} > 0$ and all $k \in \N$. Then, under the additional assumption
\[  \sum \sigma_k < \infty, \] 
it holds $\lim_{k \to \infty} {\Exp}[\|F^\Lambda(x^k)\|^2] = 0$ and $\lim_{k\to\infty} F^\Lambda(x^k) = 0$ a.s. for any $\Lambda \in \Spp$ .
\end{theorem}

\begin{proof}  As in the proof of \cref{theorem:conv-prox-v2}, we want to derive suitable lower bounds for the $\psi$-descent $\psi(x^k) - \psi(x^{k+1})$. We first consider the case where $x^{k+1} = \pstep$ is generated by the proximal gradient method. Then, as shown in the proof of \cref{theorem:conv-prox-v2} and using the bound on $\alpha_k$, we have 
\[ \psi(x^k) - \psi(x^{k+1}) \geq c_1  \underline{\alpha} \|F^\Lambda(x^{k+1})\|^2 - c_2 \cdot \mathcal E_k^2. \]
%
%where we reused the notation and constants introduced in \cref{theorem:conv-prox-v2} and the bound on $\alpha_k$.

Next, we discuss the second case $x^{k+1} = \nstep$. By \cref{lemma:str-conv} and reusing the estimate $\|\Fsubkk(\nstep)\| \leq C_z$ (see again \cref{eq:new-fund-2}), it holds
\[ \|\nstep - x^*\|^2 \leq B_1(\tau)  \|\Fsubkk(\nstep)\|^2 + B_2(\tau) \mathcal E_{k+1}^2 \leq B_1(\tau)C_z \|\Fsubkk(\nstep)\| + B_2(\tau) \mathcal E_{k+1}^2 \]
for some $\tau > 0$. By assumption (A.3), the functions $f - \frac{\mu_f}{2}\|\cdot\|^2$ and $r - \frac{\mu_r}{2}\|\cdot\|^2$ are convex (and directionally differentiable) and hence, we have 
\be \label{eq:str-conv-dir} \psi(y) - \psi(x) \geq r^\prime(x;y-x) + \iprod{\nabla f(x)}{y-x} + \frac{\bar\mu}{2} \|y-x\|^2, \quad \forall~x,y \in \dom~r. \ee 
Now, applying the optimality of $x^*$, \cref{eq:str-conv-dir} with $x \equiv \nstep$ and $y \equiv x^*$, $\nstep \in \dom~r$, the Lipschitz continuity of $\nabla f$ and $F^\Lambda$, the subadditivity of the square root, and defining 
\[ d_1 := \sqrt{B_1(\tau)}(\bar g_r + \|\nabla f(x^*)\|) + B_1(\tau)C_z L, \quad d_2 := \sqrt{B_2(\tau)}(\bar g_r + \|\nabla f(x^*)\|), \] 
we obtain
\begin{align*} \psi(x^k) - \psi(x^{k+1}) & = \psi(x^k) - \psi(x^*) + \psi(x^*) - \psi(x^{k+1}) \\ & \geq r^\prime(\nstep; x^*-\nstep) + \iprod{\nabla f(\nstep)}{x^* - \nstep} + \frac{\bar\mu}{2} \|\nstep - x^*\|^2 \\ & \geq -(\bar g_r + \|\nabla f(x^*)\|) \cdot \|\nstep - x^*\| + \left(\frac{\bar\mu}{2} - L \right) \|\nstep - x^*\|^2 \\ & \geq - d_1  \|\Fsubkk(\nstep)\| - d_2 \mathcal E_{k+1}  - B_2(\tau)L \mathcal E_{k+1}^2 + \frac{\bar\mu{\underline{\lambda}}^2}{2L_F^2} \|F^\Lambda(x^{k+1})\|^2.  \end{align*}
Combining the last inequalities, setting $d_3 :=  \min\{c_1\underline{\alpha}, (2L_F^2)^{-1}\bar\mu{\underline{\lambda}}^2\}$ and using again \cref{eq:new-fund-3} with $q = 1$, it holds
\begin{align*}
\psi(x^0) - \psi(x^{R+1}) \geq \sum_{k=1}^{R+1} \left[ d_3 \|F^\Lambda(x^{k})\|^2 - d_2 \mathcal E_k \right] - (c_2 + B_2(\tau)L) \sum_{k=0}^{R+1} \mathcal E_k^2 - d_1 C_z(1),  \end{align*}
for all $R \in \N$. Taking expectation, our first claim now follows from (C.2), Jensen's inequality, $\sum \sigma_k < \infty$, and from the lower boundedness of $\psi(x^{R+1})$. The probabilistic convergence of the sequence $(F^\Lambda(x^k))_k$ can then be inferred as in the proof of \cref{theorem:conv-prox-v2}.
\end{proof}

\begin{remark} \label{remark:example-c3} Let us note that assumption {\rm(A.4)} is required to derive a suitable lower bound for the difference terms $\psi(x^k) - \psi(\nstep)$ which allows us to apply \cref{lemma:gen-conv}. Furthermore, condition {\rm(A.4)} is always satisfied in the following situation. Suppose that the mapping $r$ has the special form $r = \iota_{\mathcal C} + \varphi$, where $\varphi : \Rn \to \R$ is a real-valued, convex function and $\iota_{\mathcal C} : \Rn \to \Rex$ is the indicator function of a nonempty, convex, and closed set $\mathcal C \subset \Rn$. Then, assumption {\rm(A.4)} holds if the set $\mathcal C$ is either compact or if $\varphi$ is positively homogeneous. In particular, condition {\rm(A.4)} is satisfied if $r$ is a norm. \end{remark}
\begin{proof} For every feasible point $x \in \dom~r = \mathcal C$, we have $0 \in \partial \iota_{\mathcal C}(x)$ and thus, $\partial \varphi(x) \subset \partial r(x)$. By \cite[Proposition 16.17]{BauCom11}, the set $\bigcup_{x \in \mathcal C} \partial \varphi(x)$ is bounded if $\mathcal C$ is bounded. On the other hand, if $\varphi$ is positively homogeneous, then it follows $\partial \varphi(x) = \{\lambda \in \partial \varphi(0): \iprod{\lambda}{x} = \varphi(x)\} \subset \partial \varphi(0)$, see, e.g., \cite[Proposition 16.18]{BauCom11} and \cite[Example 2.5.17]{Mil16}. Since $\partial \varphi (0)$ is again a compact set, this proves our claim.
\end{proof}

\begin{remark} The result in \cref{theorem:conv-prox-strong} can be further improved by additionally damping the semismooth Newton step and setting $\nstep := x^k + \alpha_k d^k$. Then, due to the convexity of $\psi$, we have $\psi(x^k) - \psi(\nstep) \geq \alpha_k (\psi(x^k) - \psi(x^k+d^k))$ and we can use the weaker conditions
\[ (\alpha_k)_k \, \text{ is monotonically decreasing}, \quad \sum~\alpha_k = \infty, \quad \sum~\alpha_k\sigma_k < \infty \]
to guarantee $\liminf_{k \to \infty} {\mathds E}[\|F^\Lambda(x^k)\|^2] = 0$. Similar to \cite{FS2012, XuYin15,BBN2016} it is also possible to derive global convergence rates in terms of the expected distance to optimality $\mathds E[\psi(x^k) - \psi(x^*)]$. However, in our case these rates will depend on the occurrence and total number of accepted Newton steps which are of stochastic nature in general.  
 \end{remark}
 
Finally, let us emphasize that our global results do not explicitly depend on the sampling strategy or on any (uniform) invertibility properties of the stochastic second order oracle $\hfsub$ or of the chosen generalized derivatives $M_k$. Moreover, our results still hold if a different type of direction $d^k$ is used instead of the semismooth Newton direction $d^k = - M_k^+ \Fsubk(x^k)$. (In our proofs, we only require $\mathcal F_k$-measurability of $d^k$). 
%
%\[ \cG(x^k,s^k) \approx \nabla f(x^k) \quad \text{and} \quad \cH(x^k,t^k) \approx \nabla^2 f(x^k), \]
%
%as in, e.g., \cite{GL2013,GL2016,GhaLanZha16,WMGL2017}, are used for building the gradient and Hessian sample. %\cite{JZ2013,DBLJ2014,XZ2014,SLB2017}

%----------------------------------------------------------------------------------------------------------------%
% LOCAL CONVERGENCE
%----------------------------------------------------------------------------------------------------------------%

\section{Local convergence}\label{sec:local} 

In this part of the paper, we analyze the local convergence properties of our proposed method in detail. We will focus on a probabilistic setting, i.e., we consider a single trajectory of the stochastic process $(x^k)_k$ and show that transition to fast local convergence and a fast rate of convergence can be achieved with high probability if the sample sizes $\ngk$ and $\nhk$ are chosen appropriately. With a slight abuse of notation, we will use $(x^k)_k$ to denote either the underlying stochastic process or a corresponding trajectory generated by a single run of Algorithm \ref{alg:ssn} which should be clear from the context.

Our analysis heavily relies on different second order properties of the proximity operator $\proxt{\Lambda}{r}$ and on concentration inequalities for vector- and matrix-valued martingales. In particular, these inequalities will allow us to quantify and control the errors induced by the stochastic oracles and by approximating the gradient and Hessian of $f$. A similar strategy was also used in \cite{RKM2016I,RKM2016II,XRKM2017,YeLuoZha17} for the analysis of pure, sub-sampled Newton methods for smooth optimization problems. %As we will see, a major challenge in our discussion is to effectively handle the dependence of the set $\dsetM(x)$ on both sample batches $s$ and $t$.  %More specifically, since the chosen generalized derivatives $D \in \partial \proxt{\Lambda}{r}(\usub(x))$ explicitly depend on $s$, fast local convergence can only be established under some appropriate and uniform regularity conditions on the proximity operator $\proxt{\Lambda}{r}$. 
In the next subsection, we present our local assumptions and the mentioned concentration results. 
%By applying appropriate , these conditions will allow us to quantify the occurrence and (conditional) probability of the events

%----------------------------------------------------------------------------------------------------------------%
% (4.1) ASSUMPTIONS AND CONDITIONED CONCENTRATION 
%----------------------------------------------------------------------------------------------------------------%

\subsection{Assumptions and Conditional Concentration Inequalities} \label{sec:loc-condconc}
We will mainly work with the following set of local assumptions.

\begin{assumption} \label{assumption:local} Let the trajectory $(x^k)_k$ and the sequence $(\Lambda_k)_k$ be generated by Algorithm \ref{alg:ssn} and suppose that $x^*$ and $\Lambda_*$ are accumulation points of $(x^k)_k$ and $(\Lambda_k)_k$, respectively. We assume that the following conditions are satisfied. \vspace{.5ex}
\begin{itemize}
\item[{\rm(D.1)}] There exists $\bar k \in \N$ such that $\Lambda_k = \Lambda_*$ for all $k \geq \bar k$.
\item[{\rm(D.2)}] The function $f : \Rn \to \R$ is twice continuously differentiable on $\Rn$.
\item[{\rm(D.3)}] The proximity operator $\proxt{\Lambda_*}{r}$ is semismooth at $u^{\Lambda_*}(x^*)$. %and there exists $K_* > 0$ such that $\|\nabla^2 f(x)\| \leq K_*$ for all $x \in B_{\bar\veps}(x^*)$. 
%\end{itemize}
%
%\[  \ewmin{(\nabla^2 f(x))} \geq \nu_* \quad \text{and} \quad \ewmax{(\nabla^2 f(x))} \leq K_*, \quad \forall~x \in B_{\bar\veps}(x^*). \]
%
%\begin{itemize} 
\item[{\rm(D.4)}] The function $\psi$ is Lipschitz continuous in a neighborhood of $x^*$ with constant $L_\psi$. 
\item[{\rm(D.5)}] There exists $C > 0$ such that every generalized derivative $M \in \mathcal M^{\Lambda_*}(x^*)$ is nonsingular with $\|M^{-1}\| \leq C$. 
%The residual mapping $F^{\Lambda_*}$ is semismooth at $x^*$ %in the following sense: there exists $C_s > 0$ such that for all $x \in B_{\bar\veps}(x^*)$ and all sample sets $\cG$ with $\|\nabla f_\cG(x) - \nabla f(x)\| \leq \bar\delta$ it holds
%\item[{\rm(D.6)}] The mapping $\cG(\cdot,z) : \Rn \to \R$ is continuously differentiable and we have $\cH(\cdot,z) = \nabla_x \cG(\cdot,z)$ for all $z \in \Xi$. Moreover,  Uniform semismoothness.

\end{itemize}
If, in addition, $x^*$ is a stationary point of \cref{eq:prob}, then we assume:
\begin{itemize}
\item[{\rm(D.6)}] The accumulation point $x^*$ is a local minimum of the problem \cref{eq:prob}.
\end{itemize}
%\[ \|F^{\Lambda_*}_\cG(x) - F^{\Lambda_*}_\cG(x^*) - M(x-x^*)\| \leq C_s \|x-x^*\|^2, \quad \forall~M \in \mathcal M^{\Lambda_*}_{\cG,\cG}(x). \]
\end{assumption}

Let us briefly discuss the conditions in \cref{assumption:local}. Assumption (D.5) can be interpreted as a \textit{BD-} or \textit{CD-regularity condition} which is a common condition in the local analysis of nonsmooth optimization methods, see, e.g., \cite{QiSun93,Qi93,PanQi93}. %On the other hand, differentiability of the proximity operator -- as stated in (D.5) -- is undoubtedly a strong requirement. We use this condition to show that the distance between the matrices $M \in \mathcal M^{\Lambda_*}_{s,t}(x)$ and the set $\mathcal M^{\Lambda_*}(x)$ can be arbitrarily small and uniformly controlled whenever $x$ is sufficiently close to $x^*$ and the stochastic gradient and Hessian are close to $\nabla f(x)$ and $\nabla^2 f(x)$, respectively. In \cite{SteThePat17}, Stella et al$\text{.}$ use a very similar assumption to establish local convergence of their deterministic FBE-Newton-type method for convex composite problems. We note that if $x^*$ is a stationary point of \cref{eq:prob} and if $r$ is \textit{fully decomposable} at $x^*$, then Fr\'{e}chet differentiability of $\proxt{\Lambda_*}{r}$ is equivalent to the \textit{strict complementarity condition} $-\nabla f(x^*) \in \rinter~\partial r(x^*)$. Full decomposability is related to the concept of \textit{cone reducibility} in conic programming, \cite{BonSha00}, and many important and interesting mappings, such as polyhedral functions, group-sparse penalty terms, and the nuclear norm or the Ky Fan $k$-norm, can be shown to be fully decomposable. We refer to \cite{SteThePat17} and to \cite{Sha03,Mil16} for a detailed discussion of these topics. Additionally, in \cref{theorem:super}, we introduce a stronger, uniform semismoothness condition that still allows us to derive local convergence properties without requiring differentiability of the proximity operator. 
 %Let us note that assumption (D.2) is obviously satisfied if the Hessian $\nabla^2 f(x^*)$ is positive definite. In particular, if $x^*$ is a stationary point, assumption (D.2) also implies that $x^*$ is an isolated local minimum of problem \cref{eq:prob}. 
Let us also mention that, by \cite[Corollary 8.30]{BauCom11}, the Lipschitz condition (D.4) is equivalent to $x^* \in \inter~\dom~r$. Hence, in a suitable neighborhood of $x^*$, any point $x$ will be feasible with $x \in \dom~r$. See \cref{remark:proj-loc} for further comments. Finally, as shown in \cite[section 5.4]{Mil16}, the assumptions (D.5) and (D.6) are both satisfied if $x^*$ is a stationary point and $\nabla^2 f(x^*)$ is positive definite.
%
%Let us consider the error functions $\cG : \Rn \times \Xi \to \R$, $\cH : \Rn \times \Xi \to \R$
%
%\[ \cG(x,z) := \nabla_x F(x,z) - \nabla f(x)  \quad \text{and} \quad \cH(x,z) := \nabla^2_{xx} F(x,z) - \nabla^2 f(x) \]
%
%and let us define $\cG_k : \Rn \to \R^{\ngk}$ and $\cH_k : \Rn \to \R^{\nhk}$,
%
%\[ \cG_k(x) := (\|\cG(x,s^k_1)\|,...,\|\cG(x,s^k_{\ngk})\|), \quad \cH_k(x) := (\|\cH(x,t^k_1)\|,...,\|\cH(x,t^k_{\nhk})\|). \]
%
%In this section, we assume that the integrability conditions mentioned in \cref{sec:global-assumption} are satisfied, i.e., we have $\Exp[F(x,\xi)] < \infty$, $\Exp[\|\nabla_x F(x,\xi)\|] < \infty$, $\Exp[\|\nabla_{xx}^2 F(x,\xi)\|] < \infty$, and there exists $\theta, \zeta \in L^1(\Omega)$ such that
%
%\[ \|\nabla_x F(x,\xi(\omega))\| \leq \theta(\omega), \quad \text{and} \quad \|\nabla_{xx}^2 F(x,\xi(\omega))\| \leq \zeta(\omega), \quad \forall~x \in \Rn. \]
%
%Moreover, we suppose that the following finite mean conditions are satisfied
%
%\be \label{eq:cond-exp-bound} \Exp[\cG_{k}(z^{k-1}_{\sf n})] < \infty, \quad \Exp[\cG_{k}(z^{k-1}_{\sf p})] < \infty, \quad \Exp[\cH_{k}(x^{k})] < \infty \ee
%
%for all $k \in \N$. 
% 

In the following, we introduce three additional conditions that are connected to the variance of the error terms $\errgk$ and $\errhk$,
\[ \errgk(x) :=  \|\gfsubk(x) - \nabla f(x)\|, \quad \errhk(x) := \|\hfsubk(x) - \nabla^2 f(x)\|, \]
and that extend assumption (C.2). %By applying appropriate concentration inequalities for vector- and matrix-valued martingales, these conditions will allow us to quantify the occurrence and (conditional) probability of the events
%
%\[ {\sf G}^{\sf n}_k(\veps) := \{\omega \in \Omega: {\mathcal E}^{\sf g}_{k}(z^{k-1}_{\sf n}(\omega)) \leq \veps\}, \quad {\sf G}^{\sf p}_k(\veps) := \{\omega\in \Omega: {\mathcal E}^{\sf g}_{k}(z^{k-1}_{\sf p}(\omega)) \leq \veps \}, \] 
%\\ & {\sf G}_k(\veps) := \{\omega \in \Omega:\mathcal E^{\sf g}_{s^k(\omega)}(x^k(\omega)) \leq \veps\}, \quad {\sf H}_k(\veps) := \{\omega \in \Omega: {\mathcal E}^{\sf h}_{t^k(\omega)}(x^k(\omega)) \leq \veps\} \end{align*}
%
%${\sf G}_k(\veps) := {\sf G}^{\sf n}_k(\veps) \cap {\sf G}^{\sf p}_k(\veps)$, and ${\sf H}_k(\veps) := \{\omega \in \Omega: {\mathcal E}^{\sf h}_{k}(x^k(\omega)) \leq \veps\}$. 

\begin{assumption} \label{assumption:con-var} We consider the conditions:
%
%Throughout this section, we assume that the stochastic oracles $\sfo$ and $\sso$ generate unbiased estimators of the gradient and Hessian of $f$, i.e., it holds
%
%\[ \Exp[\cG(x,s^k_i)] = \nabla f(x), \quad \text{and} \quad \Exp[\cH(x,t^k_j)] = \nabla^2 f(x), \quad \forall~x \in \Rn, \]
%
%and for all $i \in [\ngk]$, $j \in [\nhk]$ and $k \in \N_0$. Furthermore, we consider the conditions: \vspace{0.5ex}
\begin{itemize}
\item[{\rm (E.1)}] The random mappings $s^k_i$, $t^k_j$ are mutually independent to each other for all $i \in [\ngk]$, $j \in [\nhk]$, and $k \in \N_0$. Furthermore, the stochastic oracles $\sfo$ and $\sso$ generate unbiased estimators of the gradient and Hessian of $f$, i.e., for all $x \in \Rn$ it holds
\end{itemize}
\[ \Exp[\cG(x,s^k_i)] = \nabla f(x), \quad \Exp[\cH(x,t^k_j)] = \nabla^2 f(x), \quad \forall~i \in [\ngk], \; \forall~j \in [\nhk], \; \forall~k \in \N. \]
\begin{itemize}
%\item[] and for all  and $k \in \N_0$.
\item[{\rm (E.2)}] We have $\Exp[[\errgk(z^{k-1}_{\sf n})]^2] < \infty$, $\Exp[[\errgk(z^{k-1}_{\sf p})]^2] < \infty$, $\Exp[[\errhk(x^k)]^2] < \infty$ for all $k \in \N$ and there exists $\bar\sigma, \bar\rho \geq 0$ such that for all $k \in \N_0$, $i \in [\ngk]$, $j \in [\nhk]$ it holds
\end{itemize}
\be \label{eq:var-bound-quad} \Exp[\|\cG(x,s^k_i) - \nabla f(x)\|^2] \leq \bar\sigma^2, \quad \Exp[\|\cH(x,t^k_j) - \nabla^2 f(x)\|^2] \leq \bar\rho^2, \quad \forall~x \in \Rn. \ee
\begin{itemize}
\item[{\rm (E.3)}] There exists $\bar\sigma, \bar\rho > 0$ such that $\Exp[\exp([\errgk(z^{k-1}_{\sf n})]^2/\bar\sigma^2)] < \infty$, 
\[ \Exp[\exp([\errgk(z^{k-1}_{\sf p})]^2/\bar\sigma^2)] < \infty, \quad \Exp[\exp([\errhk(x^k)]^2/\bar\rho^2)] < \infty, \quad \forall~k \in \N, \]
and for all $x \in \Rn$ and all $k \in \N_0$, $i \in [\ngk]$, $j \in [\nhk]$, we have
\end{itemize}
\be \label{eq:var-bound-exp} \Exp[\exp(\|\cG(x,s^k_i)- \nabla f(x)\|^2/\bar\sigma^2)] \leq e, \;\; \Exp[\exp(\|\cH(x,t^k_j) - \nabla^2 f(x)\|^2/\bar\rho^2)] \leq e. \ee
%\end{itemize}
\end{assumption}

The inequalities and properties stated in (E.1)--(E.2) have already been discussed in \cref{remark:sample-size}. As we have seen and as we will verify in \cref{lemma:tail-bound} and in \cref{sec:trans-loc} in more detail, these conditions allow us to bound and control the error terms $\errgk$ and $\errhk$ by means of the sample sizes $\ngk$ and $\nhk$. We note that the conditions given in \cref{assumption:con-var} are commonly used in the complexity and convergence analysis of stochastic optimization methods, see, e.g., \cite{GL2013,BBN2016,GhaLanZha16}. Next, we summarize several conditional, large deviation bounds for vector- and matrix-valued martingales. For more information on tail bounds and additional matrix concentration inequalities, we refer to the papers \cite{JudNem08,Tro12}. To the best of our knowledge, the ``light tail'' result for symmetric random matrices presented in \cref{lemma:tail-bound} (ii) seems to be new. 

\begin{lemma} \label{lemma:tail-bound} Let  $(\mathcal U_k)_{k=0}^m$ be a given filtration of the $\sigma$-algebra $\mathcal F$ and let $\bar \sigma \in \R^m$ be a given vector with $\bar\sigma_k \neq 0$ for all $k$. It holds:
\begin{itemize}
\item[{\rm(i)}] Let $({\sf X}_k)_{k=1}^m$, ${\sf X}_k : \Omega \to \Rn$, be a family of random vectors, satisfying ${\sf X}_k \in \mathcal U_k$, $\Exp[{\sf X}_k \mid \mathcal U_{k-1}] = 0$, and $\Exp[\|{\sf X}_k\|^2 \mid \mathcal U_{k-1}] \leq \bar\sigma_k^2$ $\text{a.e.}$ for all $k \in [m]$. Then, we have 
\end{itemize}
\[ \Exp[\|{\textstyle \sum_{k=1}^m} {\sf X}_k\|^2 \mid \mathcal U_0] \leq \|\bar\sigma\|^2 \quad \text{and} \quad \Prob( \|{\textstyle \sum_{k=1}^m} {\sf X}_k\| \geq \tau \|\bar\sigma\| \mid \mathcal U_0) \leq \tau^{-2}, \quad \forall~t > 0 \]
%\[  \quad \text{and} \quad  \]
\begin{itemize}
\item[] almost everywhere. In addition, if it holds $\Exp[ \exp(\|{\sf X}_k\|^2/\bar\sigma_k^2) \mid \mathcal U_{k-1}] \leq \exp(1)$
$\text{a.e.}$ and for all $k \in [m]$, then with probability 1 it follows
\end{itemize}
\[ \Prob\left(\left\|{\textstyle\sum_{k=1}^m} {\sf X}_k\right\| \geq (1 + \tau) \|\bar\sigma\| \mid \mathcal U_0 \right) \leq \exp(-\tau^2/3), \quad \forall~\tau > 0. \]
\begin{itemize}
\item[\rm(ii)] Let $({\sf X}_k)_{k=1}^m$ be a sequence of symmetric random $n \times n$ matrices satisfying ${\sf X}_k \in \mathcal U_k$, $\Exp[{\sf X}_k \mid \mathcal U_{k-1}] = 0$, and $\Exp[\|{\sf X}_k\|^2 \mid \mathcal U_{k-1}] \leq \bar\sigma_k^2$ $\text{a.e.}$ for all $k \in [m]$. Then, it holds
\end{itemize}
\[ \Prob( \|{\textstyle \sum_{k=1}^m} {\sf X}_k\| \geq \tau \|\bar\sigma\| \mid \mathcal U_0) \leq \kappa_n \cdot \tau^{-2}, \quad \forall~\tau > 0 \]
\begin{itemize}
\item[] $\text{a.e.}$ with $\kappa_n :=(2\log(n+2)-1)e$. Additionally, if we have $\Exp[ \exp(\|{\sf X}_k\|^2/\bar\sigma_k^2) \mid \mathcal U_{k-1}] \leq \exp(1)$ $\text{a.e.}$ and for all $k \in [m]$, then with probability 1 it follows
\end{itemize}
\be \label{eq:mat-bound-light} \Prob\left(\left\|{\textstyle\sum_{k=1}^m} {\sf X}_k\right\| \geq \tau \|\bar\sigma\| \mid \mathcal U_0 \right) \leq 2n \cdot \exp(-\tau^2/3), \quad \forall~\tau > 0. \ee
\end{lemma}

\begin{proof} The first result in part (i) is well-known, see \cite{JudNem08,GhaLanZha16,IusJofOliTho17}. The associated probability bound directly follows from the conditional Markov inequality. Since the Euclidean norm is \textit{$1$-smooth}, the second result in part (i) follows from \cite[Theorem 4.1]{JudNem08}. In \cite{JudNem08}, Juditsky and Nemirovski also verified that the spectral norm is \textit{$\kappa_n$-regular} which implies 
\[ \Exp[\|{\textstyle\sum_{k=1}^m} {\sf X}_k\|^2 \mid \mathcal U_0] \leq \kappa_n \|\bar\sigma\|^2 \quad \text{a.e.} \]
and establishes the first bound in part (ii). The remaining result can be shown by combining the techniques presented in \cite{Tro12} and \cite[Proposition 4.2]{JudNem08}. For the sake of completeness, an explicit proof is given in \cref{sec:app-3}. \end{proof}

Let us now suppose that the assumptions (E.1)--(E.2) are satisfied. Then, using the integrability condition $\Exp[[\errgk(z^{k-1}_{\sf n})]^2] < \infty$, \cref{fact:one}, and since $\cG$ is a Carath\'{e}odory function and the $\sigma$-algebras $\mathcal F_{k-1}$ and $\sigma(s^{k}_i)$ are independent, it follows \vspace{1ex}
\begin{itemize}
\item $\Exp[\cG(z^{k-1}_{\sf n},s^{k}_i) - \nabla f(z^{k-1}_{\sf n}) \mid \mathcal F_{k-1}] = \Exp[\cG(\cdot,s^{k}_i)](z^{k-1}_{\sf n}) - \nabla f(z^{k-1}_{\sf n}) = 0$, \vspace{.5ex}
\item $\Exp[\|\cG(z^{k-1}_{\sf n},s^{k}_i) - \nabla f(z^{k-1}_{\sf n})\|^2 \mid \mathcal F_{k-1}] = \Exp[\|\cG(\cdot,s^{k}_i) - \nabla f(\cdot)\|^2](z^{k-1}_{\sf n}) \leq \bar\sigma^2$, \vspace{1.5ex}
\end{itemize}
$\text{a.e.}$ and for all $i = 1,...,\ngk$ and $k \in \N$, see \cite[Theorem 2.10]{BhaWay16}. In a similar fashion (and since $\sigma(t_j^k)$ and $\hat{\mathcal F}_{k-1}$ are independent), we can derive almost sure bounds for the proximal gradient step $z^{k-1}_{\sf p}$ and for the Hessian error terms $\cH(x^k,t^k_j) - \nabla^2 f(x^k)$, $j \in [\nhk]$. Hence, the results and bounds in \cref{lemma:tail-bound} are applicable in this situation. We also want to point out that the integrability conditions in \cref{assumption:con-var} are only required for the well-definedness of the conditional expectations and for the associated calculus, \cite{Wil91,BhaWay16}. 

\subsection{Second Order Properties and Bounded Invertibility} 
In this subsection, we derive a general invertibility result that can be applied to guarantee bounded invertibility of the generalized derivatives used in our stochastic semismooth Newton framework.  

We start with the presentation of several second order properties of the proximity operator that are essential for our analysis and that are motivated by the results in \cite{HirStrNgu84,MenSunZha05}.

Since the proximity operator is a Lipschitz continuous function, Rademacher's theorem implies that the proximal mapping $\proxt{\Lambda}{r}$ is Fr\'echet differentiable almost everywhere. Let $\Omega_{r}^\Lambda \subset \Rn$ denote the set of all points at which the proximity operator $\proxt{\Lambda}{r}$ is differentiable. Then, as shown in \cite[Section 3.3]{Mil16}, the following statements are valid: \vspace{1ex}
\begin{itemize} 
%\item The function $\envt{\Lambda}{r}$ is twice Fr\'echet differentiable on $\Omega_{r}^\Lambda$. \vspace{.5ex}
\item For all $x \in \Omega_{r}^\Lambda$ the matrix $\Lambda{D}\proxt{\Lambda}{r}(x)$ is symmetric and positive semidefinite. \vspace{.5ex}
\item For all $x \in \Omega_{r}^\Lambda$ the matrix $\Lambda(I - {D}\proxt{\Lambda}{r}(x))$ is symmetric and positive semidefinite. 
\end{itemize}

\vspace{1ex}
A continuity argument shows that the last properties are also satisfied for every generalized derivative $D \in \partial_B \proxt{\Lambda}{r}(x)$. In the following lemma, we summarize our observations and state an analogue result for the Clarke subdifferential of the proximity operator $\proxt{\Lambda}{r}$. Let us mention that Meng et al. \cite{MenSunZha05} established a similar result for metric projections onto convex, nonempty, and closed sets. Our result also extends Theorem 3.2 in \cite{PatSteBem14}.

\begin{lemma} \label{lemma:prox:sec-clarke}
Let $\Lambda \in \Spp$ and $x \in \Rn$ be arbitrary. Then, for every $D \in \partial \proxt{\Lambda}{r}(x)$, the following statements are true: \vspace{0.5ex}
\begin{itemize} 
\item[\rmn{(i)}] The matrices $\Lambda D$ and $\Lambda (I - D)$ are symmetric and positive semidefinite. \vspace{0.5ex}
\item[\rmn{(ii)}] It holds $\iprod{Dh}{\Lambda(I-D)h} \geq 0$ for all $h \in \Rn$.
\end{itemize}
\end{lemma}

\begin{proof} The first part is a consequence of $\partial \proxt{\Lambda}{r}(x) = {\rm conv}(\partial_B \proxt{\Lambda}{r}(x))$ and of the invertibility of $\Lambda$. The proof of the second part is identical to the proof of \cite[Proposition 1]{MenSunZha05} and therefore will be omitted. \end{proof}

Next, we present the promised, local invertibility result. 

\begin{lemma} \label{lemma:invertibility} Suppose that the conditions {\rm(A.1)}, {\rm(D.2)}, and {\rm(D.5)} are satisfied and let $s$ and $t$ be arbitrary sample mini-batches. Furthermore, let $x^*$ and $\Lambda_*$ be given as in \cref{assumption:local} with $\lamM I \succeq \Lambda_* \succeq \lamm I$. Then, for all $\gamma_c \in (0,\beta_c/C)$, $\beta_c \in (0,1)$, there exists $\veps_c > 0$ (that does not depend on $s$ or $t$) such that under the additional conditions  
\be \label{eq:bd-inv-cond} x \in B_{\veps_c}(x^*), \quad \|\nabla f(x) - \gfsub(x)\| \leq \veps_c,  \quad \|\nabla^2 f(x) - \hfsub(x)\| \leq 0.5\lamm\gamma_c, \ee
%
%it follows
the matrices $M \in \mathcal M^{\Lambda_*}_{s,t}(x)$ are all boundedly invertible with $\|M^{-1}\| \leq C / (1-\beta_c)$.
%
%\[ M \in \mathcal M^{\Lambda_*}(x^*) + B_{\gamma_c}(0), \quad \forall~M \in \mathcal M^{\Lambda_*}_{s,t}(x). \]
%
%the following statements are true:
%\begin{itemize}
%\item[{\rm(i)}] If {\rm(D.4)} is satisfied and the bounds \cref{eq:bd-inv-cond} hold with $\gamma_c \leq \beta_c/C$, $\beta_c \in [0,1)$, then the matrices $M \in \mathcal M^{\Lambda_*}_{s,t}(x)$ are boundedly invertible with $\|M^{-1}\| \leq C / (1-\beta_c)$.
%\item[{\rm(ii)}] If condition {\rm(D.5)} holds, then it follows $\mathcal M^{\Lambda_*}_{s,t}(x) \subset \mathcal M^{\Lambda_*}(x) + B_{2\gamma_c}(0)$.
%\end{itemize}
\end{lemma}

\begin{proof} Let us set $K_* := \|\nabla^2 f(x^*)\|$. Since the multifunction $\partial \proxt{\Lambda_*}{r} : \Rn \rightrightarrows \R^{n \times n}$ is upper semicontinuous, there exists $\tilde\veps > 0$ such that
\be \label{eq:prf-upp-semi} \partial \proxt{\Lambda_*}{r}(y) \subset \partial \proxt{\Lambda_*}{r}(u^{\Lambda_*}(x^*)) + B_{\tilde\delta}(0), \quad \forall~y \in B_{\tilde\veps}(u^{\Lambda_*}(x^*)), \ee
where $\tilde\delta := \lamm\gamma_c / (4(\lamm + K_*))$, see, e.g., \cite[Proposition 2.6.2]{Clarke1990}. Moreover, by the continuity of the Hessian $\nabla^2 f$, we also have $\|\nabla^2 f(x) - \nabla^2 f(x^*)\| \leq 0.25\lamm\gamma_c$ for all $x \in B_{\tilde\veps}(x^*)$ without loss of generality. Let us now set $\veps_c := \min\{(L+\lamm)^{-1},1\}\frac{\tilde\veps\lamm}{2}$ and let us consider an arbitrary matrix $M \in \mathcal M^{\Lambda_*}_{s,t}(x)$ with 
\[ M = I - D + D\Lambda^{-1}_* \hfsub(x) \quad \text{and} \quad D \in \partial \proxt{\Lambda_*}{r}(u^{\Lambda_*}_s(x)). \] %and $\bar M \in \mathcal M^{\Lambda_*}(x)$ with 
Then, due to
\begin{align} \nonumber \|u^{\Lambda_*}_s(x) - u^{\Lambda_*}(x^*)\| & \leq \|u^{\Lambda_*}_s(x) - u^{\Lambda_*}(x)\| + \|u^{\Lambda_*}(x) - u^{\Lambda_*}(x^*)\| \\ & \leq \lamm^{-1} \|\nabla f(x) - \gfsub(x)\| + (1+L\lamm^{-1}) \|x-x^*\| \leq \tilde\veps \label{eq:bound-ulam} \end{align}
and \cref{eq:prf-upp-semi}, there exists $D_* \in \partial \proxt{\Lambda_*}{r}(u^{\Lambda_*}(x^*))$ such that $\|D - D_*\| \leq \tilde \delta$. Using \cref{lemma:prox:sec-clarke} (i), we have $\|D\Lambda^{-1}_*\| = \|\Lambda_*^{-\half}[\Lambda_*^{\half}D\Lambda^{-\half}_*]\Lambda^{-\half}_*\| \leq \|\Lambda^{-1}_*\|$. Thus, defining $M_* := I - D_* + D_* \Lambda^{-1}_* \nabla^2 f(x^*) \in \mathcal M^{\Lambda_*}(x^*)$, it follows
\begin{align*} \|M - M_*\| & = \|(D_* - D)(I - \Lambda_*^{-1}\nabla^2 f(x^*)) + D\Lambda_*^{-1}(\hfsub(x) - \nabla^2 f(x^*))\| \\ & \leq (1+\lamm^{-1}K_*)\tilde\delta + \lamm^{-1} (\|\nabla^2 f(x) - \hfsub(x)\| + \|\nabla^2 f(x) - \nabla^2 f(x^*)\|) \leq \gamma_c
\end{align*}
Due to (D.5) and $\gamma_c \leq \beta_c/C$, it now holds $\|M_*^{-1}(M - M_*)\| \leq \beta_c < 1$. Consequently, by the Banach perturbation lemma, $M$ is invertible with
\[ \|M^{-1}\| = \|(M_* + M - M_*)^{-1}\| \leq \frac{\|M_*^{-1}\|}{1- \|M_*^{-1}(M - M_*)\|} \leq \frac{C}{1-\beta_c}. \]
This finishes the proof of \cref{lemma:invertibility}. %Next, let $\bar M \in \mathcal M^{\Lambda_*}(x)$ be arbitrary. Following our last steps, it is easy to show $\bar M \in \mathcal M^{\Lambda_*}(x^*) + B_{\gamma_c}(0)$ (using the same definition of $\veps_c$). Since condition (D.5) implies that $\partial \proxt{\Lambda_*}{r}(u^{\Lambda_*}(x^*))$ and $\mathcal M^{\Lambda_*}(x^*)$ are singletons (see, e.g., \cite[Theorem 2.6.7]{Mil16}), it follows $\|M - \bar M\| \leq 2\gamma_c$ for any $M \in \mathcal M^{\Lambda_*}_{s,t}(x)$. This establishes (ii).
%\[ , \quad \text{and} \quad \bar M = I - \bar D + \bar D\Lambda_*^{-1} \nabla^2 f(x), \]
%
 %, $\bar D \in \partial \proxt{\Lambda_*}{r}(u^{\Lambda_*}(x))$. 
\end{proof}

\subsection{Transition to Fast Local Convergence and Convergence Rates} \label{sec:trans-loc}

We now present our local convergence theory. As mentioned, our analysis and results rely on the observation that the stochastic Newton step $\nstep = x^k + d^k$ is always accepted as a new iterate with high probability if $x^k$ is close to a local solution $x^*$ and if the sample sizes $\ngk$ and $\nhk$ are sufficiently large. This will be discussed in detail in the next theorem. %We first consider a more general setting, i.e., for a given point $x \in \Rn$ and for given sample sets $\cG$, $\cH$, we discuss the properties of a Newton step
%
%\be \label{eq:loc-conv-new} z = x - M(x)^{-1} F^{\Lambda_*}_\cG(x), \quad M(x) \in \mathcal M^{\Lambda_*}_{\cG,\cH}(x). \ee
%
Throughout this subsection, we will work with the following functions and constants $\mu(x) := \min\{(2L_FC)^{-1},1\}  x$, %$C := (2\lamM + 3K_*)/\nu_*$, 
${\mu_p}(x) := \min \{x,\min\{x^{p^{-1}},x^{(1-p)^{-1}}\}\}$,
\[ {\Upsilon}_k := \min \{ \mu(\veps_k^1),{\mu_p}(\veps_k^2)\}, \quad \text{and} \quad \Gamma_k := \min\{{\Upsilon}_{k-1},{\Upsilon}_k\}, \]  
where $L_F$ denotes the Lipschitz constant of $F^{\Lambda_*}$. (See the proof of \cref{theorem:conv-prox-v2}). Furthermore, let us set $\beta_1 := (2^pL_\psi C)^{-1} \min\{\beta,\half\}$ and
\be \label{eq:def-gam}  \beta_2 := \min  \left\{\frac{6\eta}{4L_F C+3\eta}, \beta_1^{\frac{1}{1-p}}\right\}, \quad \gamma_f := \frac{1}{2\max\{C,1\}}\min \left \{ \frac{1}{2},\beta_2 \right \}. \ee

%\begin{lemma} \label{lemma:mat-conc} Let $\delta, \rho \in (0,1)$ and let $\cG, \cH$ be given sample sets.
%\begin{itemize}
%\item[{\rm(i)}] Suppose that the indices in the sample set $\cG$ are chosen uniformly at random {\rm with} replacement and define $\Delta (x) := \max_{i, j \in \cN} \|\nabla f_i(x) - \nabla f_j(x)\|$. Then, with probability $1-\delta$, we have
%\end{itemize}
%\be \label{eq:conc-grad} \| \nabla f(x) - \nabla f_\cG(x) \|_2 \leq \frac{R_\delta(x)}{\sqrt{{\bf n}^g}}, \quad R_\delta(x) := \sqrt{\frac{N-1}{N}} S(x) + \sqrt{0.5\log(\delta^{-1})} \Delta(x). \ee
%\[ |\cG| \geq \frac{1}{\veps^2} \left( \sqrt{\frac{N-1}{N}} S(x) + \sqrt{2\log(\delta^{-1})} \Delta(x) \right)^2 \]
%\begin{itemize}
%\item[{\rm(ii)}] Suppose that assumption {\rm(D.2)} is satisfied and that the indices in $\cH$ are chosen uniformly {\rm with} or {\rm without} replacement. Furthermore, let $x \in B_{\bar\veps}(x^*)$ be arbitrary and assume that the sample set $\cH$ satisfies 
%\end{itemize}
%\be \label{eq:lower-bound-H} {\bf n}^h \geq (16\kappa_*^2 \log(2n\delta^{-1})) \cdot \rho^{-2}, \ee
%\begin{itemize}
%\item[]where $\kappa_* := K_* \nu_*^{-1}$. Then, the following two statements are fulfilled simultaneously with probability $1-\delta$:
%\end{itemize}
%\be \label{eq:conc-hess} \ewmin(\nabla^2 f_\cH(x)) \geq (1-\rho)\nu_* \quad \text{and} \quad \|\nabla^2 f(x) - \nabla^2 f_\cH(x)\|_2 \leq \rho\nu_*. \ee
%\end{lemma}

\begin{theorem} \label{theorem:mega} Assume that the conditions {\rm(A.1)}--{\rm(A.2)}, {\rm(B.1)}--{\rm(B.2)}, {\rm(C.1)}, and {\rm(E.1)}--{\rm(E.2)} are satisfied and let the sequences $(x^k)_k$, $(\Lambda_k)_k$, and $(\alpha_k)_k$ be generated by Algorithm \ref{alg:ssn}. Let $x^*$ and $\Lambda_*$ be accumulation points of $(x^k)_k$ and $(\Lambda_k)_k$ fulfilling the conditions {\rm(D.1)}--{\rm(D.6)} and let $\gamma_f \in (0,1)$ be given as in \cref{eq:def-gam}. Moreover, suppose that the selected step sizes $(\alpha_k)_k$ are bounded via $\alpha_k \in [\underline{\alpha},\min\{1,\overline{\alpha}\}]$ for some $\underline{\alpha} > 0$ and all $k \in \N$. Then, for a given sequence $(\delta_k)_k \subset (0,1)$, the following statements are true:
\begin{itemize}
\item[{\rm(i)}] Suppose that there exists $\bar \ell \in \N$ such that \vspace{-0.5ex}
\end{itemize}
\be \label{eq:bounds-i} \ngk \geq \frac{1}{\delta_k}\left[\frac{2\bar\sigma}{\lamm \Gamma_k}\right]^2, \quad \nhk \geq \frac{\kappa_n}{\delta_k}\left[\frac{2\bar\rho}{\lamm\gamma_f} \right]^2, \quad \forall~k \geq \bar\ell. \vspace{-0.5ex} \ee
\begin{itemize}
\item[] Then, with probability $\delta_* := \prod_{k=\bar\ell}^\infty(1-\delta_k)(1-2\delta_k)$, the point $x^*$ is a stationary point of \cref{eq:prob}, there exists $\ell_* \in \N$ such that $x^k$ results from a stochastic semismooth Newton step for all $k \geq \ell_*$, and the whole sequence $(x^k)_k$ converges to $x^*$. 
\item[{\rm(ii)}] Let $\gamma_\eta \in (0,1)$ and $\bar \ell \in \N$ be given constants and let us set $\Gamma_k^\circ := \min \{\Upsilon_{k-1}^\circ,\Upsilon_k^\circ\}$ and $\Upsilon_k^\circ := \min\{\mu(\min\{\veps_k^1,\gamma_\eta^{k-\bar\ell}\}),\mu_p(\veps_k^2)\}$. Suppose that the bounds \vspace{-0.5ex}
\end{itemize}
\be \label{eq:bounds-ii} \ngk \geq \frac{1}{\delta_k}\left[\frac{2\bar\sigma}{\lamm \Gamma^\circ_k}\right]^2, \quad \nhk \geq \frac{\kappa_n}{\delta_k}\left[\frac{2\bar\rho}{\lamm\gamma_f} \right]^2 \vspace{-0.5ex} \ee
%\begin{align*}  (\circ_k) \quad\; {\bf n}^{g}_k \geq \left[\frac{2R_{\delta_k}(x^k)}{\lamm \min\{w_p^{k-1},w_p^k\}}\right]^2, \quad\; w_p^k := \min\left\{\frac{\min\{\veps_k^1,\omega^{k-\bar\ell}\}}{2L_FC},e_{\veps_k^2}\right\}  \end{align*}
\begin{itemize}
\item[] hold for all $k \geq \bar\ell$. Then, with probability $\delta_*$, the statements in part {\rm(i)} are satisfied and $(x^k)_k$ converges r-linearly to $x^*$ with rate $\max\{\gamma_\eta,\half\}$. 
\item[{\rm(iii)}] Let $(\gamma_k)_k \subset (0,\infty)$ be a non-increasing sequence with $\gamma_k \to 0$ and let $(\rho_k)_k \subset (0,\infty)$ with $\rho_k \to 0$ and $\bar \ell \in \N$ be given. Let us define $\Gamma_k^\diamond := \min \{\Upsilon_{k-1}^\diamond,\Upsilon_k^\diamond\}$ and $\Upsilon_k^\diamond := \min\{\mu(\min\{\veps_k^1,\gamma_k^{k-\bar\ell}\}),\mu_p(\veps_k^2)\}$ and assume that the sample sizes fulfill \vspace{-0.5ex}
\end{itemize}
\be \label{eq:bounds-iii} \ngk \geq \frac{1}{\delta_k}\left[\frac{2\bar\sigma}{\lamm \Gamma^\diamond_k}\right]^2, \quad \nhk \geq \frac{1}{\delta_k\rho_k}  \vspace{-0.5ex} \ee
\begin{itemize}
\item[] for all $k \geq \bar \ell$. Then, with probability $\delta_*$, the statements in part {\rm(i)} are satisfied (for a possibly different $\ell_*$) and $(x^k)_k$ converges r-superlinearly to $x^*$.
\end{itemize}
\end{theorem}

\begin{proof}  
%
%Let $\veps = \veps(\gamma)$ be given as in \cref{lemma:mega} and let us set $\mathcal K_\veps := \{k \in \N: x^k \in B_\veps(x^*)\}$. As in \cref{lemma:mega}, the main idea of our proof is to repeatedly apply the concentration results in \cref{lemma:mat-conc} to obtain estimates for the error terms $\mathcal E_{\cG_k}(x^k) = \mathcal E_k$ and $\mathcal E_{\cH_k}^2(x^k)$. These estimates can then be combined to establish and verify the claims in \cref{theorem:mega}. 
The proof is split into several steps. First, we utilize the concentration results in \cref{lemma:tail-bound} to quantify the occurrence and (conditional) probability of the events
\[ {\sf G}^{\sf n}_k(\veps) := \{\omega \in \Omega: {\mathcal E}^{\sf g}_{k}(z^{k-1}_{\sf n}(\omega)) \leq \veps\}, \quad {\sf G}^{\sf p}_k(\veps) := \{\omega\in \Omega: {\mathcal E}^{\sf g}_{k}(z^{k-1}_{\sf p}(\omega)) \leq \veps \}, \] 
%\\ & {\sf G}_k(\veps) := \{\omega \in \Omega:\mathcal E^{\sf g}_{s^k(\omega)}(x^k(\omega)) \leq \veps\}, \quad {\sf H}_k(\veps) := \{\omega \in \Omega: {\mathcal E}^{\sf h}_{t^k(\omega)}(x^k(\omega)) \leq \veps\} \end{align*}
%
${\sf G}_k(\veps) := {\sf G}^{\sf n}_k(\veps) \cap {\sf G}^{\sf p}_k(\veps)$, and ${\sf H}_k(\veps) := \{\omega \in \Omega: {\mathcal E}^{\sf h}_{k}(x^k(\omega)) \leq \veps\}$ for appropriately chosen $\veps > 0$. %This allows us to control and bound the errors induced by approximating the gradient and Hessian of $f$ with high probability. 
As a result and to some extent, the local convergence analysis then reduces to a discussion of a highly inexact (rather than stochastic) version of the deterministic semismooth Newton method for convex composite programming. In particular, we can reuse some of the strategies presented in \cite{MilUlb14,Mil16} in our proof. Based on the invertibility result in \cref{lemma:invertibility}, we establish convergence of the whole sequence $(x^k)_k$ in step 2. Afterwards, in step 3 and 4, we show that the growth conditions \cref{eq:growth-1}--\cref{eq:growth-2} are always satisfied whenever $k$ is sufficiently large. In the last steps we derive r-linear and r-superlinear convergence rates and prove part (ii) and (iii) of \cref{theorem:mega}. 
%\comm{Short description of the different steps of the proof.}

\textit{Step 1: Probability bounds}. %At first, we utilize the concentration results in \cref{lemma:tail-bound} to quantify the occurrence and (conditional) probability of the events
%
%\[ {\sf G}^{\sf n}_k(\veps) := \{\omega \in \Omega: {\mathcal E}^{\sf g}_{k}(z^{k-1}_{\sf n}(\omega)) \leq \veps\}, \quad {\sf G}^{\sf p}_k(\veps) := \{\omega\in \Omega: {\mathcal E}^{\sf g}_{k}(z^{k-1}_{\sf p}(\omega)) \leq \veps \}, \] 
%\\ & {\sf G}_k(\veps) := \{\omega \in \Omega:\mathcal E^{\sf g}_{s^k(\omega)}(x^k(\omega)) \leq \veps\}, \quad {\sf H}_k(\veps) := \{\omega \in \Omega: {\mathcal E}^{\sf h}_{t^k(\omega)}(x^k(\omega)) \leq \veps\} \end{align*}
%
%${\sf G}_k(\veps) := {\sf G}^{\sf n}_k(\veps) \cap {\sf G}^{\sf p}_k(\veps)$, and ${\sf H}_k(\veps) := \{\omega \in \Omega: {\mathcal E}^{\sf h}_{k}(x^k(\omega)) \leq \veps\}$. 
We want to show that the event ${\sf E} := \bigcap_{k = \bar \ell}^\infty {\sf G}_k({\lamm\Gamma_k}/{2}) \cap {\sf H}_k({\lamm\gamma_f}/{2})$ occurs with probability $\Prob({\sf E}) \geq \delta_*$. Using the assumptions (E.1)--(E.2), \cref{fact:one} and as demonstrated in the paragraph after \cref{lemma:tail-bound}, we can apply the concentration results in \cref{lemma:tail-bound} via identifying ${\sf X}_i \equiv \cG(z_{\sf n}^{k-1},s^k_i) - \nabla f(z^{k-1}_{\sf n})$, $i \in [\ngk]$, etc. Specifically, setting $\tau = \sqrt{1/\delta_k}$ in part (i) and $\tau = \sqrt{\kappa_n/\delta_k}$ in part (ii) of  \cref{lemma:tail-bound} and using the bounds \cref{eq:bounds-i}, it easily follows $\Prob({\sf G}^{\sf n}_k(\lamm\Gamma_k/2) \mid \mathcal F_{k-1}) \geq 1 - \delta_k$, %${\sf G}^{\sf n}_k(\veps_k), {\sf G}^{\sf p}_k(\veps_k) \in \hat{\mathcal F}_{k-1}$, 
\[  \Prob({\sf G}^{\sf p}_k(\lamm\Gamma_k/2) \mid \mathcal F_{k-1}) \geq 1 - \delta_k, \quad \text{and} \quad \Prob({\sf H}_k(\lamm\gamma_f/2) \mid \hat{\mathcal F}_{k-1}) \geq 1 - \delta_k \]
almost everywhere and for all $k \geq \bar\ell$. Let us define $\bar\veps_k := \frac{\lamm\Gamma_k}{2}$ and $\bar\gamma := \frac{\lamm\gamma_f}{2}$. Then, by the tower property of the conditional expectation and by utilizing
\[ {\mathds 1}_{{\sf G}_k(\bar\veps_k)}(\omega) = {\mathds 1}_{{\sf G}^{\sf n}_k(\bar\veps_k) \cap {\sf G}^{\sf p}_k(\bar\veps_k)}(\omega) \geq {\mathds 1}_{{\sf G}^{\sf n}_k(\bar\veps_k)}(\omega) + {\mathds 1}_{{\sf G}^{\sf p}_k(\bar\veps_k)}(\omega) - 1, \quad \forall~\omega \in \Omega,\] 
and ${\sf G}^{\sf n}_k(\bar\veps_k), {\sf G}^{\sf p}_k(\bar\veps_k) \in \hat{\mathcal F}_{k-1}$, ${\sf H}_k(\bar\gamma) \in \mathcal F_k$, $k \geq \bar \ell$, we inductively obtain
\begin{align*} \Prob\left({\textstyle \bigcap_{k=\bar\ell}^L}~{\sf G}_k(\bar\veps_k) \cap {\sf H}_k(\bar\gamma) \right) & \\ & \hspace{-20ex}= \Exp\left[{\textstyle\prod_{k=\bar\ell}^{L-1}}{\mathds 1}_{{\sf G}_k(\bar\veps_k)}{\mathds 1}_{{\sf H}_k(\bar\gamma)} \left\{ \Exp[{\mathds 1}_{{\sf G}^{\sf n}_L(\bar\veps_L)\cap{\sf G}^{\sf p}_L(\bar\veps_L)}\Exp[{\mathds 1}_{{\sf H}_L(\bar\gamma)} \mid \hat{\mathcal F}_{L-1} ]  \mid {\mathcal F}_{L-1}] \right\} \right] \\ & \hspace{-20ex} \geq (1-2\delta_L)(1-\delta_L) \cdot \Exp\left[{\textstyle\prod_{k=\bar\ell}^{L-1}}{\mathds 1}_{{\sf G}_k(\bar\veps_k)}{\mathds 1}_{{\sf H}_k(\bar\gamma)} \right]  \geq \ldots \geq {\textstyle\prod_{k=\bar\ell}^L} (1-2\delta_k)(1-\delta_k)  \end{align*}
%
%If assumption (E.2) is satisfied, the bounds for the batch sizes $\ngk$ and $\nhk$ can be improved by using the respective ``light-tail'' results in \cref{lemma:tail-bound}.
%
for any $L > \bar \ell$. Hence, taking the limit $L \to \infty$, this yields $\Prob({\sf E}) \geq \delta_*$. We now assume that the trajectories $(x^k)_k$, $(z^k_{\sf n})_k$, and $(z^k_{\sf p})_k$ are generated by a sample point $\bar\omega \in {\sf E}$, i.e., we have $(x^k)_k \equiv (x^k(\bar \omega))_k$ etc. (As we have just shown this happens with probability at least $\delta_*$).

\textit{Step 2: Convergence of $(x^k)_k$}. Let us continue with the proof of the first part. %Due to assumption (B.2), we have $\mathcal R_k \to 0$ and thus, there exists  such that $\min\{\veps,C\Gamma_k\} = C\Gamma_k$ for all $k \geq \tilde \ell$. 
Using the definition of the event ${\sf G}_k$, we can infer
\be \label{eq:bound-cg} \mathcal E_k := \errgk(x^k) \leq \max\{\errgk(z^{k-1}_{\sf n}),\errgk(z^{k-1}_{\sf p})\} \leq \frac{\lamm\Gamma_k}{2} \leq \frac{\lamm \Upsilon_k}{2} \leq \frac{\lamm \veps_k^1}{2} \ee 
for all $k \geq \bar \ell$ %Let ${\mathsf E}_1$ denote the event that the estimate \cref{eq:bound-cg} holds simultaneously \textit{for all} $k \geq \tilde \ell$. Since the sample sets $\cG_k$ are chosen independently, the success probability of the event ${\mathsf E}_1$ is at least $\prod_{k = \tilde \ell}^\infty (1-\delta_k) \geq \sqrt{\delta_*}$. 
and hence, by (B.2), the sequence $(\mathcal E_k)_k$ is summable. Following the proof of \cref{theorem:conv-prox-v2} and utilizing the boundedness of $(\alpha_k)_k$, this implies 
\be \label{eq:prf-flam-conv} F^{\Lambda_*}(x^k) \to 0 \ee
and thus, in this situation every accumulation point of $(x^k)_k$ is a stationary point of problem \cref{eq:prob}. Since the event ${\sf H}_k(\bar\gamma)$ occurs for all $k \geq \bar\ell$ by assumption, we can apply \cref{lemma:invertibility} with $\gamma_c := \gamma_f$ and $\beta_c := \frac{1}{4}$. Hence, there exists a constant $\veps_c$ (that neither depends on the samples $s^k$, $t^k$ nor on $k$) such that the statement in \cref{lemma:invertibility} holds whenever we have $x^k \in B_{\veps_c}(x^*)$ and $\mathcal E_k \leq \veps_c$. Furthermore, since $(\mathcal E_k)_k$ converges to zero, there exists $\tilde \ell \geq \max\{\bar\ell,\bar k\}$ such that the inequality $\mathcal E_k \leq \veps_c$ is satisfied for all $k \geq \tilde \ell$. Setting $\mathcal K_{\veps_c} := \{k \geq \tilde \ell: x^k \in B_{\veps_c}(x^*)\}$, this now implies
%
%\be \label{eq:bound-ch} \errhk(x^k) \leq \frac{\gamma}{3} \leq \frac{\nu_*}{2}, \quad \ewmin(\hfsubk(x^k)) \geq \frac{\nu_*}{2}, \quad \ewmax(\hfsubk(x^k)) \leq 1.5 K_*, \ee 
% 
%hold simultaneously for all $ k \in \mathcal K_{\bar\veps}$. Consequently, applying \cref{lemma:non-str-bound-1}, it follows 
%
\be \label{eq:uni-bound-mk} \| M_k^{-1} \|  \leq \frac{4C}{3}, %\quad  M_k \in  \mathcal M^{\Lambda_*}(x^k) + B_{2\gamma_c}(0), 
\quad \forall~M_k \in \mathcal M_{s^k,t^k}^{\Lambda_*}(x^k), \quad \forall~k \in \mathcal K_{\veps_c}. \ee
%
%and for all $k \in \mathcal K_{\veps_c}$.
%where the matrix $M_k$ is a generalized derivative chosen in step \ref{alg:SSN3} of Algorithm \ref{alg:ssn} and $C$ was defined at the beginning of this subsection. 
Next, let $(x^k)_{k \in K}$ denote an arbitrary subsequence of $(x^k)_k$ converging to $x^*$. Then, there exists $\tilde k \in K$ such that $\{k \in K: k \geq \tilde k\} \subset {\mathcal K}_{\veps_c}$ and
%Let us now set $\tilde {\mathcal K}_{\veps}:= \{k \in \mathcal K_\veps: k \geq \tilde\ell\}$ and let ${\mathsf E}_2$ denote the event that \cref{eq:bound-ch} holds simultaneously \textit{for all} $k \in \tilde {\mathcal K}_{\veps}$. As before, we can utilize the mutual independence of the sample sets $\cG_k$ and $\cH_k$, $k \in \N$, which yields $\Prob({\mathsf E}_2) = \prod_{k \in \tilde {\mathcal K}_{\veps}}^\infty (1-\delta_k)$ and $\Prob({\mathsf E}_1 \cap {\mathsf E}_2) \geq \delta_* $. Let us suppose that the event ${\mathsf E}_1 \cap {\mathsf E}_2$ occurs and 
 %by  and as in the proof of \cref{lemma:mega}, we have
%
%(More generally, this result is true for all $k \in \tilde {\mathcal K}_\veps$). 
consequently, using \cref{eq:prf-flam-conv}, the estimate $\|F^{\Lambda_*}_{s^k}(x) - F^{\Lambda_*}(x)\| \leq \lamm^{-1}\errgk(x)$, and the summability of $(\mathcal E_k)_k$, we obtain
\[ \|x^{k+1} - x^k\| \leq \min\{4C/3,1\} \|F^{\Lambda_*}_{s^k}(x^k)\| \leq \min\{4C/3,1\} [ \|F^{\Lambda_*}(x^k)\| + \lamm^{-1} \mathcal E_k ] \to 0 \]
as $K \ni k \to \infty$. Let us emphasize that this limit behavior holds for any arbitrary subsequence $(x^k)_K$ converging to $x^*$. Moreover, combining the assumptions (D.2) and (D.3) and as mentioned in \cref{subsec:algo}, it follows that $F^{\Lambda_*}$ is semismooth at $x^*$ with respect to the multifunction $\mathcal M^{\Lambda_*}$. Thus, as in \cite[Proposition 3]{PanQi93}, it can be shown that $x^*$ is an isolated stationary point. Since condition \cref{eq:prf-flam-conv} ensures that every accumulation point of $(x^k)_k$ is a stationary point of problem \cref{eq:prob}, this also proves that $x^*$ is an isolated accumulation point of $(x^k)_k$. Hence, a well-known result by Mor\'{e} and Sorensen \cite{MorSor83} yields convergence of the whole sequence $(x^k)_k$ to $x^*$. %As a consequence, there exists $\hat\ell \geq \tilde \ell$ such that $x^k \in B_\veps(x^*)$ for all $k \geq \hat\ell$ and for all $k \geq \hat\ell$ the following inequalities are satisfied
%
%\[ \mathcal E_k \leq \frac{\lamm}{2C} \min \{\veps, Ce_p^k\}, \quad  \mathcal E_{k+1} \leq \frac{\lamm}{2C} \min \{\veps, Ce_p^k\},\quad \mathcal E_{\cH_k}^2(x^k) \leq \frac{\gamma}{3}. \]
%
%As shown in \cref{lemma:mega}, this implies that each Newton step $z^k = x^k - M_k^{-1} F^{\Lambda_*}_{\cG_k}(x^k)$ is well-defined and feasible and it holds
%
%\be \label{eq:prf-conv-rate} \|z^k - x^*\| \leq \eta_\gamma \|x^k - x^*\| + \veps_k^1, \ee
%
%and
%
%\[ \|F^{\Lambda_*}_{\cG_{k+1}}(z^k)\| \leq \eta \|F^{\Lambda_*}_{\cG_k}(x^k)\| + \veps_k^1, \quad \psi(z^k) - \psi(x^k) \leq \beta \|F^{\Lambda_*}_{\cG_k}(x^k)\|^{1-p} \|F^{\Lambda_*}_{\cG_{k+1}}(z^k)\|^p, \]
%
%for all $k \geq \hat \ell$. 

\textit{Step 3: Acceptance of Newton steps}. Utilizing the assumptions (D.2)--(D.4), and (D.6), there exists $\veps_s > 0$ such that the following properties and inequalities hold simultaneously: \vspace{0.5ex}
%Our proof relies on an explicit construction of $\gamma$ and $\veps$ and on utilizing the concentration results in \cref{lemma:mat-conc} (and their implications). We first define the error terms
%
%\[ \mathcal E_\cG(x) :=  \|\nabla f(x) - \nabla f_\cG(x)\|, \quad {\mathcal E}^2_\cH(x) := \|\nabla^2 f(x) - \nabla^2 f_\cH(x)\|. \]
%
%Furthermore, let us set $\delta_\rho := (3\nu_*)^{-1} \gamma$, $\gamma_f := (L_\psi C)^{-1} \min\{\beta,\half\}$, and
%
%\be \label{eq:def-gam}  \gamma_s := \min  \left\{\frac{\eta}{L_F C+\eta}, \gamma_f^{\frac{1}{1-p}}\right\}, \quad \gamma := \min \left \{ \frac{3\nu_*}{2}, \frac{\lamm}{C}\min \left\{\half \gamma_s, \eta_\gamma \right\}  \right \}. \ee
%
\begin{itemize} 
\item Let us set $\bar \beta := (12(1+L\lamm^{-1}))^{-1}$. Then, for all $u \in B_{\veps_s}(u^{\Lambda_*}(x^*))$ and all generalized derivatives $D(u) \in \partial \proxt{\Lambda_*}{r}(u)$ we have \vspace{-0.5ex}
\end{itemize}
\be \label{eq:prf-semism} \|\proxt{\Lambda_*}{r}(u) - \proxt{\Lambda_*}{r}(u^{\Lambda_*}(x^*)) - D(u)(u-u^{\Lambda_*}(x^*))\| \leq \bar\beta\gamma_f \|u-u^{\Lambda_*}(x^*)\|. \ee
\begin{itemize}
\item It holds $\|\nabla f(x) - \nabla f(x^*) - \nabla^2 f(x^*)(x-x^*)\| \leq \frac{\lamm\gamma_f}{12} \|x - x^*\|$ and $\|\nabla^2 f(x) - \nabla^2 f(x^*)\| \leq \frac{\lamm\gamma_f}{12}$ for all $x \in B_{\veps_s}(x^*)$. \vspace{.5ex}
\item The objective function $\psi$ is Lipschitz continuous on $B_{\veps_s}(x^*)$ (with constant $L_\psi$). \vspace{.5ex}
\item The stationary point $x^*$ is a global minimizer of $\psi$ on $B_{\veps_s}(x^*)$. \vspace{.5ex}
\end{itemize}
%
%Since $\proxt{\Lambda_*}{r}$ is semismooth at $u^{\Lambda_*}(x^*)$ there exists $\veps_s > 0$ such that for all $u \in B_{\veps_s}(u^{\Lambda_*}(x^*))$ and all $D(u) \in \partial \proxt{\Lambda_*}{r}(u)$ we have
%
%
%
%Moreover, since the multifunction $\partial \proxt{\Lambda_*}{r} : \Rn \rightrightarrows \R^{n \times n}$ is upper semicontinuous, there exists a constant $\veps_c \leq \bar \veps$ such that
%
%\[ \partial \proxt{\Lambda_*}{r}(y) \subset \partial \proxt{\Lambda_*}{r}(u^{\Lambda_*}(x^*)) + B_{\delta_c}(0), \quad \forall~y \in B_{\veps_c}(u^{\Lambda_*}(x^*)), \]
%
%where $\delta_c := (6(\lamm+K_*))^{-1}\gamma$. 
%Moreover, by (D.4), (D.6) and without loss of generality, we may assume that $\psi$ is Lipschitz continuous on $B_{\veps_s}(x^*)$ and $x^*$ is a global minimizer of $\psi$ on $B_{\veps_s}(x^*)$. And by the differentiability properties of $f$ for all $x \in B_{\veps_s}(x^*)$
%
%\[ \|\nabla f(x) - \nabla f(x^*) - \nabla^2 f(x^*)(x-x^*)\| \leq \frac{\lamm\gamma_f}{12} \|x - x^*\|, \,\, \|\nabla^2 f(x) - \nabla^2 f(x^*)\| \leq \frac{\lamm\gamma_f}{12}. \]
%
Now, let us define 
\[ \veps_e := \min \left\{ \frac{2\lamm\veps_s}{3\lamm+2L}, \veps_c  \right \}, \,\, \veps_\psi := \min\left \{\frac{1}{2}(2L_\psi C^{1-p})^{-\frac{1}{p}}, \frac{1}{2L_F}\right\}, \,\, \veps := \min\left\{\veps_e,\veps_\psi \right\}. \] 
Then, since $(x^k)_k$ converges to $x^*$ and (B.2) implies $\Gamma_k \to 0$, there exists $\hat\ell \geq \tilde \ell$ such that $x^k \in B_\veps(x^*)$ and $\min\{(\max\{C,1\})^{-1}\veps,\Gamma_k\} = \Gamma_k$ for all $k \geq \hat\ell$. Next, let $x^k$, $k \geq \hat\ell$, be an arbitrary iterate and let us consider the associated semismooth Newton step $\nstep = x^k + d^k$, $d^k = - M_k^{+}F^{\Lambda_*}_{s^k}(x^k)$ with $M_k = I-D_k + D_k \Lambda_*^{-1}\hfsubk(x^k)$ and $D_k \in \partial \proxt{\Lambda_*}{r}(u^{\Lambda_*}_{s^k}(x^k))$. %$M(x) \in \mathcal M^{\Lambda_*}_{\cG,\cH}(x)$ with
%Next, let $x \in B_{\veps}(x^*) \backslash \{x^*\}$ be arbitrary and let us consider a semismooth Newton step $z = x + d$, $d = - M(x)^{-1}F^{\Lambda_*}_\cG(x)$, $M(x) \in \mathcal M^{\Lambda_*}_{\cG,\cH}(x)$ with
%
%\[ M_k = I - D_k + D_k \Lambda_*^{-1} \hfsubk(x^k) \quad \text{and appropriate} \quad D_k \in \partial \proxt{\Lambda}{r}(u_{s^k}^{\Lambda_*}(x^k)). \]
% 
By \cref{lemma:invertibility} and \cref{eq:uni-bound-mk}, $M_k$ is invertible with $\|M_k^{-1}\| \leq 2C$ %and there exists $\bar M(x^k) \in \mathcal M^{\Lambda_*}(x^k)$ satisfying $\|M_k - \bar M(x^k)\| \leq \frac{2}{3}\gamma_f$ 
for all $k \geq \hat \ell$. 
%Let us also set $\bar M(x^k) := I - \bar D(x^k) + \bar D(x^k) \Lambda_*^{-1} \nabla^2 f(x^k) \in \mathcal M^{\Lambda_*}(x^k)$ with $\bar D(x^k) \in \partial \proxt{\Lambda_*}{r}(u^{\Lambda_*}(x^k))$. 
Since the events ${\sf G}_k(\lamm\Gamma_k/2)$ and ${\sf G}_{k+1}^{\sf n}(\lamm\Gamma_{k+1}/2)$ occur for all $k \geq \bar \ell$ by assumption, we can reuse the bounds \cref{eq:bound-cg}, i.e., we have 
\be \label{eq:est-G} \max\{\mathcal E_k, \mathcal E_{k+1}^{\sf g}(\nstep)\} \leq \frac{\lamm}{2} \min \left\{\frac{\veps}{\max\{C,1\}}, \Upsilon_k \right\}. \ee %\min \left\{\frac{\veps}{2C}, \min\left\{ \min\left\{\frac{1}{4L_FC},1\right\} \veps_k^1,\mathcal R(\veps_k^2) \right\}\right\}. \ee
%
%and the matrix $M_k$ is invertible with $\|M_k^{-1}\| \leq C$. %This yields
%
%\[ \|u^{\Lambda_*}_{s^k}(x^k) - u^{\Lambda_*}(x^k)\| \leq \lamm^{-1} \mathcal E_k \leq \frac{\veps_c}{2}, \,\, \|u^{\Lambda_*}(x^k) - u^{\Lambda_*}(x^*)\| \leq (1+L\lamm^{-1}) \|x^k-x^*\| \leq \frac{\veps_c}{2} \]
%
%and consequently, we have $D_k, \bar D(x^k) \in \partial \proxt{\Lambda_*}{r}(u^{\Lambda_*}(x^*))+ B_{\delta_c}(0)$ and $\|D_k - \bar D(x^k)\| \leq 2 \delta_c$. Thus, using $\|D_k\Lambda^{-1}_*\| = \|\Lambda_*^{-\half}[\Lambda_*^{\half}D_k\Lambda^{-\half}_*]\Lambda^{-\half}_*\| \leq \|\Lambda^{-1}_*\|$ and (D.2), we obtain
%
%\begin{align*} \|M_k - \bar M(x^k)\| & = \|(\bar D(x^k) - D_k)(I - \Lambda_*^{-1} \nabla^2 f(x^k)) + D_k \Lambda_*^{-1} (\hfsubk(x^k) - \nabla^2 f(x^k))\| \\ & \leq (1+ \lamm^{-1}K_*) \|\bar D(x^k) - D_k\| + \|D_k \Lambda_*^{-1}\| \cdot \errhk(x^k) \leq {2\gamma}(3\lamm)^{-1}. 
%\end{align*}
%
Consequently, setting $w^k := u^{\Lambda_*}_{s^k}(x^k) - u^{\Lambda_*}(x^*)$ and as shown in the proof of \cref{lemma:invertibility}, it holds
\be \label{eq:est-wk} \|w^k\| \leq (1+L\lamm^{-1})\|x^k-x^*\|+\lamm^{-1}\mathcal E_k \leq (1 +L\lamm^{-1})\veps_e + \frac{\veps_e}{2} \leq \veps_s. \ee
Moreover, combining \cref{eq:prf-semism}, $\errhk(x^k) \leq \bar\gamma$, $\bar\beta\gamma_f \leq 1/48$, and the last estimates, we can infer
\begin{align}
\nonumber \|\nstep - x^*\| &= \|M_k^{-1} [F^{\Lambda_*}_{s^k}(x^k) - F^{\Lambda_*}(x^*) - M_k(x^k-x^*)]\| \nonumber \\ 
\nonumber & \leq (4C/3) \|\proxt{\Lambda_*}{r}(u^{\Lambda_*}_{s^k}(x^k)) - \proxt{\Lambda_*}{r}(u^{\Lambda_*}(x^*)) \\
\nonumber &\hspace{6ex}- D_k(I-\Lambda_*^{-1}\hfsubk(x^k))(x^k - x^*)\| \\
\nonumber & \leq (4C/3) \|\proxt{\Lambda_*}{r}(u^{\Lambda_*}(x^*)+w^k) - \proxt{\Lambda_*}{r}(u^{\Lambda_*}(x^*)) - D_k w^k\| \\ 
\nonumber &\hspace{6ex}+ (4C/3) \|D_k\Lambda_*^{-1}[\gfsubk(x^k)-\nabla f(x^*) - \hfsubk(x^k)(x^k-x^*)]\| \\
\nonumber &\leq (4C/3)\bar\beta\gamma_f \|w^k\| + (4C/3)\lamm^{-1} [\|\nabla f(x^k)-\nabla f(x^*) - \nabla^2 f(x^*)(x^k-x^*)\| \\
\nonumber &\hspace{6ex}+ (\|\nabla^2 f(x^k) - \nabla^2 f(x^*)\| + \errhk(x^k))\|x^k-x^*\|)+\mathcal E_k] \\
\nonumber & \leq \frac{4C}{3} \left[\bar\beta (1+L\lamm^{-1}) + \frac{1}{12} + \frac{1}{12} + \frac{1}{2} \right] \gamma_f \|x^k-x^*\| + \frac{4C}{3}\lamm^{-1} \left[\frac{1}{48}+1\right] \mathcal E_k \\
\label{eq:loc-est-1} & \leq C  \gamma_f \|x^k-x^*\| + 1.5C \lamm^{-1}\mathcal E_k. 
\end{align}
%
%where we again used $\|F^{\Lambda_*}_\cG(x) - F^{\Lambda_*}(x)\| \leq \lamm^{-1} \mathcal E_\cG(x)$ (which holds for all $x$ and samples $\cG$). 
Let us note that due to \cref{eq:loc-est-1}, \cref{eq:est-G}, $C\gamma_f \leq \frac{1}{4}$, and (D.4), we also have $\nstep \in B_\veps(x^*) \subset \dom~r$. %and $\|D_x\Lambda^{-1}\| = \|\Lambda^{-\half}[\Lambda^{\half}D_x\Lambda^{-\half}]\Lambda^{-\half}\| \leq \|\Lambda^{-1}\|$. %
The next steps essentially follow the proofs of \cite[Theorem 4.8]{MilUlb14} and \cite[Theorem 4.3.10]{Mil16}. In particular, our last result implies 
\[ \|x^k-x^*\| \leq \|\nstep-x^*\| + \|d^k\|  \leq C \gamma_f  \|x^k-x^*\| + 1.5C\lamm^{-1}\mathcal E_k + (4C/3) \|F^{\Lambda_*}_{s^k}(x^k)\| \]
and thus, it follows
\[ \|x^k - x^*\| \leq \frac{4C}{3(1-C\gamma_f)} \|F^{\Lambda_*}_{s^k}(x^k)\| + \frac{3C\lamm^{-1}}{2(1 - C \gamma_f)} \mathcal E_k. \]
Furthermore, using the Lipschitz continuity of $F^{\Lambda_*}$, $F^{\Lambda_*}(x^*) = 0$, $(1-C\gamma_f)^{-1} \leq \frac{4}{3}$, the definition of $\gamma_f$, and \cref{eq:est-G}, we obtain 
\begin{align*} \| F^{\Lambda_*}_{s^{k+1}}(\nstep) \| & \leq \|F^{\Lambda_*}(\nstep)\| + \|F^{\Lambda_*}_{s^{k+1}}(\nstep) - F^{\Lambda_*}(\nstep)\| \leq L_F \|\nstep - x^*\| + \lamm^{-1} \mathcal E_{k+1}^{\sf g}(\nstep) \\ & \leq  L_F C \gamma_f \|x^k - x^*\| + 1.5L_F C \lamm^{-1} \mathcal E_k + \lamm^{-1} \mathcal E_{k+1}^{\sf g}(\nstep) \\ & \leq \frac{ 4L_F C^2 \gamma_f}{3(1 - C \gamma_f)} \| F^{\Lambda_*}_{s^k}(x^k)\| +  \frac{3L_F C\lamm^{-1}}{2(1 - C\gamma_f)} \mathcal E_k + \lamm^{-1} \mathcal E_{k+1}^{\sf g}(\nstep)  \\ & \leq \eta \|F^{\Lambda_*}_{s^k}(x^k)\| + 2L_FC\lamm^{-1} \mathcal E_k + \lamm^{-1} \mathcal E_{k+1}^{\sf g}(\nstep) \leq \eta \|F^{\Lambda_*}_{s^k}(x^k)\| + \veps_k^1. \end{align*}
Since the Newton step $\nstep$ is contained in $B_{\veps}(x^*)$, we can utilize the semismoothness condition \cref{eq:prf-semism} for $\nstep$. Moreover, \cref{lemma:invertibility} is also applicable for any matrix $M_* = I - \bar D(\nstep) + \bar D(\nstep)\Lambda_*^{-1}\nabla^2 f(x^*)$ with $\bar D(\nstep) \in \partial \proxt{\Lambda_*}{r}(u^{\Lambda_*}(\nstep))$. This yields $\|M_*^{-1}\| \leq (4C)/3$ and similar to the estimates in \cref{eq:loc-est-1} (but simpler), we can get
\begin{align*} \| \nstep - x^*\| & = \|M_*^{-1} [F^{\Lambda_*}(\nstep) - F^{\Lambda_*}(x^*) - M_*(\nstep - x^*) - F^{\Lambda_*}(\nstep)]\| \\ %& \leq \frac{C \gamma_s}{2} \|z - x^*\| + \frac{C}{2} [ \|F^{\Lambda_*}{\tilde\cG}(z)\| + \|F^{\Lambda_*}(z) - F^{\Lambda_*}{\tilde\cG}(z)\| ] \\ 
& \leq  (4C/3)[  \|\proxt{\Lambda_*}{r}(u^{\Lambda_*}(\nstep)) - \proxt{\Lambda_*}{r}(u^{\Lambda_*}(x^*)) - \bar D(\nstep)(I-\Lambda_*^{-1}\nabla^2 f(x^*)) \\ 
& \hspace{4ex} \cdot (\nstep - x^*)\| +  \|F^{\Lambda_*}_{s^{k+1}}(\nstep) \| + {\lamm^{-1}} \mathcal E_{k+1}^{\sf g}(\nstep)] \\
& \leq(4C/3) [ \|\proxt{\Lambda_*}{r}(u^{\Lambda_*}(\nstep)) - \proxt{\Lambda_*}{r}(u^{\Lambda_*}(x^*)) - \bar D(\nstep)\bar w^k\| + \|F^{\Lambda_*}_{s^{k+1}}(\nstep) \|  \\
& \hspace{4ex} +\|\bar D(\nstep)\Lambda^{-1}_* [\nabla f(\nstep) - \nabla f(x^*) - \nabla^2 f(x^*)(\nstep - x^*)]\|   + {\lamm^{-1}} \mathcal E_{k+1}^{\sf g}(\nstep)] \\
& \leq \frac{2}{9} C\gamma_f \|\nstep - x^*\| + \frac{4C}{3} \|F^{\Lambda_*}_{s^{k+1}}(\nstep) \| + \frac{4C}{3} \lamm^{-1} \mathcal E_{k+1}^{\sf g}(\nstep), \end{align*}
where we used $\bar w^k := u^{\Lambda_*}(\nstep) - u^{\Lambda_*}(x^*)$, $\|\bar w^k\| \leq (1+L\lamm^{-1})\|\nstep - x^*\|$, \cref{eq:prf-semism}, and the differentiability of $\nabla f$. This implies 
\[ \|\nstep - x^*\| \leq \frac{12C}{9 - 2C\gamma_f} \|F^{\Lambda_*}_{s^{k+1}}(\nstep)\| + \frac{12C\lamm^{-1}}{9 - 2C\gamma_f} \mathcal E_{k+1}^{\sf g}(\nstep). \] 
Finally, using condition (D.4), $(9-2C\gamma_f)^{-1} \leq \frac{1}{8}$, the subadditivity of the mapping $x\mapsto x^q$, $q \in \{p,1-p\} \subset (0,1)$, $0.5^q + 1.5^q \leq 2$ for all $q \in [0,1]$, and the fact that the stationary point $x^*$ is a global minimum of the problem $\min_ {x \in B_{\veps}(x^*)}~\psi(x)$, it follows 
\begin{align*} \psi(\nstep) - \psi(x^k) & \\ &\hspace{-12ex}\leq \psi(\nstep) - \psi(x^*) \leq L_\psi \|\nstep - x^*\| = L_\psi \|\nstep - x^* \|^{1-p}  \|\nstep - x^*\|^{p} \\ &\hspace{-12ex}\leq  L_\psi C^{1-p} [\gamma_f^{1-p} \|x^k - x^*\|^{1-p} + (1.5\lamm^{-1}\mathcal E_k)^{1-p}] \|\nstep-x^*\|^p \\ &\hspace{-12ex}\leq L_\psi C^{1-p} \left[ (2C\gamma_f \|F^{\Lambda_*}_{s^k}(x^k)\|)^{1-p} + ((2C\gamma_f)^{1-p}+1.5^{1-p})(\lamm^{-1} \mathcal E_k)^{1-p} \right] \|\nstep-x^*\|^p \\ &\hspace{-12ex} \leq L_\psi (2C^2\gamma_f)^{1-p} \|F^{\Lambda_*}_{s^k}(x^k)\|^{1-p} [ (1.5C)^p \|F^{\Lambda_*}_{s^{k+1}}(\nstep)\|^p + (1.5C\lamm^{-1} \mathcal E_{k+1}^{\sf g}(\nstep))^{p} ] \\ & \hspace{-10ex} + 2L_\psi(C\lamm^{-1})^{1-p}\mathcal E_k^{1-p} \|\nstep-x^*\|^p \\ & \hspace{-12ex} \leq 2L_\psi C (C\gamma_f)^{1-p} \|F^{\Lambda_*}_{s^k}(x^k)\|^{1-p} \|F^{\Lambda_*}_{s^{k+1}}(\nstep)\|^p + 0.5{\veps^2_k}\\ & \hspace{-10ex} + 2L_\psi C(C \gamma_f)^{1-p}(L_F \|x^k -x^*\| + \lamm^{-1} \mathcal E_k)^{1-p} \lamm^{-p} [\mathcal E_{k+1}^{\sf g}(\nstep)]^p \\ & \hspace{-12ex} \leq  \beta \|F^{\Lambda_*}_{s^k}(x^k)\|^{1-p} \|F^{\Lambda_*}_{s^{k+1}}(\nstep)\|^p + \frac{\veps^2_k}{2} + \half (2^{p-1} + 2^{p-1}\mu_p(\veps^2_k)^{1-p}) 2^{-p} \mu_p(\veps^2_k)^p \\ & \hspace{-12ex} \leq \beta \|F^{\Lambda_*}_{s^k}(x^k)\|^{1-p} \|F^{\Lambda_*}_{s^{k+1}}(\nstep)\|^p + \veps^2_k . \end{align*}
%
%This proves part (iii). The first part essentially follows from \cref{theorem:inv} and from $B_{\bar\veps}(x^*) \subset \dom~r$. 
%Since the estimate \cref{eq:loc-est-1} and its derivation only depend on the bounds for the sample sets $\cG$ and $\cH$, the remaining second part can be established by noticing $C\lamm^{-1}\gamma \leq \eta_\gamma$ and $C\lamm^{-1} \mathcal E_\cG(x) \leq 0.5 \min\{\veps,(2L_F)^{-1}\veps_1 \} \leq 0.5 \min\{\veps,\veps_1 \}$.

\textit{Step 4: Transition to fast local convergence}. Let us point out that the bounds and inequalities derived in the last part primarily depend on the occurrence of the events ${\sf G}_k$ and ${\sf H}_k$ and hold for any $k \geq \hat\ell$. Now, let $k$ be any index $k \geq \hat \ell$ with $ \|F^{\Lambda_*}_{s^k}(x^k)\| \leq \theta_k$. Since the algorithm does not terminate after a finite number of steps and we have $F^{\Lambda_*}_{s^k}(x^k) \to 0$, there exist infinitely many such indices. Let $\ell_* - 1$ be the smallest such index. Then, as shown in step 3, the semismooth Newton step $z^{\ell_*-1}_{\sf n}$ satisfies all of the acceptance criterions and it follows $x^{\ell_*} = z^{\ell_*-1}_{\sf n}$ and $\theta_{\ell_*} = \|F^{\Lambda_*}_{s^{\ell_*}}(x^{\ell_*})\|$. Inductively, this implies $x^{k+1} = \nstep$ for all $k \geq \ell_* - 1$. Since the success probability of the event ${\sf E}$ is at least $\delta_*$, this finishes the proof of the first part. 

\textit{Step 5: Proof of part (ii)}. Revisiting the derivations in step 1, it is easy to see that the event ${\sf E}^\circ := \bigcap_{k=\bar\ell}^\infty {\sf G}_k(\lamm\Gamma_k^\circ/2) \cap {\sf H}_k(\lamm\gamma_f/2)$ also occurs with probability $\delta_*$. Moreover, due to ${\sf E}^\circ \subset {\sf E}$, all of the results and inequalities shown in step 2--4 remain valid. %By construction, the bound $(\diamond_k)$ does also hold for all $k \geq \bar \ell$. Hence, we can reuse the same inequalities, properties, and implications that were derived in part (i). 
In particular, the estimate \cref{eq:loc-est-1} holds for all $k \geq \ell_* - 1$ with $x^{k+1} = \nstep$ and using $L_F \geq 1$, we obtain
\[ \|x^{k+1} - x^*\| \leq C\gamma_f \|x^k-x^*\| + \frac{3}{4}C \Upsilon_k^\circ \leq \half \left [ \half \|x^k - x^*\| + \gamma_\eta^{k-\bar\ell} \right]. \]
Let us set $\bar\tau :=\max\{\|x^{\ell_*-1}-x^*\|,\gamma_\eta^{\ell_*-\bar\ell-2}\} $ and $\tau_k := \bar\tau \max \{\half,\gamma_\eta\}^{k-\ell_*+1}$. We now prove $\|x^k - x^*\| \leq \tau_k$ for all $k \geq \ell_*-1$ by induction. For $k = \ell_*-1$, it follows $\tau_{\ell_*-1} = \bar\tau$ and the latter inequality is obviously true. Furthermore, it holds
\begin{align*} \|x^{k+1}-x^*\| & \leq {\textstyle\half}\max\{{\textstyle\half},\gamma_\eta\} [ \tau_k + \gamma_\eta^{k-\bar \ell-1}] \\ & \leq {\textstyle\half}\max\{{\textstyle\half},\gamma_\eta\}^{k-\ell_*+2} [\bar\tau + \gamma_\eta^{\ell_*-\bar\ell-2}] \leq \tau_{k+1}.  \end{align*}
Consequently, due to $\tau_k \to 0$ and $\tau_{k+1}/\tau_k = \max\{\half,\gamma_\eta\} < 1$,
%
%\[ \frac{\tau_{k+1}}{\tau_k} = \left( 1 + \frac{1}{k-\ell_*+2}\right) \max\left\{\half,\gamma_\eta\right\} \to \max\left\{\half,\gamma_\eta\right\}, \quad k \to \infty, \]
%
the sequence $(x^k)_k$ converges r-linearly to $x^*$ with rate $\max\{\half,\gamma_\eta\}$.

\textit{Step 6: Proof of part (iii)}. Again, following our previous discussions, it can be easily shown that the event ${\sf E}^\diamond := \bigcap_{k=\bar\ell}^\infty {\sf G}_k(\lamm\Gamma_k^\diamond/2) \cap {\sf H}_k(\bar\rho\sqrt{\kappa_n\rho_k})$ occurs with probability $\delta_*$. Since $(\rho_k)_k$ converges to zero, we have ${\sf H}_k(\bar\rho\sqrt{\kappa_n\rho_k}) \subset {\sf H}_k(\lamm\gamma_f/2)$ for all $k$ sufficiently large and hence, the results derived in step 2--4 still hold after possibly adjusting the constant $\ell_*$. In particular, we have $x^k \to x^*$ as $k\to \infty$. Next, let us set
\begin{align*} \beta^1_k & = \frac{\|\proxt{\Lambda_*}{r}(u^{\Lambda_*}(x^*)+w^k) - \proxt{\Lambda_*}{r}(u^{\Lambda_*}(x^*)) - D_k w^k\|}{\|w^k\|}, \\ \beta_k^2 & = \frac{\|\nabla f(x^k) - \nabla  f(x^*) - \nabla^2 f(x^*)(x^k-x^*)\|}{\|x^k-x^*\|},  \end{align*}
and $\beta_k^3 = \|\nabla^2 f(x^k) - \nabla^2 f(x^*)\| + \errhk(x^k)$, where $w^k$ and $D_k$ have been defined in step 3. Since $(w^k)_k$ converges to zero, the semismoothness of $\proxt{\Lambda_*}{r}$ implies $\beta_k^1 \to 0$. Moreover, by the differentiability of $f$ and using $\rho_k \to 0$, we obtain $\beta^2_k \to 0$ and $\beta_k^3 \to 0$.
%
%\begin{align*} \beta_k^2 & = \|(\bar D(x^k) - D_k)(I - \Lambda_*^{-1} \nabla^2 f(x^k)) + D_k \Lambda_*^{-1} (\hfsubk(x^k) - \nabla^2 f(x^k))\| \\ & \leq (1+ \lamm^{-1}K_*^\diamond) \|\bar D(x^k) - D_k\| + \|D_k \Lambda_*^{-1}\| \cdot \errhk(x^k) \\ & \leq (1+ \lamm^{-1}K_*^\diamond) \|\bar D(x^k) - D_k\| + \lamm^{-1}\bar\rho\sqrt{\kappa_n\rho_k}
%\end{align*}
%
%for all $k \geq \bar \ell$. Since the sequences $(u^{\Lambda_*}_{s^k}(x^k))_k$ and $(u^{\Lambda_*}_{s^k}(x^k))_k$ both converge to $u^{\Lambda_*}(x^*)$ and $\partial \proxt{\Lambda_*}{r}(u^{\Lambda_*}(x^*))$ is a singleton, the upper semicontinuity and local compactness of the multifunction $x \mapsto \partial \proxt{\Lambda_*}{r}(x)$ implies $\|\bar D(x^k) - D_k \| \to 0$ and thus, we can strengthen the result in \cref{lemma:invertibility} (ii) to $\beta_k^2 \to 0$. 
Without loss of generality, we now may assume $x^{k+1} = \nstep$ for all $k \geq \ell_* - 1$ and as a consequence, by \cref{eq:est-wk} and \cref{eq:loc-est-1} and defining $\vartheta_k := (4C/3)\lamm^{-1}[(\lamm+L)\beta_k^1 + \beta_k^2 + \beta_k^3]$, we have
\begin{align*} \|x^{k+1} - x^*\| &\leq ({4C}/3)\lamm^{-1} [(\lamm+L)\beta_k^1 + \beta_k^2+\beta_k^3] \|x^k - x^*\| + ({4C}/3)\lamm^{-1}  (\bar\beta\gamma_f+1) \mathcal E_k \\ & \leq \vartheta_k \|x^k - x^*\| + (3C/4)\Upsilon_k^\diamond \leq  \vartheta_k \|x^k - x^*\| + \gamma_k^{k-\bar\ell} \end{align*}
for all $k \geq \ell_*-1$. Next, due to $\vartheta_k,\gamma_k \to 0$, there exists a constant $\ell_\diamond \geq \ell_*$ such that $\vartheta_k, \gamma_k \leq \half$ for all $k \geq \ell_\diamond$. Let us set
\[ \tau_{\ell_\diamond+1} := \max\left\{\|x^{\ell_\diamond}-x^*\|,\gamma_{\ell_\diamond}^{(\ell_\diamond-\bar\ell)/2}\right\}, \,\, \tau_{k+1} := \max\left\{(\vartheta_k+\gamma_k^{(k-\bar\ell)/2})\tau_k,\gamma_{k+1}^{(k+1-\bar\ell)/2}\right\} \] 
for all $k > \ell_\diamond$. Then, by induction and by using $\tau_k \geq \gamma_k^{(k-\bar\ell)/2}$, it follows 
\[ \|x^{k+1}-x^*\| \leq \vartheta_k \|x^k-x^*\| + \gamma_k^{(k-\bar\ell)/2}\tau_k \leq (\vartheta_k + \gamma_k^{(k-\bar\ell)/2}) \tau_k \leq \tau_{k+1} \]
for all $k > \ell_\diamond$. Due to $\vartheta_k,\gamma_k \leq \half$, this also establishes $\|x^{\ell_\diamond+1}-x^*\| \leq \tau_{\ell_{\diamond}+1}$. Finally, utilizing the boundedness of $(\tau_k)_{k>{\ell_\diamond}}$ and the monotonicity of $(\gamma_k)_k$, we obtain $\tau_k \to 0$ and
\[ \frac{\tau_{k+1}}{\tau_k} = \max\left\{ \vartheta_k + \gamma_k^{\frac{k-\bar\ell}{2}}, \gamma_{k+1}^{\half} \left[\frac{\gamma_{k+1}}{\gamma_k}\right]^{{(k-\bar\ell)}/{2}} \right\} \leq \max\left\{ \vartheta_k + \gamma_k^{\frac{k-\bar\ell}{2}}, \gamma_{k+1}^{\half} \right\} \to 0, \]
which concludes the proof of \cref{theorem:mega}. 
\end{proof}

Neglecting the dependence on $\delta_k$ and on $\veps_k^1$, $\veps_k^2$ for a moment, the results in \cref{theorem:mega} can be summarized as follows. In order to guarantee r-linear convergence of the sequence $(x^k)_k$, it suffices to increase the sample size $\ngk$ at a geometric rate and to choose $\nhk$ sufficiently large. If $\ngk$ is increased at a rate that is faster than geometric and we have $\nhk \to \infty$, then we obtain r-superlinear convergence with high probability. Similar results were established in \cite{BBN2016,RKM2016II,YeLuoZha17} for stochastic Newton-type methods for smooth optimization problems. Clearly, the rate of convergence in \cref{theorem:mega} (ii) can be further improved if $\gamma_f$ is adjusted appropriately. Moreover, if the full gradient is used in the algorithm eventually, i.e., we have $\gfsubk(x^k) \equiv \nabla f(x^k)$ for all $k$ sufficiently large, then the gradient related error terms $\mathcal E_k^{\sf g}$ vanish and we can derive q-linear and q-superlinear convergence, respectively. We continue with an additional remark.

\begin{remark} \label{remark:proj-loc} Similar to \cref{remark:example-c3}, let us again assume that the function $r$ has the special form $r = \iota_{\mathcal C} + \varphi$, where $\varphi$ is a real-valued, convex mapping and $\iota_{\mathcal C}$ is the indicator function of a closed, convex, and nonempty set $\mathcal C$. In this case, as has already been discussed in \cref{sec:algo}, we can perform a projected Newton step $z^k = \mathcal P_{\dom~r}(x^k+d^k) = \mathcal P_{\mathcal C}(x^k+d^k)$ to obtain a feasible trial point $z^k \in \dom~r$. Due to the nonexpansiveness of the projection, this additional operation does also not affect our local convergence results. Moreover, since $\varphi$ is locally Lipschitz continuous, assumption {\rm(D.4)} is no longer needed in this situation. \end{remark}

Based on the ``light tail'' assumption (E.3), we now present a straight-forward variant of \cref{theorem:mega} with an improved dependence on the probabilities $\delta_k$.

\begin{corollary} \label{corollary:mega} Consider the setup discussed in \cref{theorem:mega} and let us assume that the conditions in \cref{theorem:mega} are fulfilled. Suppose that assumption {\rm(E.3)} is satisfied. Then, the statements in \cref{theorem:mega} {\rm(i)} hold under the following improved sample size bounds
\be \label{eq:bounds-iv} \ngk \geq \left[\left(1+\sqrt{3\log(\delta_k^{-1}})\right)\frac{2\bar\sigma}{\lamm\Gamma_k}\right]^2, \quad \nhk \geq 3\log(2n\delta_k^{-1})\left[ \frac{2\bar\rho}{\lamm\gamma_f} \right]^2, \quad k \geq \bar \ell.  \ee
Furthermore, if \cref{eq:bounds-iv} holds with $\Gamma_k \equiv \Gamma_k^\circ$ or with $\Gamma_k \equiv \Gamma_k^\diamond$ and $\nhk \geq \log(2n\delta_k^{-1})\rho_k^{-1}$, then the statements in \cref{theorem:mega} {\rm(ii)} and {\rm(iii)} are satisfied, %hold under similar adjustments of the bounds \cref{eq:bounds-ii} and \cref{eq:bounds-iii}. Specifically, in \cref{eq:bounds-iv} we can set $\Gamma_k \equiv \Gamma_k^\circ$, $\Gamma_k \equiv \Gamma_k^\diamond$ and use the bound $\nhk \geq \log(2n\delta_k^{-1})\rho_k^{-1}$
respectively. 
\end{corollary}

\cref{corollary:mega} can be shown directly by applying \cref{lemma:tail-bound}. Finally, let us note that our local results can be further improved in the situation that was considered in \cref{theorem:conv-prox-strong}. We conclude this section with an example and discuss a specific choice of the parameters and sample sizes satisfying the assumptions in \cref{theorem:mega} and \cref{corollary:mega}.

\begin{example} Suppose that the conditions in \cref{corollary:mega} are satisfied and let $C_1, C_2 > 0$, and $\varpi > 0$ be given. Let the parameter sequences $(\veps_k^1)_k$, $(\veps_k^2)_k$ be chosen via
\[ \veps_1^k = {C_1}{k^{-(2+\frac{\varpi}{4})}}, \quad \veps_2^k = {C_2}{k^{-(1+\frac{\varpi}{8})}}, \quad \forall~k \in \N, \]
and let us set $p = \half$ and $\delta_k = \frac{1}{2k^{8}}$. Then, setting $\ngk =k^{4+\varpi}\log(k)$ and $\nhk = \log(k)^{1+\varpi}$, it can be shown that the statements in \cref{theorem:mega} {\rm(i)} hold with probability %$\delta_* \geq \left[\prod_{k=2}^\infty \left(1-\frac{1}{k^{8}}\right)\right]^2 \geq 0.99$,
\[ \textstyle\delta_* \geq \left[\prod_{k=2}^\infty \left(1-\frac{1}{k^{8}}\right)\right]^2 = \left[\frac{\sinh(\pi)(\cosh(\pi\sqrt{2})-\cos(\pi\sqrt{2}))}{16\pi^3}\right]^2 \geq 0.99, \]
see \cite{Wei17}. Additionally, if the gradient sample size increases geometrically, i.e., if we have $\ngk = C_3 \ell^{k}$ for some $\ell > 1$, $C_3 > 0$, the conditions in \cref{theorem:mega} {\rm(ii)} are satisfied and we can guarantee r-linear convergence with rate $\max\{\half,\frac{1}{\ell}\}$ with probability $99\%$.
\end{example}

%----------------------------------------------------------------------------------------------------------------%
% NUMERICAL RESULTS
%----------------------------------------------------------------------------------------------------------------%

\section{Numerical Results}\label{sec:numerical} 
In this section, we demonstrate the efficiency of the proposed stochastic semismooth Newton framework and compare it with several state-of-the-art algorithms on a variety of test problems. All numerical experiments are performed using MATLAB R2017b on a desktop computer with Intel(R) Core(TM) i7-7700T 2.90GHz and 8GB memory.

\subsection{Logistic Regression} \label{sec:logloss}
%The $\ell_1$-regularized stochastic logistic regression problem \cite{Xiao10} is in the form of
%
%\be\label{eq:SLR}
%\min_{x\in\R^n}~\psi(x):=\Exp_{a,b}[\log(1+\exp(-ba^\top x))]+\lambda\|x\|_1,
%\ee
%
%where the expectation $\Exp$ is taken with respect to random vector $(a,b)\in\R^n\times\{-1,1\}$ with an (unknown) distribution $\mathcal{D}$, $F(x;a,b)=\log(1+\exp(-ba^\top x))$ is the well-known logistic loss function, and $\lambda>0$ is a regularization parameter. 
In our first experiment, we consider the well-known empirical $\ell_1$-logistic regression problem
%we test algorithms for solving the following  empirical  problem
%
\be\label{eq:LR}
\min_{x\in\R^n}~\psi(x) := f(x)+\mu\|x\|_1,\quad f(x):=\frac{1}{N}\sum_{i=1}^Nf_i(x),
\ee
where $f_i(x):=\log(1+\exp(-b_i \cdot \iprod{a_i}{x}))$ denotes the logistic loss function and the data pairs $(a_i,b_i) \in \Rn \times \{-1,1\}$, $i \in [N]$, correspond to a given dataset or are drawn from a given distribution. % $\{(a_1,b_1),\ldots,(a_N,b_N)\}$ is some given dataset which can be seen as a set of observations drawn from the distribution $\mathcal{D}$, 
The regularization parameter  $\mu > 0$ controls the level of sparsity of a solution of problem \cref{eq:LR}. In our numerical tests, we always choose $\mu=0.01$. 
%\revise{References sparsity, other methods?}  Is there any reference for this choice? Usually a very small parameter $\mu = 10^{-4},10^{-5}$ yields the best prediction rates. Can this be explained by scaling?   %Let $A\in\R^{N\times n}$ be the matrix with rows being formed from the vectors $a_i, i=1,\ldots,N$. The dominant cost of computing $\nabla f(x),\nabla^2f(x)$ and their sub-sampled approximations is the evaluation of matrix-vector product $Ax$ and $A^\top y$ (for some $y\in\R^N$). 

\subsubsection{Algorithmic details and implementation} 
Next, we describe the implementational details of our method and of the state-of-the-art algorithms used in our numerical comparison. 

\textit{Stochastic oracles.} In each iteration and similar to other stochastic second order methods, \cite{BCNN2011,BBN2016,BHNS2016,RKM2016I}, we generate stochastic approximations of the gradient and Hessian of $f$ via first selecting two sub-samples $\cS_k, \mathcal T_k \subset [N]$ uniformly at random and without replacement from the index set $\{1,...,N\}$. We then define the following mini-batch-type stochastic oracles
\be \label{eq:exp-oracle} \gfsubk(x) := \frac{1}{|\cS_k|} \sum_{i \in \cS_k} \nabla f_i(x), \quad \hfsubk(x) := \frac{1}{|\mathcal T_k|} \sum_{j \in \mathcal T_k} \nabla^2 f_j(x). \ee
%
%$\cS_k$ and $\mathcal T_k$ are chosen 
% \comm{How are $\cS_k$ and $\mathcal T_k$ generated? With or without replacement? Strategy?}
We refer to the variant of Algorithm \ref{alg:ssn} using the stochastic oracles \cref{eq:exp-oracle} as S4N (\textbf{s}ub-\textbf{s}ampled \textbf{s}emi\textbf{s}mooth \textbf{N}ewton method). Furthermore, motivated by the recent success of variance reduction techniques \cite{JZ2013, XZ2014, RHSPS2016, AZH2016, WMGL2017}, we will also work with a variance reduced stochastic gradient that can be calculated as follows
\be \label{eq:var-oracle} \begin{cases} {\scriptsize \textbf{1}} \quad \textbf{if} \quad {k}~\mathrm{mod}~{m} = 0 \quad \textbf{then} \quad \text{set $\tilde x := x^k$ and calculate $\tilde u := \nabla f(\tilde x)$.} \\ {\scriptsize \textbf{2}} \quad \text{Compute $\gfsubk(x^k) := \frac{1}{|\cS_k|} \sum_{i \in \cS_k} (\nabla f_i(x^k) - \nabla f_i(\tilde x)) + \tilde u$.} \end{cases} \ee
Here, $k \in \N$ is the current iteration and $m \in \N$ denotes the number of iterations after which the full gradient $\nabla f$ is evaluated at the auxiliary variable $\tilde x$. As in \cite{RHSPS2016, AZH2016, WMGL2017}, this additional noise-free information is stored and utilized in the computation of the stochastic oracles for the following iterations.  %version of the variance reduction technique 

\textit{Overview of the tested methods.} 

\begin{itemize}
%\item \textbf{SPG}. The stochastic proximal gradient method  with increasing sample size \cite{AtcForMou17}. 
\item \textbf{Adagrad}, \cite{DHS2011}. Adagrad is a stochastic proximal gradient method with a specific strategy for choosing the matrices $\Lambda_k$. We use the mini-batch gradient \cref{eq:exp-oracle} as first order oracle in our implementation. This leads to the following update rule
\end{itemize}
\be \label{eq:adagrad-update} x^{k+1} = \proxt{\Lambda_k}{\vp}(x^k - \Lambda_k^{-1} \gfsubk(x^k)), \quad \Lambda_k := \lambda^{-1} \diag(\delta\mathds 1 + \sqrt{G_k}), \ee
\begin{itemize}
\item[] where $\delta, \lambda > 0$, $G_k := G_{k-1} + \gfsubk(x^k) \odot \gfsubk(x^k)$, and the multiplication ``$\odot$'' and the square root ``$\sqrt{\,\cdot\,}$'' are performed component-wise.
\item \textbf{prox-SVRG}, \cite{XZ2014}. Prox-SVRG is a variance reduced, stochastic proximal gradient method. Similar to  \cite{RHSPS2016,RSPS2016,WZ2017}, we substitute the basic variance reduction technique proposed in \cite{JZ2013,XZ2014} with the mini-batch version \cref{eq:var-oracle} to improve the performance of prox-SVRG. \vspace{.5ex}
\item \textbf{S2N-D}. S2N-D is the deterministic version of the stochastic semismooth Newton method using the full gradient and Hessian of $f$ instead of stochastic oracles. \vspace{.5ex}
\item \textbf{S4N-HG}. S4N with both sub-sampled gradient and Hessian \cref{eq:exp-oracle}. In the numerical experiments, the maximum sample size $|\cS_k|$ of the stochastic oracle $G_{s^k}$ is limited to 10\%, 50\% and 100\% of the training data size $N$, respectively. \vspace{.5ex}
\item \textbf{S4N-H}. This version of S4N uses the full gradient $\nabla f$ and the sub-sampled Hessian $H_{t^k}$ as defined in \cref{eq:exp-oracle}. \vspace{.5ex}
\item \textbf{S4N-VR}. S4N-VR is a variant of S4N combining the variance reduced stochastic oracle \cref{eq:var-oracle} with the basic sub-sampling strategy \cref{eq:exp-oracle} for the Hessian of $f$. \vspace{.5ex}
\end{itemize}
Adagrad and prox-SVRG are two popular and efficient first order stochastic optimization approaches for solving nonsmooth and possibly nonconvex problems of the form \cref{eq:prob}. We compare them with four different versions of our S4N method: S2N-D (deterministic), S4N-HG (sub-sampled gradient and Hessian), S4N-H (full gradient and sub-sampled Hessian), and S4N-VR (variance reduced stochastic gradient and sub-sampled Hessian). 

%For SPG and Adagrad, in experiments we increase the sample size of gradient gradually by a factor $\alpha=1.5$ for every 50 iterations.
\textit{Implementational details.} For Adagrad, the sample size $|\cS_k|$ of the stochastic gradient is fixed to 5\% of the training data size $N$ and  we set $\delta = 10^{-7}$. The parameter $\lambda$ varies for the different tested datasets and is chosen from the set $\{i\cdot10^j : i \in [9], j \in \{-2,-1,0,1\}\}$ to guarantee optimal performance. The iterative scheme of prox-SVRG basically coincides with \cref{eq:adagrad-update}. Here, we also use a fixed sample size $|\cS_k| = \lfloor 0.01 N \rfloor$ and we set $m = 10$. The parameter matrix $\Lambda_k$ is defined via $\Lambda_k := (1/\lambda_k)I$ and based on the full gradient values $\nabla f(\tilde x)$, $\lambda_k$ is chosen adaptively to approximate the Lipschitz constant of the gradient $\nabla f$. %As is suggested in \cite{JZ2013}, the step length of SVRG method does not have to decay and a fixed relative large step length will lead to faster convergence in practice. and prox-SVRG, and 1\%
% \comm{? ... usually the step-size is of the form $c/k$ - how does this guarantee convergence. Is this a uniform step size for all steps?} 

In S4N-HG, the initial sample size of the stochastic gradient is set to $|\cS_0| = \lfloor 0.01 N \rfloor$. The size of the mini-batch $\cS_k$ is then increased by a factor of 3.375 every 30 iterations until $|\cS_k|$ reaches the maximum sizes $\lfloor 0.1 N \rfloor$, $\lfloor 0.5 N \rfloor$, and $N$, respectively. In the following, we will use S4N-HG 10\%, S4N-HG 50\%, and S4N-HG 100\% to denote the different variants of S4N-HG. In S4N-VR, we use the fixed sample size $|\cS_k| = \lfloor 0.01 N \rfloor$ for all $k$ and $m = 6$. The mini-batch sizes of the stochastic Hessians are adjusted in a similar way. More specifically, in S4N-HG, S4N-H, and S4N-VR, we first set $|\mathcal T_0| = \lfloor 0.01 N \rfloor$. As soon as the sample $\cS_k$ reaches its maximum size, we repeatedly increase the size of the set $\mathcal T_k$ by a factor of 3.375 after 15 iterations. The upper limit of the Hessian sample size is set to 10\% of the training data size, i.e., we have 
\[ |\mathcal T_k| \leq t_{\max}, \quad t_{\max} := \lfloor 0.1 N \rfloor, \quad \forall~k. \] 
In S4N-HG 10\% and different from the other methods, the size of $\mathcal T_k$ is not changed, i.e., it holds $t_{\max} = \lfloor 0.01 N \rfloor$.
%The sample sizes of Hessian matrices for S4N-H, S4N-HG and S4N-VR are fixed to $10\%\cdot N$. 
% \comm{The sample size keeps increasing from 0.5\% to 2\% and now we have 10\% - which is rather large! Why?}  
%All methods, if not specified, start with an initial sub-sampled gradient with size of 1\% of the data. In S4N-HG method, the sample size of gradient is increased by a factor of $. 
As in prox-SRVG, we use $\Lambda_k := (1/\lambda_k)I$ and choose $\lambda_k$ adaptively to estimate the Lipschitz constant of the gradient. In particular, we compute 
\[ \lambda_k^1 = \frac{\|x^k - x^{k-1}\|}{\|G_{s^k}(x^k) - G_{s^{k-1}}(x^{k-1})\|}, \quad \lambda_k^2 = \max\{10^{-3},\min\{10^4,\lambda_k^1\}\}. \] 
In order to prevent outliers, we calculate a weighted mean of $\lambda_k^2$ and of the previous parameters $\lambda_j$, $j \in [k-1]$. This mean is then used as the new step size parameter $\lambda_k$. The initial step size is set to $\lambda_0 = 0.1$. %For S4N-VR method, we set $m$ between $5\sim10$ and sample size of gradient to $1\%$ of $N$.

The proximity operator of the $\ell_1$-norm has the explicit representation $\proxt{\Lambda_k}{\mu\|\cdot\|_1}(u) = u - \mathcal P_{[-\mu\lambda_k,\mu\lambda_k]}(u)$ and is also known as the shrinkage operator or soft-thresholding function. Similar to \cite{MilUlb14,PatSteBem14,XLWZ2016}, we will work with the following generalized Jacobian of $\proxt{\Lambda_k}{\mu\|\cdot\|_1}$ at some $u \in \Rn$  %matrix by chain rule (), %for example, when $F$ is defined as
%
%\[ F^\Lambda(x) := x - \proxt{\Lambda}{r}(x- \Lambda^{-1} \nabla f(x)) \]
%
%where $r = \lambda\|\cdot\|_1$ and $\Lambda = 1/\mu \cdot I$, a generalized Jacobian matrix is taken as follows:
%
%\[ M = I - D(I-\Lambda^{-1}\nabla^2f(x)) \]
%
%where $D(z) \in \partial \proxt{\Lambda}{r}$ is a diagonal matrix with diagonal entries
%
\[ D(u) := \diag(d(u)), \quad d(u) \in \Rn, \quad d(u)_i := \begin{cases} 1 & |u_i| > \mu\lambda_k, \\ 0 & \text{otherwise.} \end{cases} \]
The generalized derivatives of $\Fsubk$ are then built as in \cref{eq:gen-deriv}. As described in \cite{MilUlb14,PatSteBem14,XLWZ2016}, we can exploit the structure of the resulting semismooth Newton system and reduce it to a smaller and symmetric linear system of equations. We utilize an early terminated conjugate gradient (CG) method to solve this system approximately. The maximum number of iterations and the desired accuracy of the CG method are adjusted adaptively depending on the computed residual $\|\Fsubk(x^k)\|$. When the residual is large, only few iterations are performed to save time. The initial relative tolerance and the initial maximum number of iterations are set to 
 0.01 and 2, respectively. The total maximum number of CG-iterations is restricted to 12.   
% \comm{More details: the Newton system is reduced to a smaller and symmetric (!) linear system of equations; what kind of generalized derivatives are used? Procedure similar to \cite{MilUlb14,PatSteBem14,XLWZ2016}}.
% \comm{5,10,20? What's the maximum number of CG-iterations at the beginning?} 
% When the generalized derivative $M_k$ is ill-conditioned \comm{how do we find out whether $M_k$ is ill-conditioned in the algorithm??}, 
In order to numerically robustify the computation of the Newton step $\nstep$, we also consider the following, regularized version of the Newton system
\[\left(M_k + \rho_k I\right) \cdot d^k = -\Fsubk(x^k), \quad M_k \in \dsetMk(x^k), \]
where $\rho_k>0$ is a small positive number. We adjust $\rho_k$ according to the norm of the residual $\Fsubk(x^k)$ so that $\rho_k \to 0$ as $\|\Fsubk(x^k)\| \to 0.$
% \comm{Choice of $\rho_k$? This influences local convergence.}

Finally, in our implementation of S4N, we only check the first growth condition \cref{eq:growth-1} to measure the quality of the Newton step $\nstep$. Although both growth conditions \cref{eq:growth-1} and \cref{eq:growth-2} are generally required to guarantee global convergence, this adjustment does not affect the globalization process and convergence of S4N in the numerical experiments. Moreover, as we have shown in \cref{theorem:conv-prox-strong} and in \cref{theorem:mega}, the condition \cref{eq:growth-2} is actually not necessary for strongly convex problems and it is satisfied locally close to a stationary point of problem \cref{eq:prob} under certain assumptions. These different observations motivate us to restrict the acceptance test of the semismooth Newton steps to the cheaper condition \cref{eq:growth-1}. We use the following parameters  $\eta = 0.85$, $\nu_k =\veps_k^1 = c_\nu k^{-1.1}$, $c_\nu = 500$, and $\alpha_k = 10^{-2}$. 

%we just ignore the growth conditions and persist Newton steps. 
%In our numerical experiments 
%In the theoretical analysis of S4N, to ensure global convergence, it is advised to check growth conditions \cref{eq:growth-1} and \cref{eq:growth-2} in every iteration and a proximal gradient step is performed when the conditions fail. However, in our numerical experiments, we found that the overall performance of S4N is not affected much if the globalization process is taken away. 
%\revise{
% We provide a simple test to show this phenomena in \cref{figure:growth_cond}, in which  The test is performed on gisette and MNIST dataset. We compare the original S4N-H method and a variant of it, where Newton steps are always accepted, regardless of growth conditions (though they are checked). From the figures, we can see that the process may help to improve the method's behavior in a few iterations but doesn't help with the overall performance. 
  
%}

\subsubsection{Numerical comparison}
The datasets tested in our numerical comparison are summarized in \cref{table:datasets}. We linearly scale the entries of the data-matrix $ (a_1,...,a_N)$ to $[0,1]$ for each dataset. The datasets for multi-class classification have been manually divided into two types or features. For instance, the MNIST dataset is used for classifying even and odd digits. For all methods, we choose $x^0 = 0$ as initial point.

\begin{table}[t]
\centering
\begin{tabular}{|c||c|c|c|c|}
\hline
Data Set & Data Points $N$ & Variables $n$ & {Density} & Reference \\
\hline
$\mathtt{CINA}$ & 16033 & 132 & 29.56\% &\cite{CINA}\\
$\mathtt{gisette}$ & 6000 & 5000 & 12.97\% &\cite{gisette}\\
$\mathtt{MNIST}$ & 60000 & 784 & 19.12\% &\cite{MNIST}\\
%Covtype & 581012 & 54 & 22.12\% &\cite{Covtype}\\
$\mathtt{rcv1}$ & 20242 & 47236 & 0.16\% &\cite{RCV1}\\
%mushroom & 8124 & 112 & 18.75\% &\cite{mushroom}\\
%BIO & 145751 & 74 & 99.21\% &\cite{BIO}\\
%synthetic & 10000 & 50 & 22.12\% &\cite{synthetic}\\
\hline
\end{tabular}
\caption{A description of binary datasets used in the experiments}
\label{table:datasets}
\end{table}

%\begin{figure}[t]
%\centering
%\begin{tabular}{cc}
%\subfloat[gisette]{
%\includegraphics[width=6cm]{growth_logistic_gisette.eps}} &
%\subfloat[MNIST]{
%\includegraphics[width=6cm]{growth_logistic_MNIST.eps}}
%\end{tabular}
%\caption{A test of how growth conditions affect S4N-H. A circle is plotted when the growth conditions are not satisfied.}
%\label{figure:growth_cond}
%\end{figure}

\begin{figure}
\centering
\begin{tabular}{cc}
\subfloat[$\mathtt{CINA}$]{
\includegraphics[width=6cm]{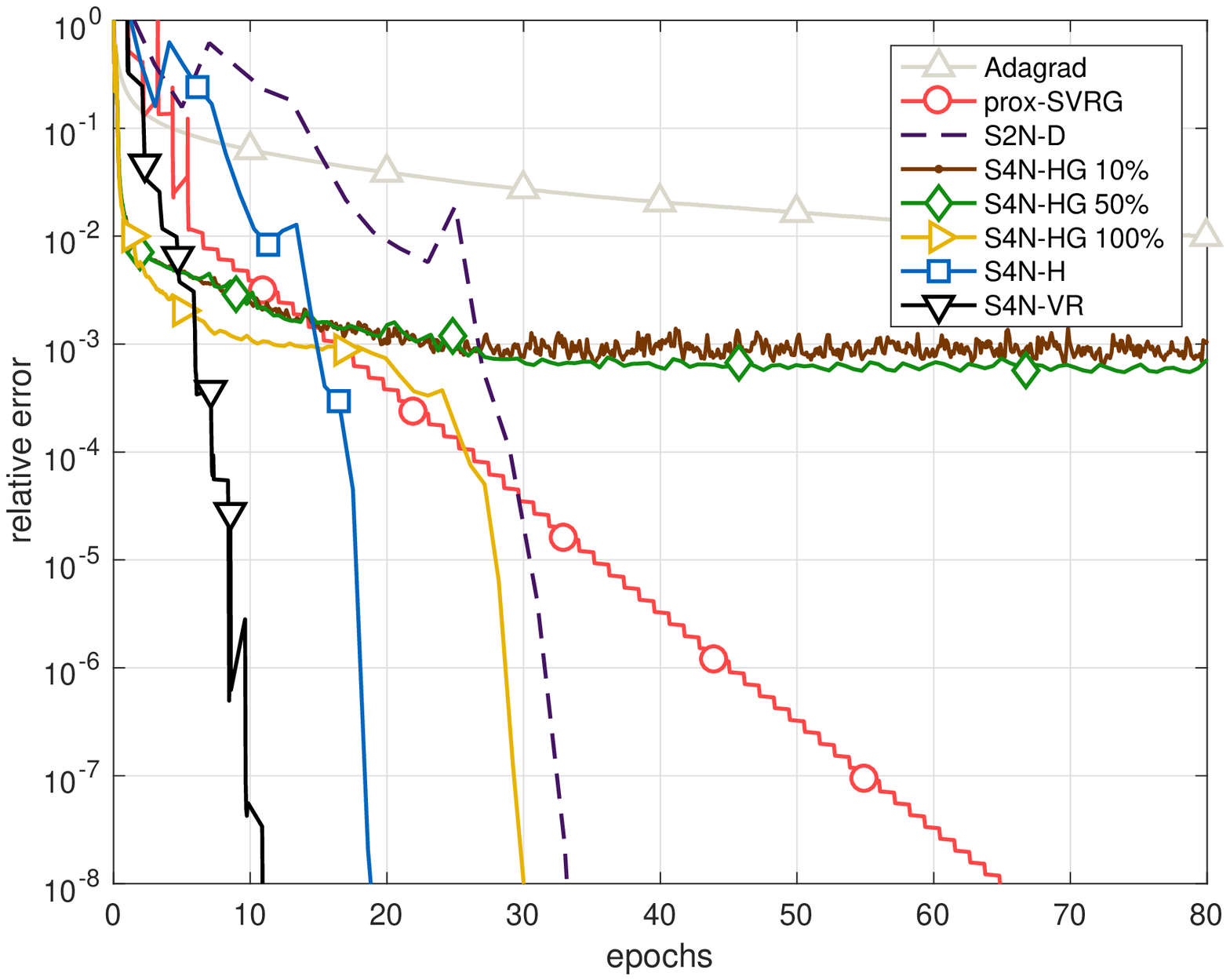}} &
\subfloat[$\mathtt{gisette}$]{
\includegraphics[width=6cm]{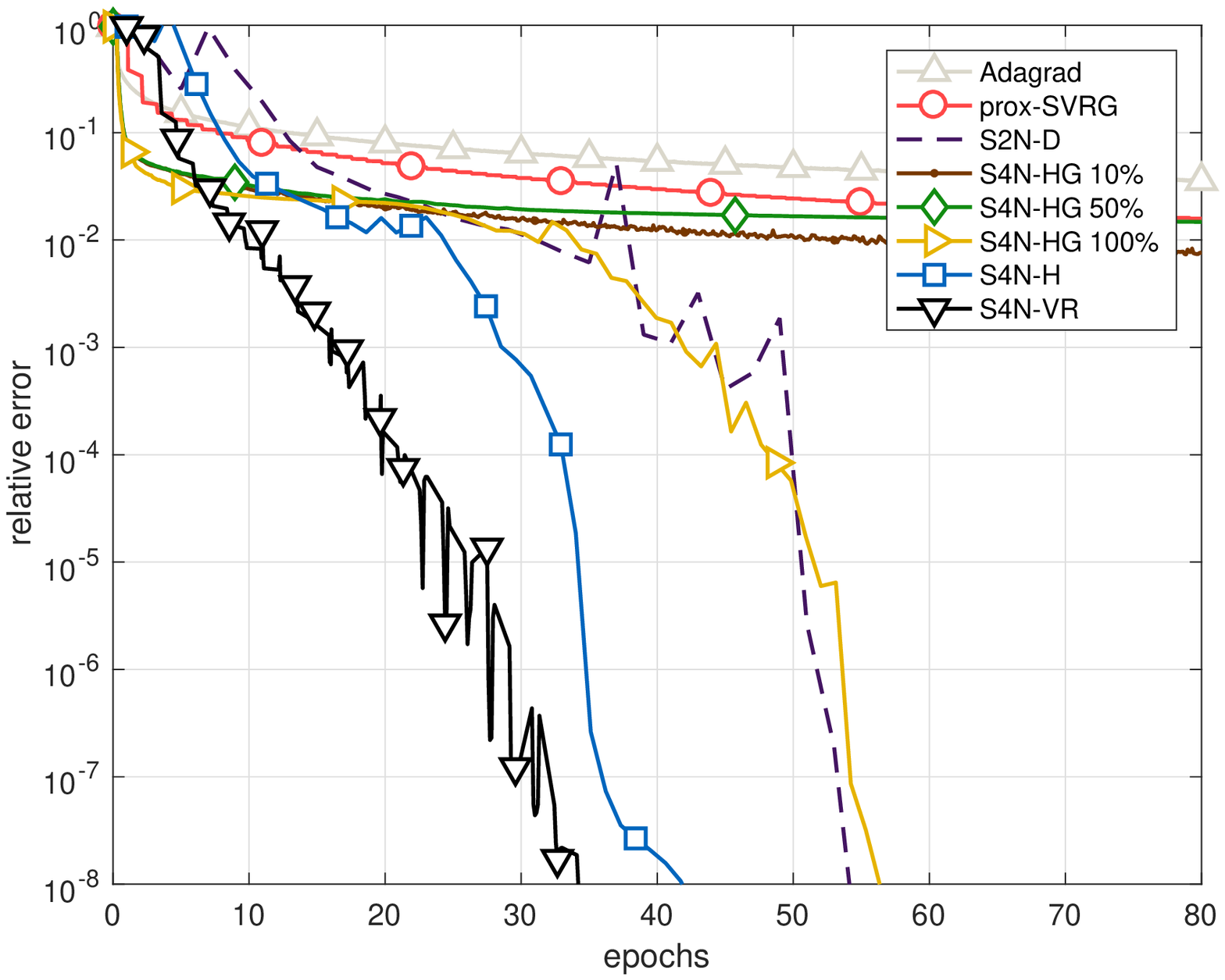}}\\
\subfloat[$\mathtt{MNIST}$]{
\includegraphics[width=6cm]{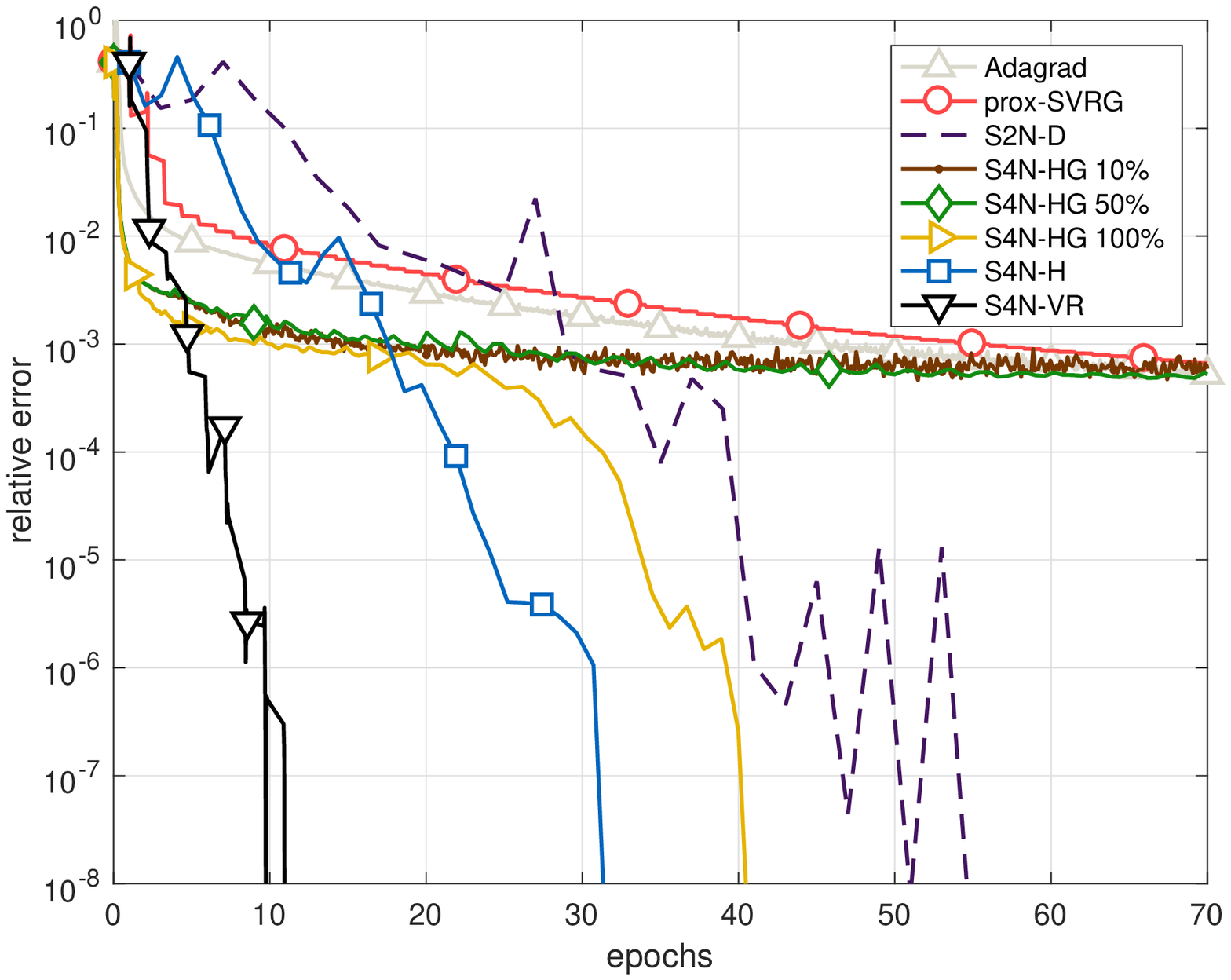}} &
\subfloat[$\mathtt{rcv1}$]{
\includegraphics[width=6cm]{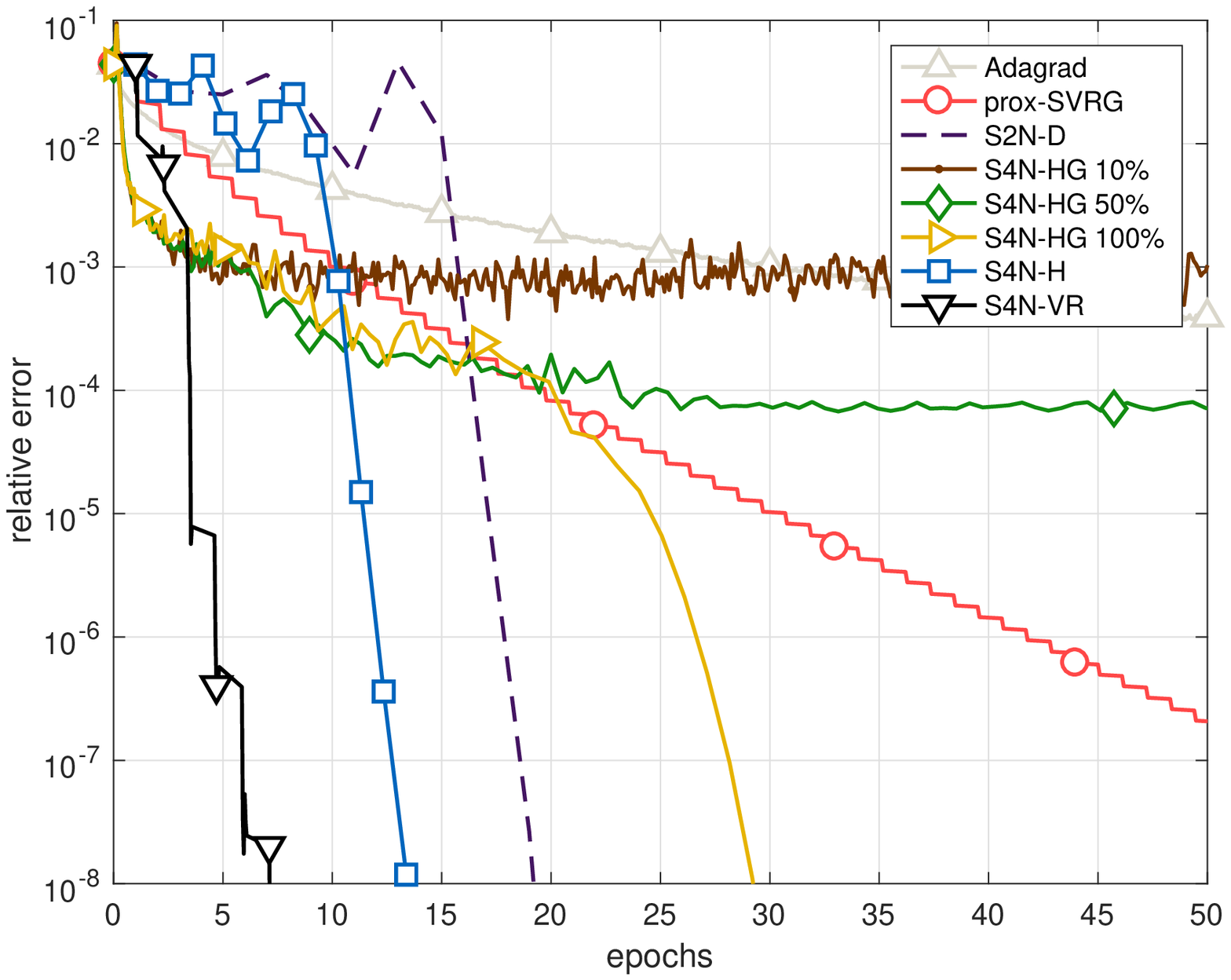}} 
\end{tabular}
\caption{Change of the relative error with respect to the required epochs for solving the $\ell_1$-logistic regression problem \cref{eq:LR}. (Averaged over 50 independent runs).}
\label{figure:logistic_epoch}
\end{figure}

\begin{figure}
\centering
\begin{tabular}{cc}
\subfloat[$\mathtt{CINA}$]{
\includegraphics[width=6cm]{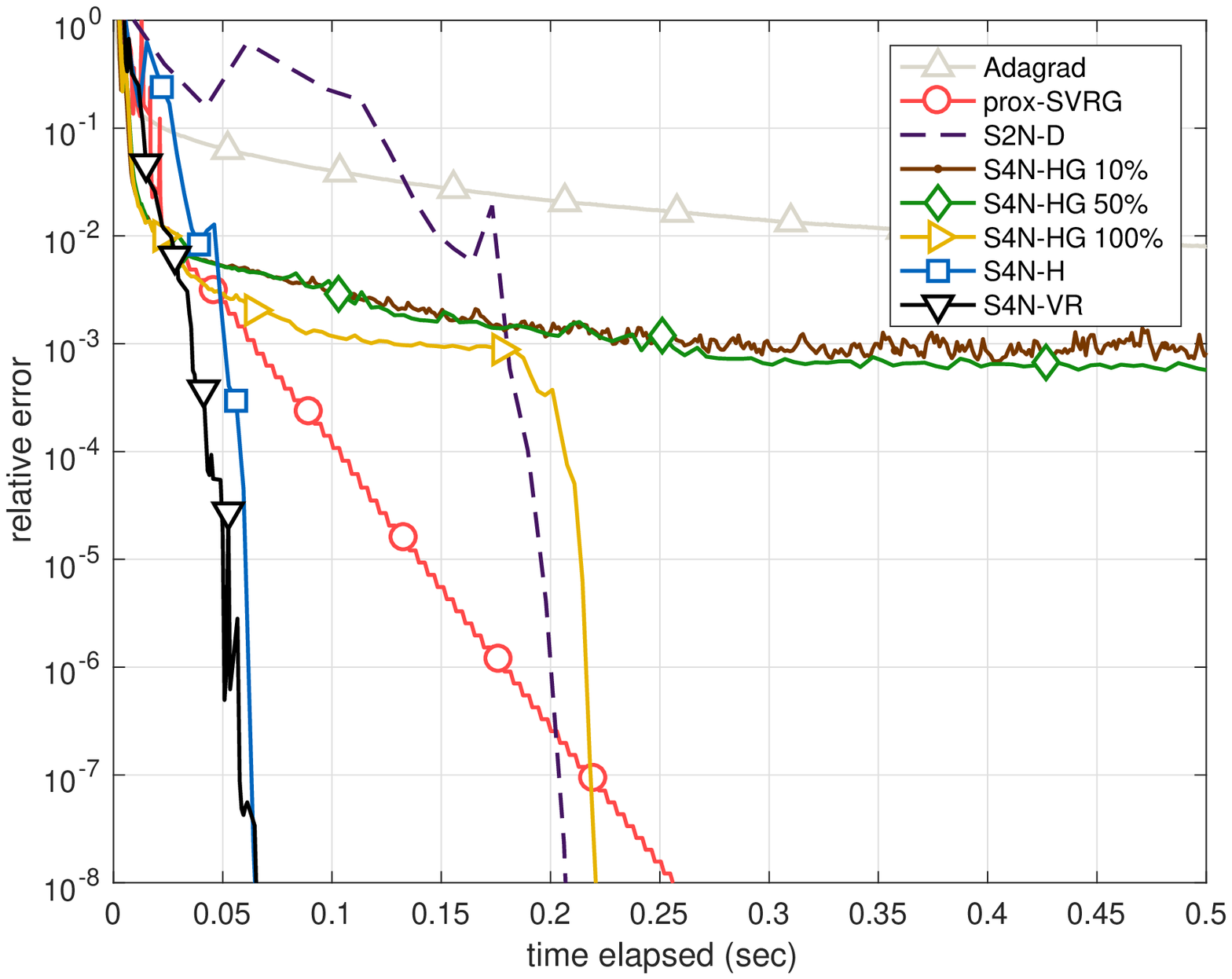}} &
\subfloat[$\mathtt{gisette}$]{
\includegraphics[width=6cm]{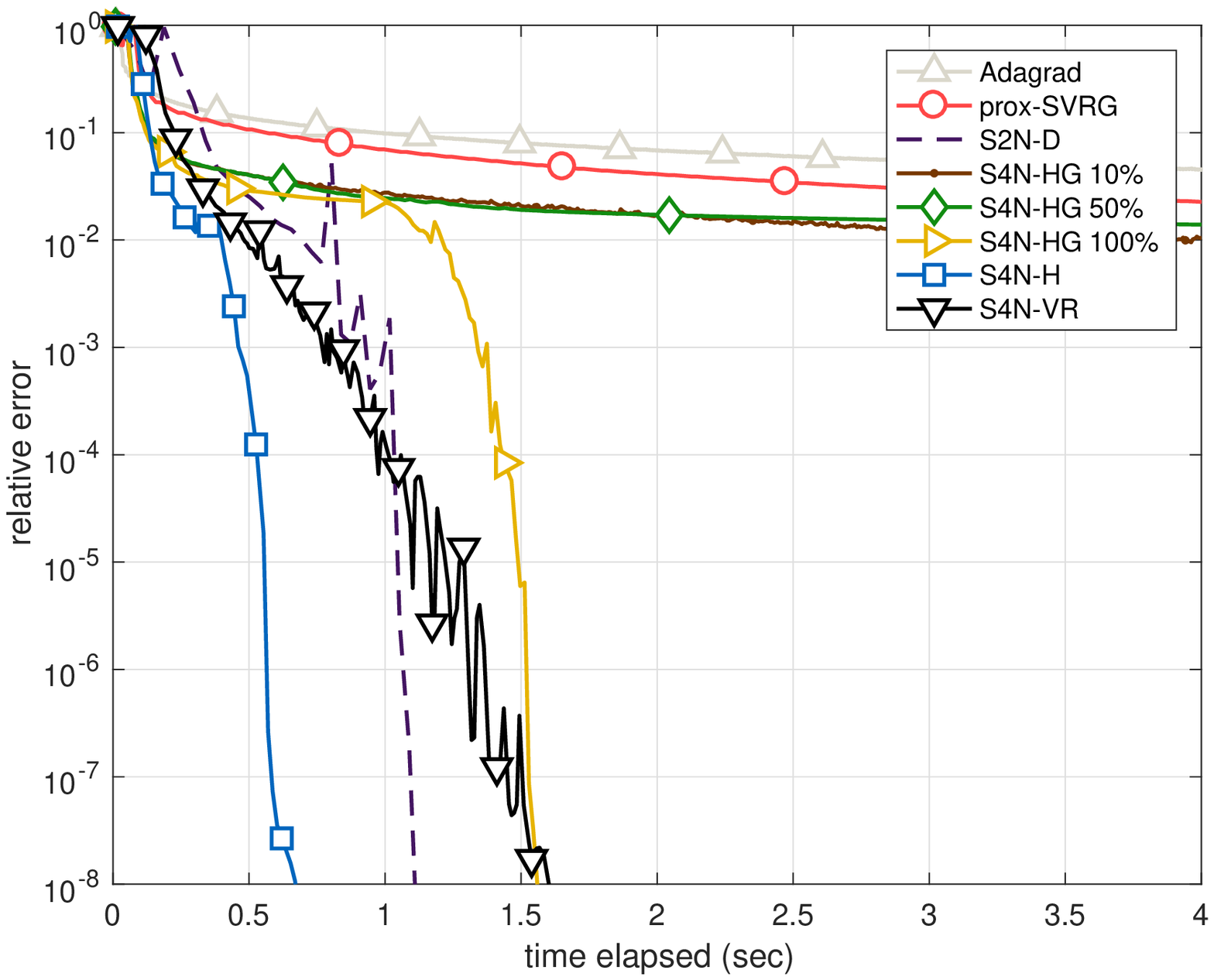}}\\
\subfloat[$\mathtt{MNIST}$]{
\includegraphics[width=6cm]{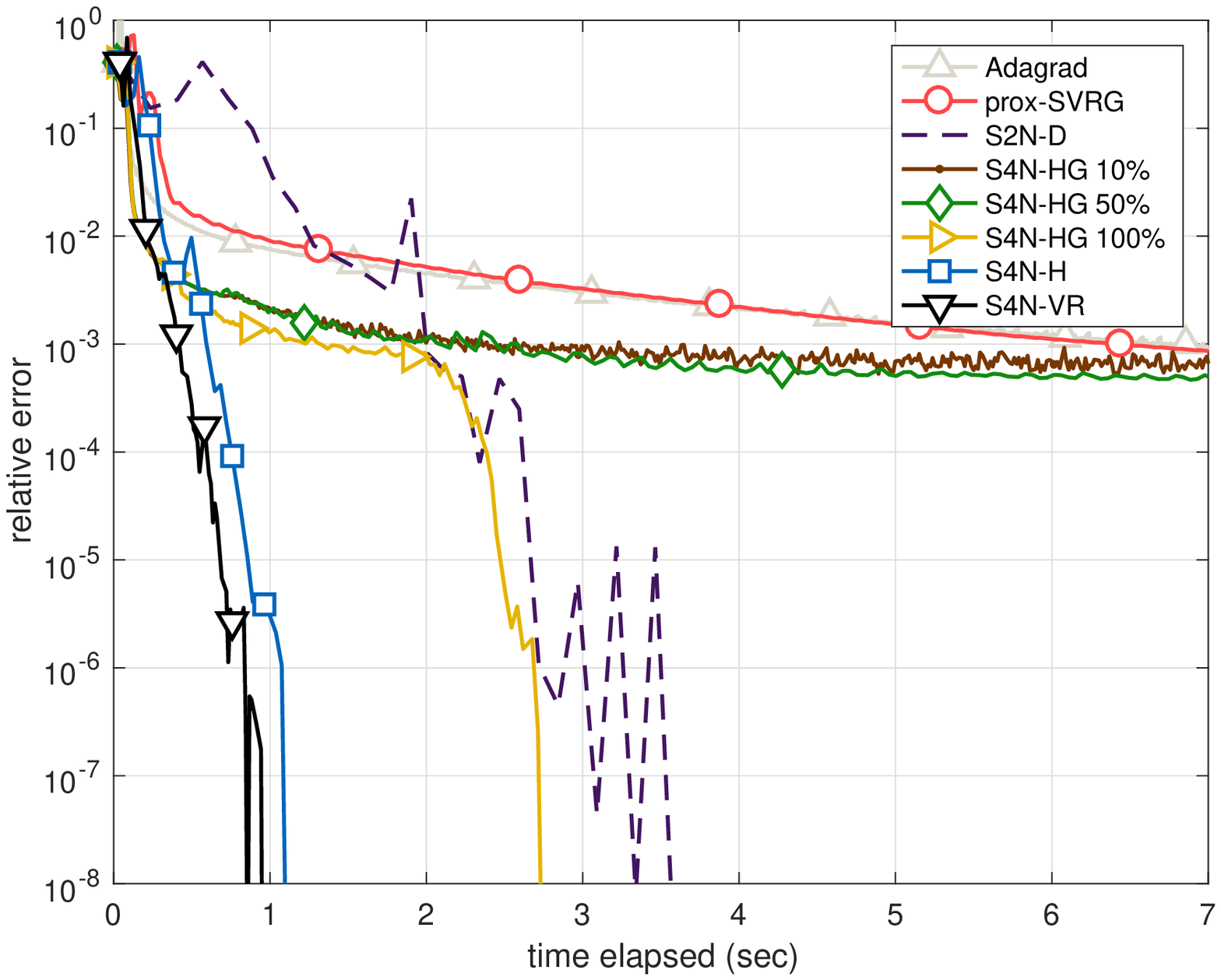}} &
\subfloat[$\mathtt{rcv1}$]{
\includegraphics[width=6cm]{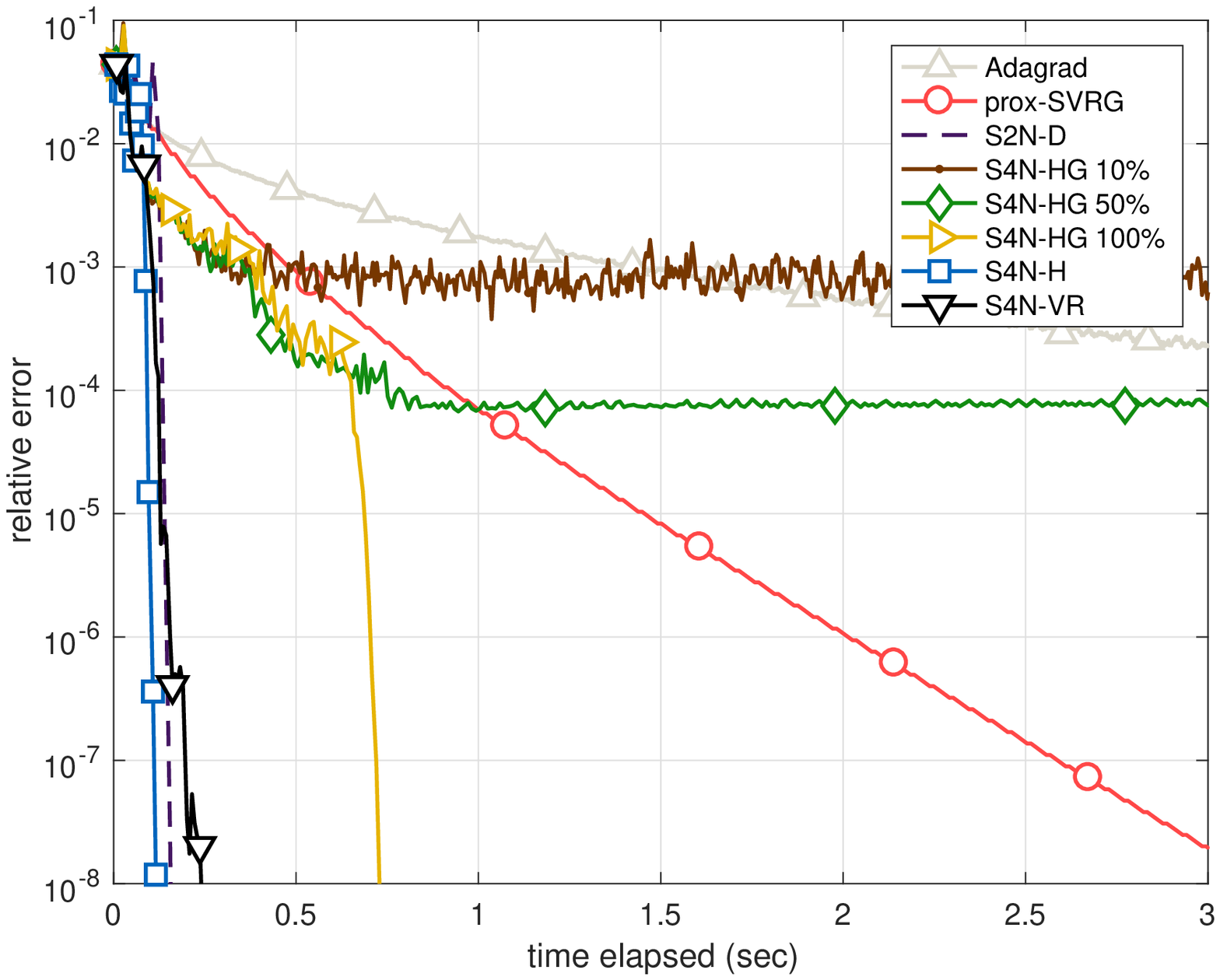}} 
\end{tabular}
\caption{Change of the relative error with respect to the cpu-time for solving the $\ell_1$-logistic regression problem \cref{eq:LR}. (Averaged over 50 independent runs).}
\label{figure:logistic_time}
\end{figure}

%\emph{Relative Error} is defined as  and an epoch stands for a full pass over the dataset. 50 runs. $x^*$?

In \cref{figure:logistic_epoch} and \cref{figure:logistic_time}, we show the performances of all methods for solving the logistic regression problem \cref{eq:LR}. The change of the \textit{relative error} $(\psi(x) - \psi(x^*)) / \max\{1,|\psi(x^*)| \}$ is reported with respect to \textit{epochs} and \textit{cpu-time}, respectively. Here, $x^*$ is a reference solution of problem \cref{eq:LR} generated by S2N-D with stopping criterion $\|F^I(x)\| \leq 10^{-12}$. Moreover, one epoch denotes a full pass over a dataset. The results presented in \cref{figure:logistic_epoch} and \cref{figure:logistic_time} are averaged over 50 independent runs. 

% vs. epochs and optimality vs. cputime on all four datasets. 

At first, we observe that S2N-D, S4N-HG 100\%, S4N-H, and S4N-VR outperform the first order method Adagrad both with respect to the required number of epochs and cpu-time. Furthermore, the different variants S4N and S2N-D seem to be especially well-suited for recovering high accuracy solutions. %In \cref{figure:logistic_time} for cputime evaluation, the advantage is more obvious.

%Except for the dataset \texttt{rcv1}, S2N-D and S4N-HG 100\% show a similar behavior. 
%It is shown that S2N-D behaves similar as the standard deterministic Newton-type methods. 
The deterministic semismooth Newton method S2N-D decreases slowly in the early stage of the iteration process, but converges rapidly when the iterates are close to an optimal solution. The results show that in the early stage the performance of S2N-D is inferior to the performance of the other stochastic methods. If a higher precision is required, then $\text{S2N-D}$ becomes more efficient and behaves similar to S4N-HG 100\%. Overall, S2N-D is not competitive with the stochastic variants S4N-H and S4N-VR and converges slower. These observations indicate the strength of stochastic algorithms in general.

Our numerical experiments show that the different performances of the stochastic methods can be roughly split into two categories. The first category includes Adagrad, S4N-HG 10\%, and S4N-HG 50\%, while the second category consists of S4N-H, S4N-HG 100\%, and S4N-VR. The performance of prox-SVRG depends on the tested datasets. While in $\mathtt{gisette}$ and $\mathtt{MNIST}$, it converges slowly and performs similarly to Adagrad, prox-SVRG shows much faster convergence on the datasets $\mathtt{CINA}$ and $\mathtt{rcv1}$. Comparing the results in \cref{figure:logistic_epoch}, it appears that the performance of S4N-HG 10\% and S4N-HG 50\% is comparable to the one of the first order method Adagrad. Since the maximum sample set size of the stochastic gradient in S4N-HG 10\% and S4N-HG 50\% is limited to $\lfloor0.1N\rfloor$ and $\lfloor0.5N\rfloor$, the associated gradient error terms still might be too large, preventing transition to fast local convergence and causing stagnation of the methods. %This phenomenon is typical for stochastic first order methods \cite{}. 
Thus and similar to the observations for stochastic quasi-Newton methods \cite{BHNS2016,WMGL2017}, the performance of S4N is greatly affected by the sampling strategy and the accuracy of the gradient approximation. This is also partly illustrated by the acceptance of the growth condition \cref{eq:growth-1}. While in S2N-D, S4N-HG 100\%, S4N-H, and S4N-VR usually every semismooth Newton step is accepted as a new iterate, a small number of Newton steps is rejected in S4N-HG 10\% and S4N-HG 50\%. This situation typically occurs in the $\mathtt{rcv1}$ dataset, when either S4N-HG 10\% or S4N-HG 50\% stagnates and the stochastic Newton step does not provide sufficient progress to be accepted. The results in \cref{figure:logistic_epoch} and \cref{figure:logistic_time} demonstrate that the performance of S4N can be further improved by increasing the sample size of the gradient gradually to its full size, as in S4N-HG 100\%, or by introducing an additional variance reduction technique as in S4N-VR. We also observe that S4N-VR outperforms most of the other methods (especially with respect to number of required epoch evaluations) which indicates that the combination of second order information and variance reduction is advantageous and very promising. %The theoretical investigation and more numerical tests of S4N-VR are left for future work.

\begin{figure}
\centering
\begin{tabular}{cc}
\subfloat[$\mathtt{CINA}$]{
\includegraphics[width=6cm]{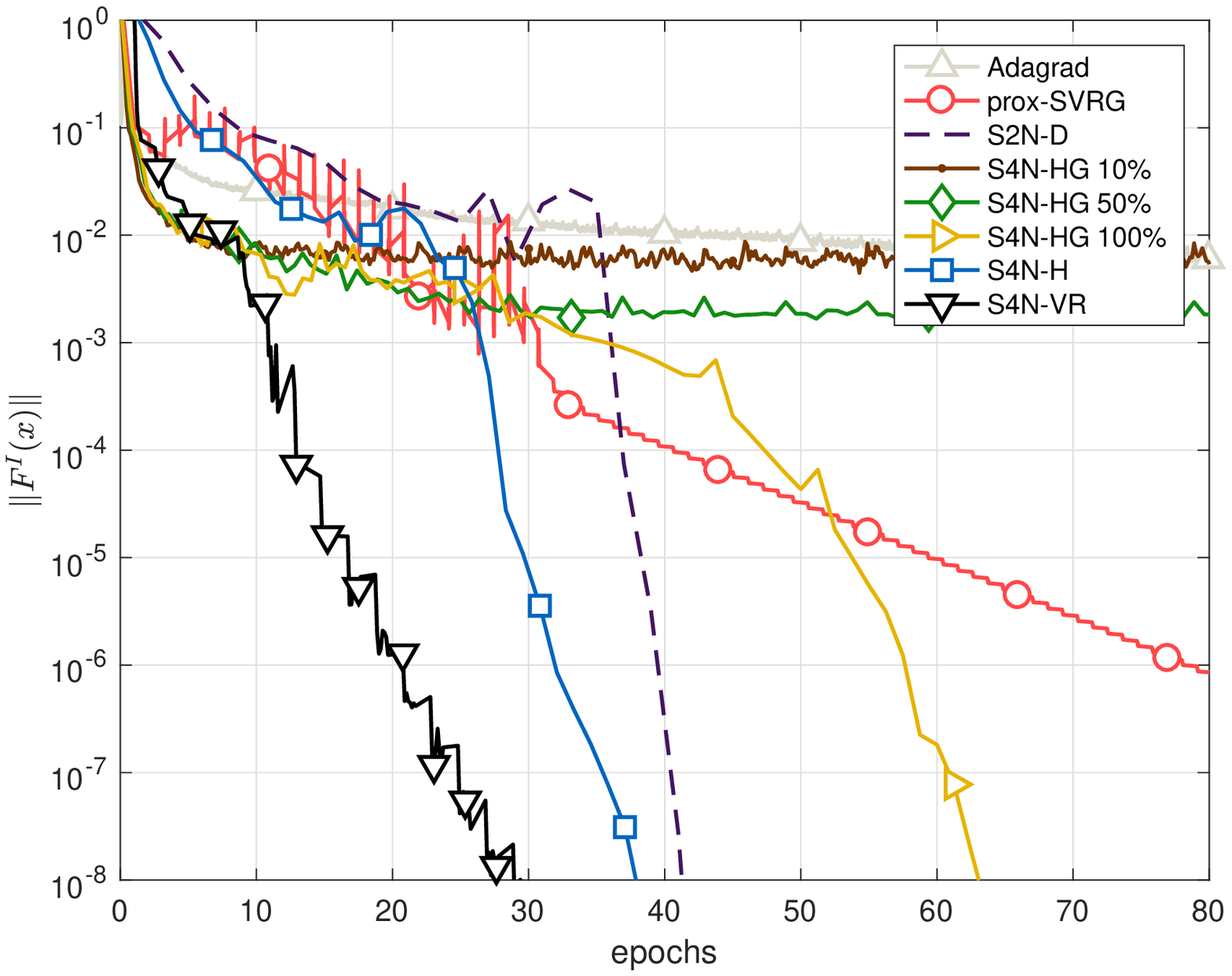}} &
\subfloat[$\mathtt{gisette}$]{
\includegraphics[width=6cm]{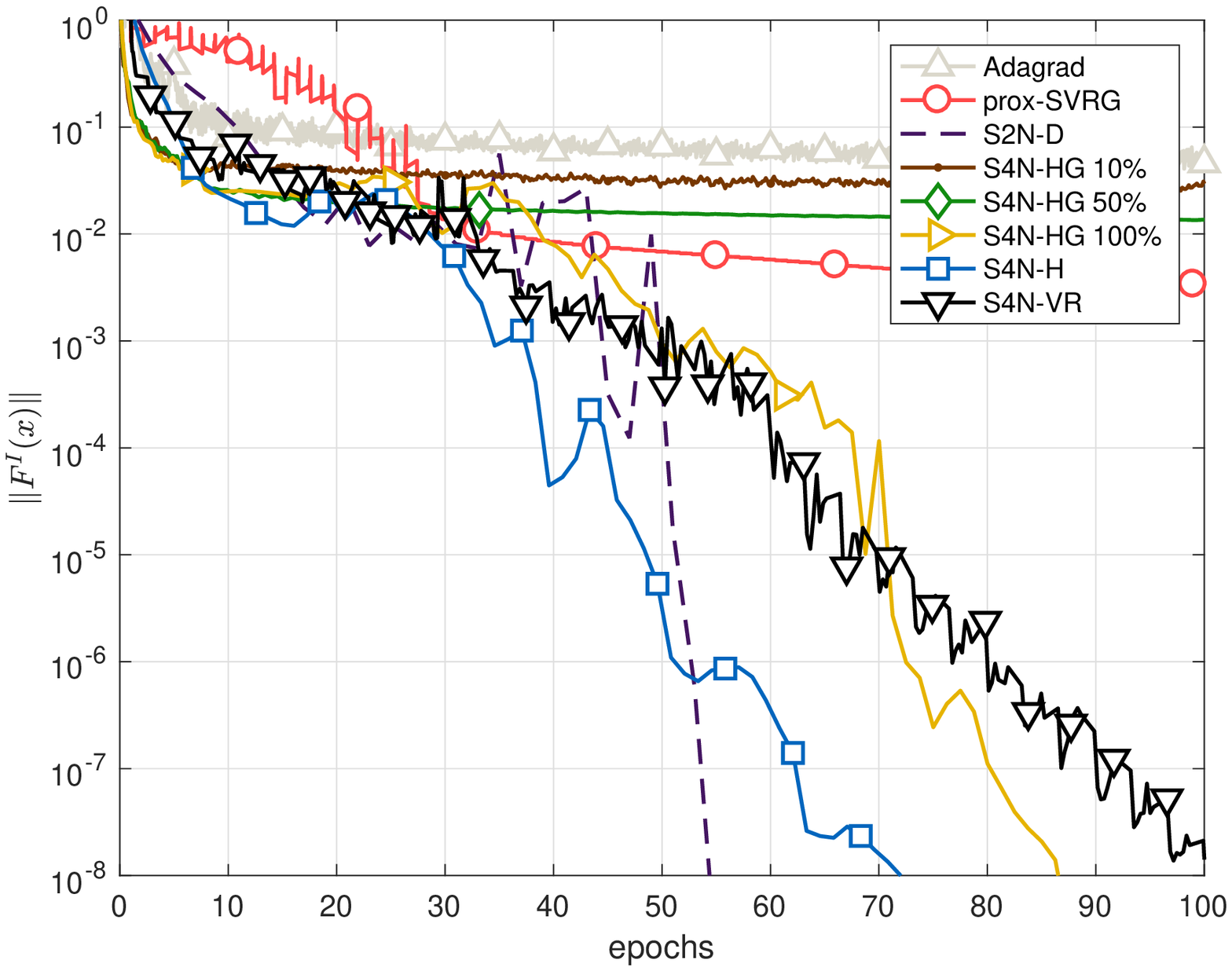}} \\
\subfloat[$\mathtt{MNIST}$]{
\includegraphics[width=6cm]{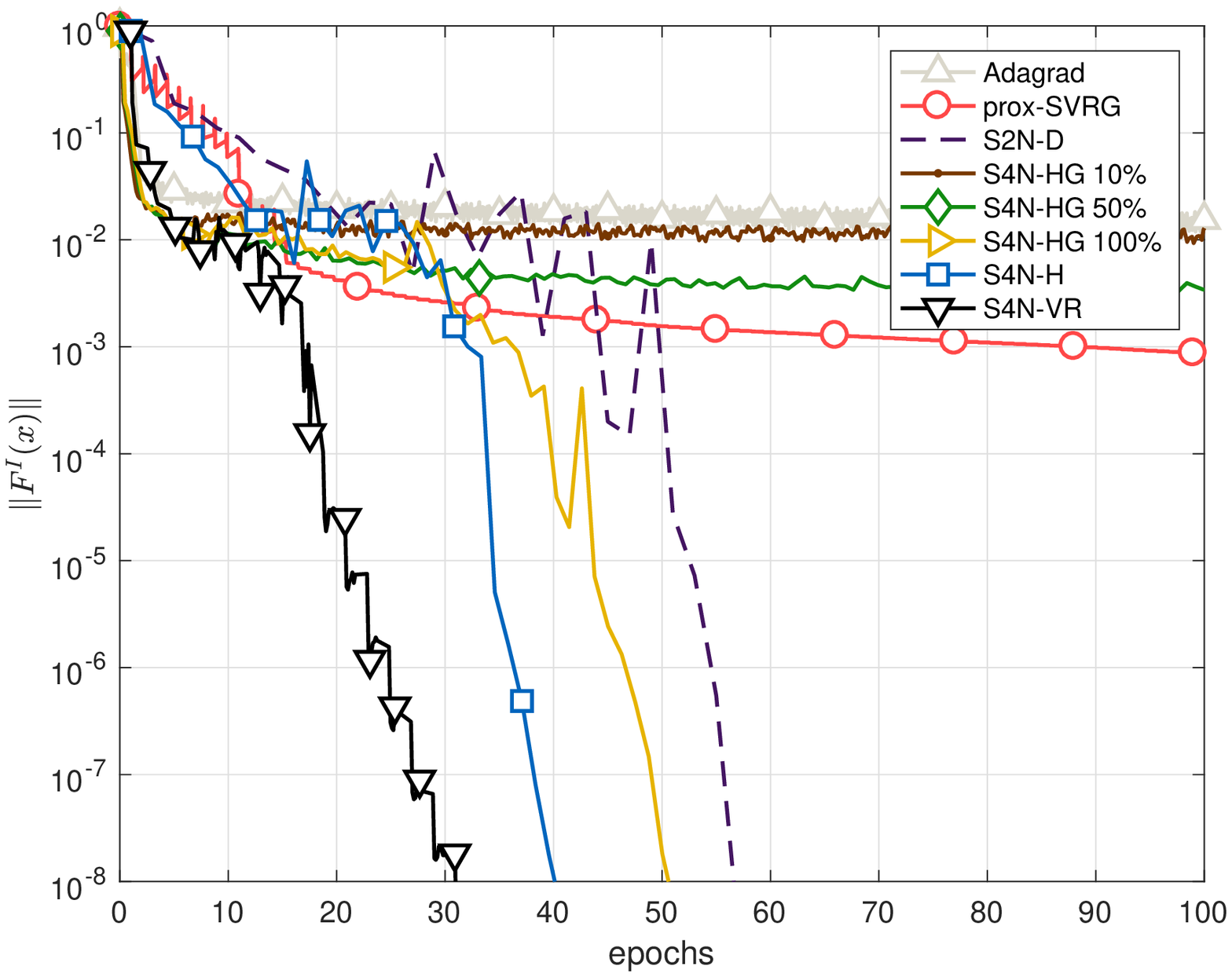}} &
\subfloat[$\mathtt{rcv1}$]{
\includegraphics[width=6cm]{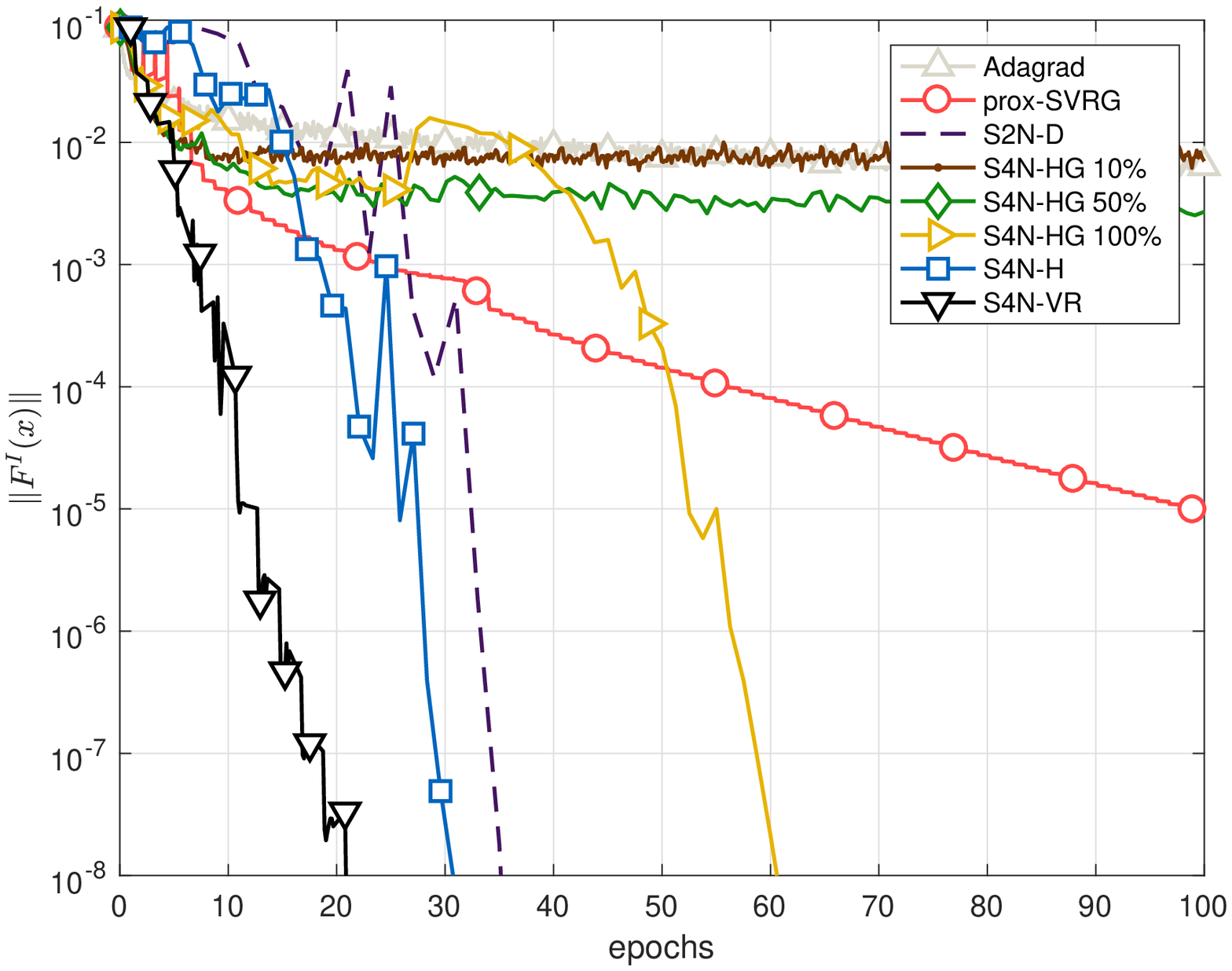}} 
\end{tabular}
\caption{Change of the residual $\|F^I(x)\|$ with respect to epochs for solving the nonconvex binary classification problem \cref{eq:nonconvex2}. (Averaged over 50 independent runs).}
\label{figure:sigmoid_epoch}
\end{figure}

\begin{figure}
\centering
\begin{tabular}{cc}
\subfloat[$\mathtt{CINA}$]{
\includegraphics[width=6cm]{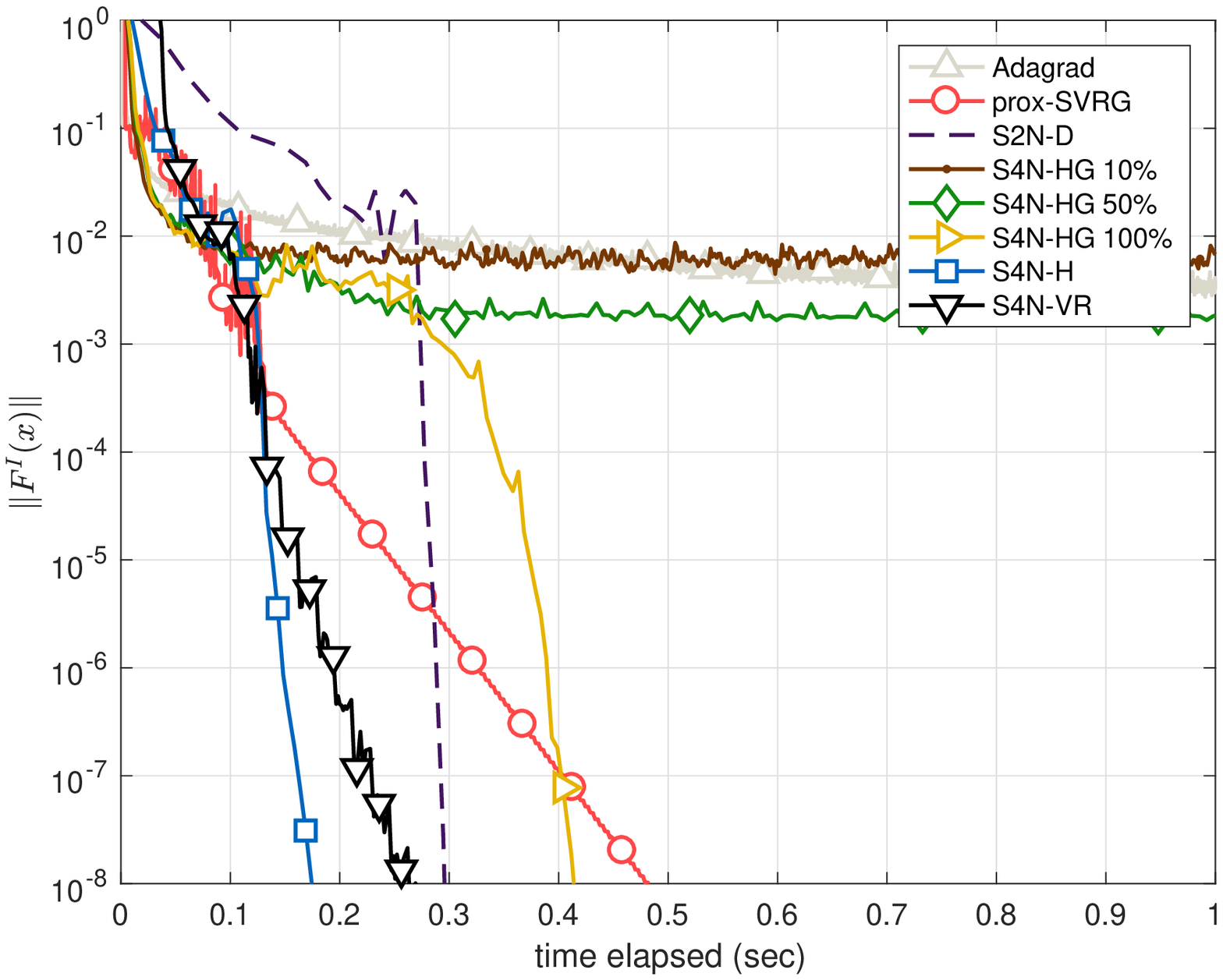}} &
\subfloat[$\mathtt{gisette}$]{
\includegraphics[width=6cm]{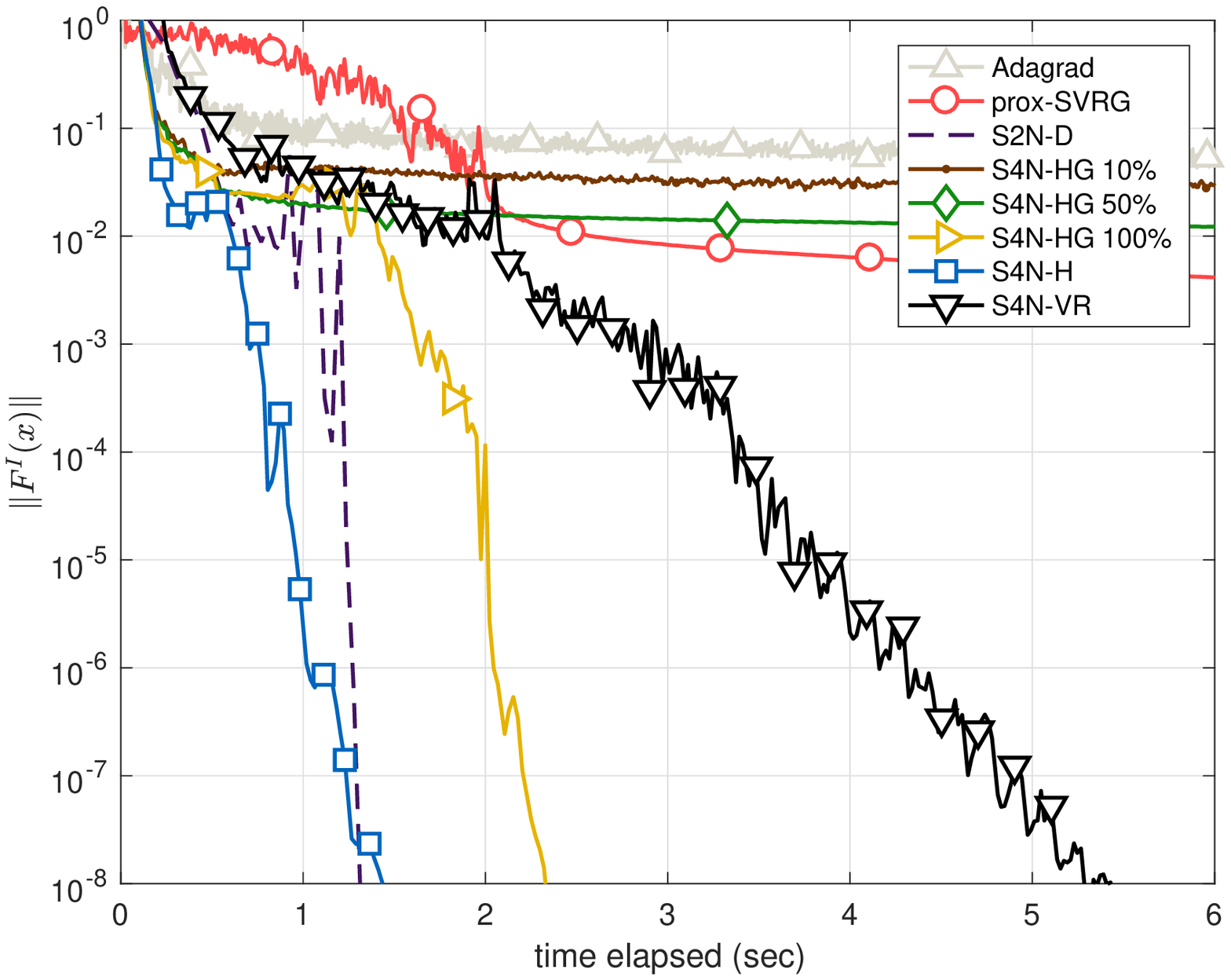}} \\
\subfloat[$\mathtt{MNIST}$]{
\includegraphics[width=6cm]{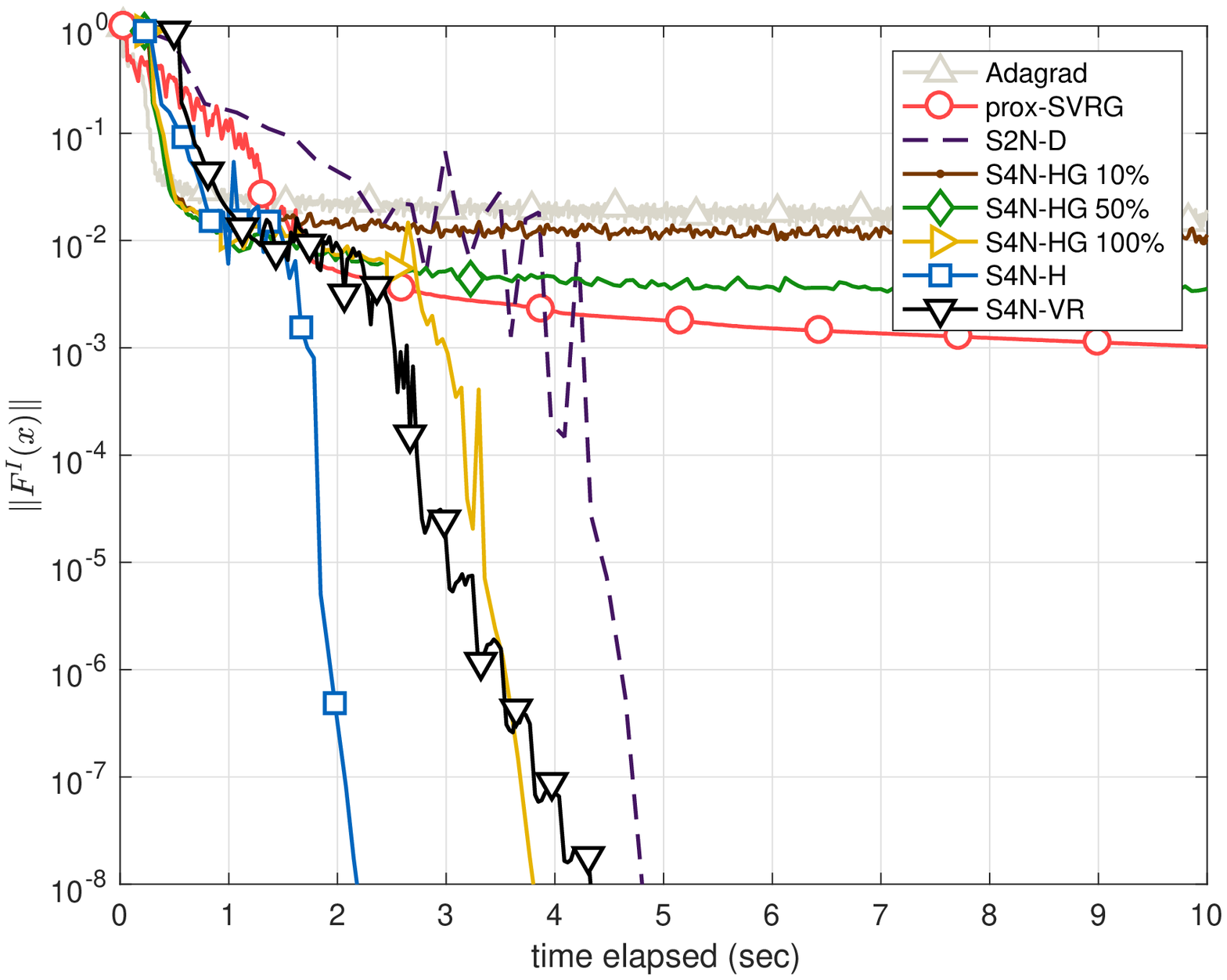}} &
\subfloat[$\mathtt{rcv1}$]{
\includegraphics[width=6cm]{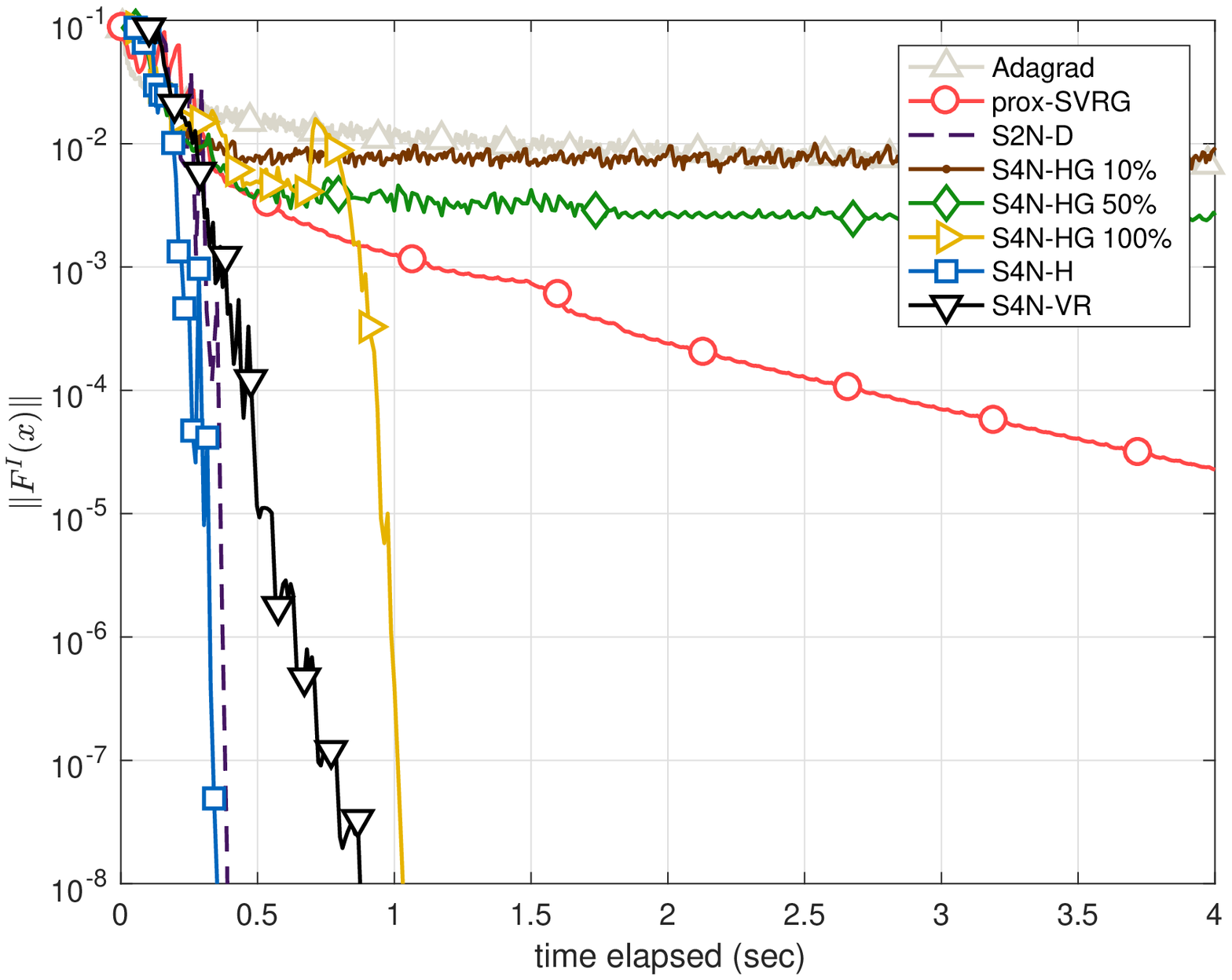}} 
\end{tabular}
\caption{Change of the residual $\|F^I(x)\|$ with respect to cpu-time for solving the nonconvex binary classification problem \cref{eq:nonconvex2}. (Averaged over 50 independent runs). \vspace{-1ex}}
\label{figure:sigmoid_time}
\end{figure}

\subsection{Nonconvex Binary Classification}
In this subsection, we consider the following nonconvex, binary classification problem \cite{MasBaxBarFre99,WMGL2017}
%
%\be\label{eq:nonconvex1}
%\min_{x\in\R^n} \psi(x):=\Exp_{a,b}[1-\tanh(ba^\top x))]+\lambda\|x\|_1,
%\ee
%
%where  $F(x;a,b)=1-\tanh(ba^\top x))$ 
%
\be\label{eq:nonconvex2}
\min_{x\in\R^n} \frac{1}{N}\sum_{i=1}^N\,[1-\tanh(b_i\cdot\iprod{a_i}{x}))]+\mu \|x\|_1,
\ee
where $f_i(x) := 1-\tanh(b_i\iprod{a_i}{x}))$ is the sigmoid loss function and $\mu = 0.01$ is a regularization parameter. We test the same  datasets as in \cref{sec:logloss} to evaluate the performance of the different versions of S4N. 

The sampling strategy and parameters are adjusted as follows. For all S4N methods, the initial mini-batch size of the stochastic gradient is increased to $|\cS_0| = \lfloor 0.05N \rfloor$. For S4N-HG 10\%, S4N-HG 50\%, and S4N-VR, we set $|\mathcal T_0| = \lfloor 0.025N \rfloor$. The other variants of S4N start with the initial sample size $|\mathcal T_0| = \lfloor 0.05N \rfloor$. We set $t_{\max} = \lfloor 0.25N \rfloor$, except for S4N-HG 10\% where $t_{\max} = \lfloor 0.05N \rfloor$ is used. We utilize the minimal residual method (MINRES) to solve the reduced Newton system and the maximum number of MINRES-iterations is set to 32. We choose $c_\nu = 2500$ and in S4N-VR, the parameter $m$ is adjusted to $m = 8$. For the $\mathtt{gisette}$ dataset, we changed the initial value for $\lambda_0$ to 5, which significantly improved the performance of the S4N methods. All remaining parameters and strategies follow the setup discussed in the last subsection. 

The numerical results are presented in \cref{figure:sigmoid_epoch} and \cref{figure:sigmoid_time}. We report the change of the residual $\|F^I(x)\|$ with respect to the required epochs and cpu-time. In general, the overall performance of the different methods is similar to the results shown in the last subsection. However, in contrast to the convex logistic regression problem, the more accurate approximation of the Hessian seems to be beneficial and can accelerate convergence. This observation is also supported by the slightly improved performance of the deterministic semismooth Newton method S2N-D. Our results show that prox-SVRG now consistently outperforms S4N-HG 10\% and S4N-HG 50\% in recovering high precision solutions. Similar to the convex example, S4N-HG 100\% manages to significantly reduce the residual $\|F^I(x)\|$ in the first iterations. As soon as the stochastic error in the gradient approximation becomes negligible, the behavior of S4N-HG 100\% changes and fast local converges can be observed. The methods S4N-H and S4N-VR still compare favorably with the other stochastic approaches. With respect to the epochs, S4N-VR again outperforms the other methods (expect for the dataset $\mathtt{gisette}$). Regarding the required cpu-time, S4N-H achieves good results and converges quickly to highly accurate solutions.   %From the figures on all four datasets, it is shown that S4N-VR works obviously better than any of the others. We also observe that , S4N-HG 100\% and S4N-H oscillate quite a lot. However, when incorporating with variance reduction technique, the oscillations of  the method S4N-VR are much less. In  \cref{figure:sigmoid_time} on the evaluation of cputime performance, we can see that S4N-VR and S4N-H perform very similarly and decrease faster than other methods, except that on the $\mathtt{CINA}$ dataset SVRG is slightly better. 

%----------------------------------------------------------------------------------------------------------------%
% CONCLUSION
%----------------------------------------------------------------------------------------------------------------%

\section{Conclusion}
In this paper, we investigate a stochastic semismooth Newton method for solving nonsmooth and nonconvex minimization problems. In the proposed framework, the gradient and Hessian of the smooth part of the objective function are approximated by general stochastic first and second order oracles. This allows the application of various sub-sampling and variance reduction techniques or other stochastic approximation schemes. The method is based on stochastic semismooth Newton steps, stochastic proximal gradient steps, and growth conditions and a detailed discussion of the global convergence properties is provided. Under suitable assumptions, transition to fast local convergence is established. More specifically, we show that the method converges locally with high probability with an r-linear or r-superlinear rate if the sample sizes of the stochastic gradient and Hessian are increased sufficiently fast. The approach is tested on an $\ell_1$-logistic regression and a nonconvex binary classification problem on a variety of datasets. The numerical comparisons indicate that our algorithmic framework and especially the combination of (generalized) second order information and variance reduction are promising and competitive.
%We should point out that our method can be further improved %by incorporating variance reduced technique \cite{WZ2017} and 
%by using non-uniform sampling strategies  \cite{XYRRM2016,XRKM2017}.

%----------------------------------------------------------------------------------------------------------------%
% APPENDIX
%----------------------------------------------------------------------------------------------------------------%

\appendix
\section{Proofs}

%----------------------------------------------------------------------------------------------------------------%
% SUB-SECTION: PROOF MEASURABILITY OF ITERATES
%----------------------------------------------------------------------------------------------------------------%

\subsection{Proof of \cref{fact:one}} \label{sec:app-fact}

We start with some preparatory definitions. Let $(\Omega, \mathcal F)$ and $(\Xi,\mathcal X)$ be measurable spaces, then $\mathcal F \otimes \mathcal X$ denotes the usual product $\sigma$-algebra of the product space $\Omega \times \Xi$. We use $\mathcal B(\Rn)$ to denote the Borel $\sigma$-algebra of $\Rn$. Let $\Gamma : \Omega \rightrightarrows \Rn$ be a multifunction. Then, for some set $C \subset \Rn$, we define $\Gamma^{-1}(C) := \{\omega \in \Omega : \Gamma(\omega) \cap C \neq \emptyset \}$. Following \cite{AliBor06}, the multifunction $\Gamma$ is called \textit{weakly measurable} if for all \textit{open} sets $C \subset \Rn$, the set $\Gamma^{-1}(C)$ is measurable. The function $\Gamma$ is called \textit{measurable} if $\Gamma^{-1}(C)$ is measurable for all \textit{closed} sets $C \subset \Rn$. If $\Gamma$ is closed-valued then these two concepts coincide and are actually equivalent, see \cite{Roc76}. The \textit{graph} of $\Gamma$ is given by $\gra\;\Gamma := \{(\omega,x) \in \Omega \times \Rn : x \in \Gamma(\omega) \}$ and the \textit{closure} $\bar \Gamma$ of $\Gamma$ is defined via $\bar \Gamma(\omega) := \cl~\Gamma(\omega)$ for all $\omega \in \Omega$. The next result is straightforward.

\begin{lemma} \label{lemma:app-meas-1} Let $(\Omega,\mathcal F)$ be a measurable space and let $S : \Omega \rightrightarrows \Rn$ be a multifunction. Then, the indicator function $\iota : \Omega \times \Rn \to \R$, $\iota(\omega,x) := {\mathds 1}_{S(\omega)}(x)$ is jointly $\mathcal F \otimes \mathcal B(\Rn)$-measurable if and only if $\gra\;S \in \mathcal F \otimes \mathcal B(\Rn)$. 
\end{lemma}

\begin{lemma} \label{lemma:app-meas-2} Let $(\Omega,\mathcal F)$ be a measurable space and let $f : \Omega \times \Rn \to \R$ be a given Carath\'{e}odory function. Moreover, suppose that the mapping $\vp : \Rn \to \Rex$ is convex, lower semicontinuous, and proper and let us consider the multifunction 
\[ S_\gamma : \Omega \rightrightarrows \Rn, \quad S_\gamma(\omega) := \{x \in \Rn : \vp(x) + f(\omega,x) \leq \gamma \} \]
for some $\gamma \in \R$. Then, $S_\gamma$ has measurable graph, i.e., it holds $\gra\,S_\gamma \in \mathcal F \otimes \mathcal B(\Rn)$.
\end{lemma}

\begin{proof} Since $\vp$ is convex, lower semicontinuous, and proper, it can be written as the supremum of its affine minorants. In particular, by \cite[Corollary 13.36 and Proposition 7.11]{BauCom11} there exists a countable, dense subset $\mathcal D := \{(y_1,\tau_1),(y_2,\tau_2),... \}$ of $\epi~\vp^*$ such that  
\[ \vp(x) = \vp^{**}(x) = \sup_{(y,\tau) \in \epi~\vp^*}~\iprod{y}{x} - \tau = \sup_{k \in \N}~\iprod{y_k}{x} - \tau_k, \quad \forall~x \in \Rn. \] 
(Here, we also used the closedness of the epigraph $\epi~\vp^*$). Next, let us define the family of multifunctions $S_{\gamma,k} : \Omega \rightrightarrows \Rn$,
\[  S_{\gamma,k}(\omega) := \{x \in \Rn: \iprod{y_k}{x} - \tau_k + f(\omega,x) < \gamma \}, \quad k \in \N. \] 
Since the mapping $\tilde f_k(\omega,x) := \iprod{y_k}{x} - \tau_k + f(\omega,x)$ is obviously a Carath\'{e}odory function, it follows from \cite[Lemma 18.7]{AliBor06} that each multifunction $S_{\gamma,k}$ is (weakly) measurable. Moreover, \cite[Theorem 18.6]{AliBor06} implies that the closures $\bar S_{\gamma,k}$ have measurable graph. Thus, due to $\gra~S_\gamma = \bigcup_{k=1}^\infty \gra~\bar S_{\gamma,k}$, we can infer $\gra~S_\gamma \in \mathcal F \otimes \mathcal B(\Rn)$.
\end{proof}

\begin{lemma} \label{lemma:app-penrose} Let $(\Omega,\mathcal F)$ be a measurable space and let $T : \Omega \to \R^{n \times n}$ and $y : \Omega \to \Rn$ be measurable functions. Then, the mapping $\omega \mapsto \zeta(\omega) := T(\omega)^+ y(\omega)$ is measurable. \end{lemma}
\begin{proof} By \cite{GolLoa13}, we have $\lim_{\lambda \to 0} (A^\top A + \lambda I)^{-1} A^\top = A^+$ for any matrix $A \in \R^{n \times n}$. Now, let $(\lambda_k)_k \subset \R$ be an arbitrary sequence converging to zero. Then, a continuity argument implies that the mapping $\omega \mapsto \zeta_k(\omega) := (T(\omega)^\top T(\omega) + \lambda_k I)^{-1} T(\omega)^\top y(\omega)$ is $\mathcal F$-measurable for all $k \in \N$. Since $\zeta$ is the point-wise limit of the sequence $(\zeta_k)_k$, this finishes the proof
\end{proof} 

We now turn to the proof of \cref{fact:one}.

\begin{proof} %First, let us note that the mappings $\nabla_x F : \Rn \times \Xi \to \Rn$ and $\nabla^2_{xx} F : \Rn \times \Xi \to \R^{n \times n}$ are also Carath\'{e}odory functions. As a consequence, $F$, $\nabla_x F$, and $\nabla_{xx}^2 F$ are jointly measurable, i.e., it holds 
Since the stochastic oracles $\cG$ and $\cH$ are Carath\'{e}odory functions, it follows that $\cG$ and $\cH$ are jointly measurable, i.e., it holds
%
%\be \label{eq:app-joint-meas} F \in \mathcal B(\Rn) \otimes \cG, \quad \nabla_x F \in \mathcal B(\Rn) \otimes \cG, \quad \nabla^2_{xx} F \in \mathcal B(\Rn) \otimes \cG, \ee 
\be \label{eq:app-joint-meas} \cG \in \mathcal B(\Rn) \otimes \mathcal X, \quad \cH \in \mathcal B(\Rn) \otimes \mathcal X, \ee 
see, e.g., \cite[Lemma 8.2.6]{AubFra09} or \cite[Section 4.10]{AliBor06}. We now prove the slightly extended claim
\be \label{eq:app-induction} \pstep, \; \nstep \in \mathcal F_k, \quad {\sf Y}_{k+1}, \; \theta_{k+1} \in \hat{\mathcal F}_k, \quad \text{and} \quad x^{k+1} \in \hat{\mathcal F}_k \ee
inductively. Let us suppose that the statement \cref{eq:app-induction} holds for $k - 1$, $k \in \N$. Then, due to \cref{eq:app-joint-meas} and $x^k \in \hat {\mathcal F}_{k-1} \subset \mathcal F_k$ and using the $(\mathcal F_k,\mathcal B(\Rn) \otimes \mathcal X)$-measurability of the functions $\xi^k_i : \Omega \to \Rn \times \Xi$, $\xi^k_i(\omega) := (x^k(\omega),s^k_i(\omega))$ and $\tau^k_j : \Omega \to \Rn \times \Xi$, $\tau^k_j(\omega) := (x^k(\omega),t^k_j(\omega))$ for $i = 1,...,\ngk$ and $j = 1,...,\nhk$, it follows
\[ \gfsubk(x^k) = \frac{1}{\ngk} \sum_{i=1}^{\ngk} \cG(\xi^k_i) \in \mathcal F_k \quad \text{and} \quad \hfsubk(x^k) = \frac{1}{\nhk} \sum_{j=1}^{\nhk} \cH(\tau^k_j) \in \mathcal F_k. \]
Since the proximity operator $\proxt{\Lambda_k}{r}$ is (Lipschitz) continuous, this also implies $\usubk(x^k) \in \mathcal F_k$, $\psubk(x^k) \in \mathcal F_k$, and $\Fsubk(x^k) \in \mathcal F_k$ and thus, we have $\pstep \in \mathcal F_k$. Moreover, by assumption (C.1) and \cref{lemma:app-penrose}, we can infer $\nstep \in \mathcal F_k$. Next, we study the measurability of the set ${\sf S}_k$ used in the definition of ${\sf Y}_{k+1}$. Since lower semicontinuous functions are (Borel) measurable and we have $\theta_k, x^k \in \hat {\mathcal F}_{k-1} \subset \mathcal F_k$, the mapping ${\sf p}_k(\omega,a,b) := -\beta\theta_k(\omega)^{1-p} \|b\|^{p} - \psi(x^k(\omega))$ is a jointly $\mathcal F_k \otimes \mathcal B(\Rn) \otimes \mathcal B(\Rn)$-measurable Carath\'{e}odory function. Thus, by \cref{lemma:app-meas-2}, the multifunction $\lev_{\veps_k^2}\,{\sf P}_k : \Omega \rightrightarrows \Rn \times \Rn$,
\[ (\lev_{\veps_k^2}\,{\sf P}_k)(\omega) =  \{(a,b) \in \Rn \times \Rn : \psi(a) + {\sf p}_k(\omega,a,b) \leq \veps_k^2\} \]
has measurable graph. Similarly, since ${\sf q}_k(\omega,a) := \|a\| - (\nu_k + \eta)\theta_k(\omega)$ is a Carath\'{e}odory function, we can argue that $\lev_{\veps_k^1}\,{\sf Q}_k : \Omega \rightrightarrows \Rn \times \Rn$ has measurable graph. (More precisely, since the set $(\lev_{\gamma}\,{\sf Q}_k)(\omega)$ is compact-valued for all $\gamma \in \R$ and $\omega \in \Omega$, we can even infer that $\lev_{\veps_k^1}\,{\sf Q}_k$ is measurable, see \cite[Lemma 18.4]{AliBor06}). This easily yields $\gra \; [\dom~r \times \lev_{\veps_k^1}\,{\sf Q}_k] \in \mathcal F_k \otimes \mathcal B(\Rn) \otimes \mathcal B(\Rn)$ and hence, by \cref{lemma:app-meas-1}, the indicator function ${\mathds 1}_{{\sf S}_k}(\cdot,\cdot)$ is jointly $\mathcal F_k \otimes \mathcal B(\Rn) \otimes \mathcal B(\Rn)$-measurable. As before, due to $\mathcal F_k \subset \hat {\mathcal F}_k$, we have $\gfsubkk(\nstep) \in \hat {\mathcal F}_k$, 
\[ u^{\Lambda_{k+1}}_{s^{k+1}}(\nstep) \in \hat {\mathcal F}_k, \quad p^{\Lambda_{k+1}}_{s^{k+1}}(\nstep) \in \hat{\mathcal F}_k, \quad \text{and} \quad \Fsubkk(\nstep) \in \hat {\mathcal F}_k, \]
which finally implies that the binary variable ${\sf Y}_{k+1}$ is $\hat{\mathcal F}_k$-measurable. Using the representation \cref{eq:def-xkk}, it now follows $x^{k+1} \in \hat{\mathcal F}_k$ and, due to $\theta_{k+1} = {\sf Y}_{k+1} \|\Fsubkk(\nstep)\| + (1-{\sf Y}_{k+1}) \theta_k$, we also have $\theta_{k+1} \in \hat {\mathcal F}_k$. Since the constant random variables $x^0$ and $\theta_0$ are trivially $\mathcal F_0$-measurable, we can use the same argumentation for the base case $k = 0$. This finishes the proof of \cref{fact:one}.
\end{proof}

We conclude with a remark on the existence of measurable selections of the multifunction $\mathcal M_k$. Due to \cite{Clarke1990}, the generalized derivative $\partial \proxt{\Lambda_k}{r} : \Rn \rightrightarrows \R^{n \times n}$ is an upper semicontinuous, compact-valued function for all $k \in \N$. Hence, by \cite[Lemma 4.4]{Ulb02}, $\partial \proxt{\Lambda_k}{r}$ is (Borel) measurable for all $k$. As we have shown inductively, the function $\omega \mapsto u^{\Lambda_k}_{s^k(\omega)}(x^k(\omega))$ is $\mathcal F_k$-measurable and thus, by \cite[Lemma 4.5]{Ulb02}, the multifunction 
\[ \mathcal D_k : \Omega \rightrightarrows \R^{n \times n}, \quad \mathcal D_k(\omega) := \partial \proxt{\Lambda_k}{r}(u^{\Lambda_k}_{s^k(\omega)}(x^k(\omega))), \] 
is nonempty, closed-valued, and measurable with respect to the $\sigma$-algebra $\mathcal F_k$ for all $k$. The Kuratowski-Ryll-Nardzewski Selection Theorem, \cite{Roc76,AubFra09,AliBor06}, now implies that $\mathcal D_k$ admits an $\mathcal F_k$-measurable selection $D_k : \Omega \to \R^{n \times n}$. Using $\hfsubk(x^k) \in \mathcal F_k$, this also implies that $M_k := I - D_k(I - \Lambda_k^{-1}\hfsubk(x^k))$ is an $\mathcal F_k$-measurable selection of $\mathcal M_k$. 

%----------------------------------------------------------------------------------------------------------------%
% SUB-SECTION: PROOF PROPERTIES F-LAMBDA-S 
%----------------------------------------------------------------------------------------------------------------%

\subsection{Proof of \cref{lemma:str-conv}} \label{sec:app-1}

\begin{proof} Since $r$ is subdifferentiable at $\psub(x)$ with $\Lambda(\usub(x) - \psub(x)) \in \partial r(\psub(x))$ we have $r^\prime(\psub(x); h) \geq \iprod{\Lambda \Fsub(x) - \gfsub(x)}{h}$ for all $x,h \in \Rn$, see, e.g., \cite[Proposition 17.17]{BauCom11}. Now, using the convexity of $f - \frac{\mu_f}{2} \|\cdot\|^2$, the $\mu_r$-strong convexity of $r$, \cref{eq:str-conv-dir} with $f \equiv 0$ and $\mu_f = 0$, and the descent lemma \cref{eq:lip-ineq}, it follows
\begin{align*}
\psi(y) & \geq f(x) + \iprod{\nabla f(x)}{y-x} + \frac{\mu_f}{2} \|y-x\|^2 + r(\psub(x)) \\ & \hspace{3ex} + \iprod{\Lambda \Fsub(x) - \gfsub(x)}{y-\psub(x)} + \frac{\mu_r}{2} \|y - \psub(x)\|^2 \\ & \geq \psi(\psub(x)) + \iprod{\nabla f(x)-\gfsub(x)}{\Fsub(x)} - \frac{L}{2} \|\Fsub(x)\|^2 + \frac{\mu_f}{2} \|y-x\|^2 \\ & \hspace{3ex} + \|\Fsub(x)\|^2_\Lambda  + \iprod{\Lambda \Fsub(x) + \nabla f(x) - \gfsub(x)}{y-x} + \frac{\mu_r}{2} \|y - \psub(x)\|^2 & \\ &= \psi(\psub(x)) + \iprod{\nabla f(x)-\gfsub(x)}{\Fsub(x)} + \half(\mu_r - L) \|\Fsub(x)\|^2  \\ & \hspace{3ex} + \|\Fsub(x)\|^2_\Lambda + \iprod{(\Lambda + \mu_r I) \Fsub(x) + \nabla f(x) - \gfsub(x)}{y-x} + \frac{\bar\mu}{2} \|y - x\|^2
\end{align*}
for all $x,y \in \Rn$. Now, setting $\mathcal E_s(x) := \|\nabla f(x)- \gfsub(x)\|$, $b_2 = (\lamM + \mu_r)^2\bar\mu^{-1}$, and applying Young's inequality twice for some $\tau, \alpha > 0$, we get 
\begin{align*} \left | \iprod{(\Lambda + \mu_r I)\Fsub(x) + \nabla f(x) - \gfsub(x)}{x^* -x} \right| & \\ & \hspace{-40ex} \leq
\half (1+\alpha) b_2 \|\Fsub(x)\|^2 + \frac{(1+\tau(1+\alpha))\bar\mu}{2(1+\tau)(1+\alpha)} \|x-x^*\|^2 + \frac{(1+\tau)(1+\alpha)}{2\tau\alpha\bar\mu} \cdot \mathcal E_s(x)^2, \end{align*}
Setting $y = x^*$ in our first derived inequality, recalling $b_1 = L - 2\lamm - \mu_r$, and using the optimality of $x^*$, $\Lambda \succeq \lamm I$, Young's inequality, and the latter estimate, it holds
\begin{align*}
\frac{\alpha\bar\mu}{2(1+\tau)(1+\alpha)} \|x - x^*\|^2 & \\ & \hspace{-12ex} \leq \half (L - \mu_r + (1+\alpha) b_2 ) \|\Fsub(x)\|^2 - \|\Fsub(x)\|_\Lambda^2  \\ & \hspace{-9ex} - \iprod{\Fsub(x)}{\nabla f(x) - \gfsub(x)} + (1+\tau)(1+\alpha)(2\tau\alpha\bar\mu)^{-1} \cdot \mathcal E_s(x)^2 \\ & \hspace{-12ex} \leq \half ( b_1 + \tau + (1+\alpha) b_2) \|\Fsub(x)\|^2 + \frac{\alpha\bar\mu + (1+\tau)(1+\alpha)}{2\tau\alpha\bar\mu} \cdot \mathcal E_s(x)^2. \end{align*}
Multiplying both sides with $2(1+\tau)(1+\alpha)(\alpha\bar\mu)^{-1}$ and choosing $\alpha := (b_1 + b_2 + \tau)^\half b_2^{-\half}$ (this minimizes the factor in front of $\|\Fsub(x)\|^2$), we finally obtain \cref{eq:str-conv-prop2} with $B_2(\tau) := (\tau\alpha^2\bar\mu^2)^{-1}(1+\tau)(1+\alpha)(\alpha\bar\mu+(1+\tau)(1+\alpha))$. If the full gradient is used, we do not need to apply Young's inequality for $\tau > 0$. Thus, we can set $B_1(\tau) \equiv B_1(0)$.
\end{proof}

%----------------------------------------------------------------------------------------------------------------%
% SUB-SECTION: CONCENTRATION INEQUALITIES
%----------------------------------------------------------------------------------------------------------------%

\subsection{Proof of \cref{lemma:tail-bound} (ii)} \label{sec:app-3}
\begin{proof} 
%As in \cite[Proposition 3.1]{Tro12}, using the conditional Markov inequality, we have
%
%\[ \Prob\left( \ewmax\left({\sum}_k {\sf X}_k\right) \geq t \mid \mathcal U_0\right) \leq e^{-\theta t} \cdot \Exp\left[ \tr \exp \left( {\sum}_k \theta {\sf X}_k \right) \mid \mathcal U_0 \right] \]
%
%almost everywhere and for all $t \in \R$, $\theta > 0$. 
As in \cite[Proposition 4.2]{JudNem08}, we first use the inequality $e^x \leq x + \exp(9x^2 / 16)$, $x \in \R$. Hence, for any matrix $X \in \Sn$ it follows 
\[ e^{X} \preceq  X + \exp(9 X^2 /16) \quad \text{and} \quad e^X \preceq \|e^X\| \cdot I \preceq e^{\|X\|} \cdot I, \]
where ``$\preceq$'' denotes the usual semidefinite order on the space of symmetric matrices. By Jensen's inequality and under the condition $\theta \leq \frac{4}{3\bar\sigma_k}$, this implies 
\begin{align*} 
\Exp\left[ e^{\theta {\sf X}_k} \mid \mathcal U_{k-1} \right] & \preceq \Exp\left[ \theta {\sf X}_k \mid \mathcal U_{k-1} \right] + \Exp\left[ \exp(9\theta^2 {\sf X}_k^2 / 16) \mid \mathcal U_{k-1} \right] \\ &\preceq \Exp\left[ \exp(9\theta^2 \|{\sf X}_k\|^2 / 16) \mid \mathcal U_{k-1} \right] \cdot I \preceq \exp(9\theta^2 \bar\sigma_k^2 /16) \cdot I \end{align*}
with probability 1. On the other hand, by Young's inequality, we have $\theta x \leq \frac{3\theta^2\bar\sigma_k^2}{8} + \frac{2x^2}{3\bar\sigma_k^2}$ for all $\theta, x \in \R$, which easily yields
\[ \Exp\left[ e^{\theta {\sf X}_k} \mid \mathcal U_{k-1} \right] \preceq \Exp[e^{\theta \|{\sf X}_k\|} \mid \mathcal U_{k-1}] \cdot I \preceq \exp\left({3\theta^2\bar\sigma_k^2}/{8} + {2}/{3}\right) \cdot I. \]
As in \cite{JudNem08}, we can now combine the last two estimates and it follows
%Due to $\frac{3\theta^2\nu_k^2}{8} + \frac{2}{3} \leq \frac{3\theta^2\nu_k^2}{4}$ for all $\theta \geq \frac{4}{3\nu_k}$, it finally follows
%
\[ \Exp\left[ e^{\theta {\sf X}_k} \mid \mathcal U_{k-1} \right] \preceq \exp(3\theta^2\bar\sigma_k^2 / 4) \cdot I, \quad \text{a.e.}, \quad \forall~\theta \geq 0. \]
The rest of the proof is similar to \cite[Theorem 7.1]{Tro12}. Specifically, following the arguments and steps in the proof of \cite[Theorem 7.1]{Tro12}, we can show
\begin{align*}
\Exp\left[ \tr \exp \left( \sum_{k=1}^m \theta {\sf X}_k \right) \mid \mathcal U_0 \right] %&= \Exp\left[\Exp\left[ \tr \exp \left( \sum_{k=1}^{m-1} \theta {\sf X}_k + \theta {\sf X}_m\right) \mid \mathcal U_{m-1} \right] \mid \mathcal U_0 \right] 
& \leq \Exp\left[ \tr \exp \left( \sum_{k=1}^{m-1} \theta {\sf X}_k + \log \Exp \left[e^{\theta {\sf X}_m} \mid \mathcal U_{m-1} \right]\right) \mid \mathcal U_0 \right] \\ & \leq \Exp\left[ \tr \exp \left(\sum_{k=1}^{m-1} \theta {\sf X}_k + \frac{3\theta^2\bar\sigma_m^2}{4} I \right) \mid \mathcal U_0 \right] \\ & \leq \ldots \leq \exp(3\theta^2\|\bar\sigma\|^2/4) \cdot n. 
\end{align*}
This estimate can be used in the (conditional) Laplace transform bound \cite[Proposition 3.1]{Tro12}. Optimizing with respect to $\theta$ and setting $\theta = \frac{2\tau}{3\|\bar\sigma\|^2}$, we then get
\[ \Prob\left( \ewmax\left({\sum}_k {\sf X}_k\right) \geq \tau \mid \mathcal U_0\right) \leq e^{-\theta \tau} \cdot \Exp\left[ \tr \exp \left( {\sum}_k \theta {\sf X}_k \right) \mid \mathcal U_0 \right] \leq n \cdot \exp\left({\textstyle -\frac{\tau^2}{3\|\bar\sigma\|^2}}\right). \]
Finally, by rescaling $\tau$ and applying this result to the \textit{dilations} of the matrices ${\sf X}_k$, \cite[Section 2.4]{Tro12}, we obtain \cref{eq:mat-bound-light}.
%We set $\cG = (g_1,...,g_{{\bf n}^g})$, $g_i \in \cN$. Since the sample set $\cG$ is chosen with replacement, the random variables $g_i$ are mutually independent. Let us now consider the sample set $\tilde \cG := (g_1,...,g_{i-1},\tilde g_i,g_{i+1},...,g_{{\bf n}^g})$ with some different $\tilde g_i \in \cN$. Then, it holds
%
%\begin{align*} |\|\nabla f(x) - \nabla f_{\cG}(x)\| - \|\nabla f(x) - \nabla f_{\tilde \cG}(x)\| | & \\ & \hspace{-25ex} \leq
% \|\nabla f_\cG(x) - \nabla f_{\tilde \cG}(x)\| = \frac{1}{{\bf n}^g} \|\nabla f_{g_i}(x) - \nabla f_{\tilde g_i}(x)\| \leq \frac{1}{{\bf n}^g} \Delta(x) \end{align*}
%
%for all $i = 1,...,{\bf n}^g$ and for all possible realizations of $\cG$ and $\tilde\cG$. Consequently, the function $F(g_1,...,g_{{\bf n}^g}) := \|\nabla f(x) - \nabla f_\cG(x)\|$ satisfies a bounded difference condition and we can apply a martingale inequality to derive the bound \cref{eq:conc-grad}. Specifically, using \cref{eq:samp-var}, Jensen's inequality, and the McDiarmid's inequality \cite{McD89}, we obtain
%
%\[ {\mathds P}\left(\|\nabla f(x) - \nabla f_\cG(x)\| \geq \sqrt{\frac{N-1}{N{\bf n}^g}} S(x) + t \right) \leq \exp\left( - \frac{2{\bf n}^g t^2}{\Delta(x)^2} \right). \]
%Setting $t = \sqrt{(2{\bf n}^g)^{-1}\log(\delta^{-1})} \Delta(x)$ gives the result.
\end{proof}

%----------------------------------------------------------------------------------------------------------------%
% ACKNOWLEDGEMENTS
%----------------------------------------------------------------------------------------------------------------%

%\section*{Acknowledgements} The authors would like to thank Michael Ulbrich for his helpful and valuable comments on an earlier version of this paper.

%----------------------------------------------------------------------------------------------------------------%
% REFERENCES
%----------------------------------------------------------------------------------------------------------------%

\bibliographystyle{siam}
\bibliography{S4N.bib}

\end{document}